\documentclass[numbook,envcountsame]{svjour4} 
  \usepackage[all,2cell,arrow]{xy}
   \usepackage{amsmath}
    \usepackage{amssymb}
 \usepackage[left=2.5cm,right=2.5cm,top=3cm,bottom=2.5cm]{geometry}
\usepackage{mathtools}
\usepackage{graphicx}
\usepackage{caption}
 \usepackage{xcolor}
\definecolor{darkgreen}{rgb}{0,0.45,0} 
\usepackage[pagebackref,colorlinks,citecolor=darkgreen,linkcolor=darkgreen]{hyperref}
\usepackage{enumerate,xspace}
\usepackage{lscape}
\usepackage{pdflscape}
\usepackage{appendix}
\usepackage{float}
\UseAllTwocells
\CompileMatrices
\smartqed
 \def\ro{\color{black}}
\def\rouge{\color{black}}
\def\red{\color{black}}
\def\blue{\color{black}}
\def\green{\color{black}}
\def\orange{\color{black}}
\newcommand{\act}{\triangleright}
\newcommand{\col}{\colon}
\newcommand{\ox}{\otimes}
\newcommand{\ospace}{ }
\newcommand{\x}{\times}
\renewcommand{\phi}{\varphi}
\newcommand{\C}{\mathcal{C}}
    \newtheorem{Convention}[theorem]{Convention}

\usepackage{tikz}
\usepackage{tikz-cd}

\begin{document}
\renewcommand\thefigure{\arabic{figure}}    
\setcounter{figure}{0}   

\title{Split extensions and actions of bialgebras and Hopf algebras}

\author{Florence Sterck}

\institute{Institut de Recherche en Math\'ematique et Physique, Universit\'e catholique de Louvain, Chemin du Cyclotron 2, 1348 
Louvain-la-Neuve \and D\'epartement de Math\'ematique, Universit\'e Libre de Bruxelles, Campus de la Plaine – CP 210
Boulevard du Triomphe
1050 Bruxelles, Belgium \\
\email{florence.sterck@uclouvain.be}}

\date{}
%\date{Received: date / Accepted: date}

\maketitle

\begin{abstract}
We introduce a notion of split extension of (non-associative) bialgebras which generalizes the notion of split extension of magmas introduced by M. Gran, G. Janelidze and M. Sobral. We show that this definition is equivalent to the notion of action of (non-associative) bialgebras. We particularize this equivalence to (non-associative) Hopf algebras by defining split extensions of (non-associative) Hopf algebras and proving that they are equivalent to actions of (non-associative) Hopf algebras. Moreover, we prove the validity of the Split Short Five Lemma for these kinds of split extensions, and we examine some examples.
\end{abstract}
\section*{Declaration}

Conflicts of interest/Competing interests : NA

Availability of data and material : NA

Code availability : NA

\section*{Acknowledgements}

The author would like to warmly thank her supervisors Marino Gran and Joost Vercruysse for all the advice in the realization of this paper. Many thanks also to George Janelidze for the suggestion to explore split extensions and actions beyond the vector spaces case, in the context of monoidal categories. This led to a real improvement of this paper. The author also would like to thank Manuela Sobral for the invitation to the University of Coimbra and for the useful discussions. The
author thanks the anonymous referee for his/her useful remarks and suggestions. The author's research is supported by a FRIA (\textit{Fonds pour la formation à la recherche dans l'industrie et dans l'agriculture}) doctoral grant no. 27485 of the \textit{Communaut\'e fran\c{c}aise de Belgique}.

\section*{Introduction}
In the category of groups, split extensions have a lot of interesting properties. A split extension of groups is a diagram of the form
\begin{equation}\label{group extension}
% textidote: ignore begin
\begin{tikzpicture}[descr/.style={fill=white},baseline=(A.base)] 
\node (A) at (0,0) {$G$};
 \node (B) at (2.5,0) {$H$};
\node (C) at (-2.5,0) {$K$};
\node (D) at (4.5,0) {$0$};
\node (E) at (-4.5,0) {$0$};
\path[->,font=\scriptsize]
(C.east) edge node[below] {$k$} (A.west)
(E.east) edge node[below] {$ $} (C.west)
(B.east) edge node[below] {$ $} (D.west)
([yshift=-4pt]A.east) edge node[below] {$f$} ([yshift=-4pt]B.west)
([yshift=2pt]B.west) edge node[above] {$s$} ([yshift=2pt]A.east);
\end{tikzpicture} 
% textidote: ignore end 
\end{equation} where $f \cdot s = 1_H$ and $k$ is the kernel of $f$. {\blue Then $f$ is the cokernel
of $k$, so that $(k, f)$ is a short exact sequence, as the diagram indicates.} One of the interesting properties of split extensions of groups is the fact that the category of split extensions is equivalent to the category of group actions. An action of a group {\blue $(G,1)$} on a group {\blue $(X,1)$} is a map $\rho \colon G \x X \rightarrow X \colon (g,x) \rightarrow \;^gx$ such that {\red for any $g,g' \in G$ and $x,x' \in X$ the following identities hold{\blue :}}
\begin{align}
^{gg'}x &= \; ^g(^{g'}x), \label{ass1 group}\\
^{1}x &= x,\\
^g(xx') &= \;^gx^{g}x'\label{ass2 group}.
\end{align}    

In any semi-abelian category \cite{JMT} there is an equivalence between (internal) actions and split extensions \cite{BJK}, since there is a natural categorical notion of semidirect product introduced by D. Bourn and G. Janelidze in \cite{BJ}. 
%The aim of the definition of semi-abelian category was to capture the properties of the categories of groups, Lie algrebas ,non-unital rings,... 
%A category $\mathcal{C}$ is called semi abelian when $\mathcal{C}$ is pointed, exact (in the sense of Barr), protomodular \cite{Bourn} (this condition is equivalent to the Split Short Five Lemma in $\mathcal{C}$, since the category is pointed) and $\mathcal{C}$ has binary coproducts. 
Unfortunately, one can not expect this correspondence to hold in general algebraic categories: for example, in the category of monoids, split extensions and monoid actions do not form equivalent categories. Nevertheless, D. Bourn, N. Martins-Ferreira, A. Montoli and M. Sobral proved in \cite{BMS} that, in this case of monoids, there is still a restricted equivalence for the so-called ``Schreier {\blue split epimorphisms}". The terminology ``Schreier {\blue split epimorphisms}" came from a paper of A. Patchkoria \cite{Patchkoria}, who worked on a notion of a Schreier internal category in the category of monoids and proved that the category of Schreier
internal categories in the category of monoids is equivalent to the
category of crossed semimodules. A further step of generalization, of importance for this work, was considered in a recent paper by M. Gran, G. Janelidze and M. Sobral, where 
a natural notion of split extension of \emph{unitary magmas} was introduced and shown to correspond to suitable actions \cite{GJS}.
 
%It is well know that a group action $\rho : G \x X \rightarrow X : (g,x) \rightarrow \;^gx$ which is a map such that
%\begin{align*}
%&^{gg'}x =  ^g(^{g'}x)\\
%&^{e}x = x\\
%&^g(xx') = ^gx^{g}x'
%\end{align*}   
%is equivalent to a split extension. 
By a result of \cite{GSV} (see also \cite{GKV}) saying that the category of cocommutative Hopf $K$-algebras is semi-abelian, where $K$ is a field, it is known that there is an equivalence between the actions of cocommutative Hopf algebras and the split extensions of cocommutative Hopf algebras (see \cite{Molnar}, \cite{GSV}, \cite{B}, \cite{VW}). Moreover, when we consider cocommutative (non-)associative bialgebras over a symmetric monoidal category $\mathcal{C}$, they can be seen as internal monoids (or magmas) in the category of cocommutative coalgebras over $\mathcal{C}$, and then we can apply the results of \cite{GJS}.
 
 But what happens in the non-cocommutative case? 
% In that case, the categorical product of coalgebras is not the monoidal product 
 This paper will give an answer to this question. We define different split extensions that are equivalent to the actions of non-associative bialgebras, bialgebras, non-associative Hopf algebras and Hopf algebras in any symmetric monoidal category. This general context provides a wide variety of possible applications. {\blue In particular, the definition of split extensions of non-associative bialgebras generalizes the notion of \emph{split extensions of unitary magmas} \cite{GJS}, which was a non-associative generalization of the concept of ``Schreier split epimorphisms'' of monoids \cite{BMS}}.
 
 The first  {\blue section} of this paper is devoted to the preliminaries, where we recall the definition of bialgebras and Hopf algebras in a symmetric monoidal category.
 
 In the second {\blue section}, we define split extensions of non-associative bialgebras and show that they form a category that is equivalent to the category of actions of non-associative bialgebras. {\blue We show that these split extensions have the interesting property of being exact sequences.} We prove a variation of the Split Short Five Lemma in the category of non-associative bialgebras, when one restricts it to the split extensions that we have introduced.
 
  {\blue In the third section, we investigate the particular cases of cocommutative and associative bialgebras in a symmetric monoidal category}. In particular, in the case of cocommutative bialgebras, the equivalence between actions and split extensions gives us the results in Section 4.6 in \cite{GJS}. 
 
 The last {\blue section} {\red describes the case of Hopf algebras and provides some}  examples of split extensions of Hopf algebras. In particular, we investigate the case where a split epimorphism $\alpha \colon A \rightarrow B$ of associative Hopf algebras % textidote: ignore begin
%\begin{tikzpicture}[descr/.style={fill=white},baseline=(A.base)] 
%\node (A) at (0,0) {$G$};
%\node (B) at (2.5,0) {$H$};
%\path[->,font=\scriptsize]
%([yshift=-4pt]A.east) edge node[below] {$\alpha$} ([yshift=-4pt]B.west)
%([yshift=2pt]B.west) edge node[above] {$e$} ([yshift=2pt]A.east);
%\end{tikzpicture} 
% textidote: ignore end 
satisfies the additional condition $HKer(\alpha) = LKer(\alpha)$, which is a condition given by N. Andruskiewitsch in \cite{A} in order to define what he calls an \emph{exact extension} of Hopf algebras.

\section{Preliminaries}
We recall \cite{MacLane} that a \textit{monoidal category} is given by a triple $(\mathcal{C}, \otimes, I)$ where $\mathcal{C}$ is a category, $\otimes \colon \mathcal{C} \times \mathcal{C} \rightarrow \mathcal{C}$ a bifunctor and $I$ is the identity {\blue object} (we {\blue omit} the three natural isomorphisms, the associator, the right unit and the left unit).

A \textit{braided monoidal category} is a 4-tuple  $(\mathcal{C}, \otimes, I, \sigma)$ where  $(\mathcal{C}, \otimes, I)$ is a monoidal category and $\sigma$ is a \textit{braiding}. A braiding consists of a natural isomorphism {\ro $\sigma = (\sigma_{X,Y} \colon X \otimes Y \to Y \ox X)_{X,Y \in \mathcal{C}}$} such that for any objects $X$, $Y$ and $Z$ in $\mathcal{C}$ the following equations are satisfied \[ \sigma_{X \ox Y, Z} = (\sigma_{X,Z} \ox 1_Y)\cdot (1_X \ox \sigma_{Y,Z} ),\]
 \[ \sigma_{X , Y \ox Z} = (1_Y \ox \sigma_{X,Z} )\cdot (\sigma_{X,Y} \ox 1_Z).\]
 A braided monoidal category is called \textit{symmetric} when \[\sigma_{Y,X}^{-1} = \sigma_{X,Y}.\]
%In this paper, we omit {\blue the indexes} of the braiding when {\blue this does not cause confusion}.

An algebra in a symmetric monoidal category $(\mathcal{C}, \otimes, I, \sigma)$ is given by an object $A \in \mathcal{C}$ endowed with a morphism $m \colon A \ox A \rightarrow A$, called the multiplication. {\blue An algebra is unital when there is a morphism $u_A \colon I \rightarrow A$ called the unit, such that the following equalities are satisfied}
\begin{equation}\label{unital multiplication}
m \cdot (u_A \ox 1_A) = 1_A = m \cdot (1_A \ox u_A),
\end{equation}
\begin{center}
% textidote: ignore begin
\begin{tikzpicture}[descr/.style={fill=white},baseline=(A.base)] 
\node (A) at (0,0) {$ A \ox A$};
\node (B) at (2,0) {$A$};
\node (C) at (-2,0) {$A $};
\node (D) at (0,-1.5) {$ A$};
\draw[commutative diagrams/.cd, ,font=\scriptsize]
(B.south west) edge[commutative diagrams/equal] (D.north east)
(C.south east) edge[commutative diagrams/equal] (D.north west);
\path[->,font=\scriptsize]
(A.south) edge node[descr] {$m$} (D.north)
(C.east) edge node[above] {$ u_A \ox 1_A$} (A.west)
(B.west) edge node[above] {$1_A \ox u_A$} (A.east)
;
\end{tikzpicture}
% textidote: ignore end
\end{center}
{\blue All the algebras that we will consider in this paper are unital. However, }we do not require any associativity condition on algebras.
A morphism of algebras $f \colon A \rightarrow B$ is a morphism in $\mathcal{C}$ such that the following two diagrams commute
%\begin{equation*}
%m \cdot (f \ox f) = f \cdot m
%\end{equation*}
%\begin{equation*}
%f \cdot u = u
%\end{equation*}
\begin{center}
% textidote: ignore begin
\begin{tikzpicture}[descr/.style={fill=white},baseline=(D.base)] 
\node (A) at (0,0) {$ A $};
\node (B) at (3,0) {$B$};
\node (C) at (3,1.5) {$B \ox B $};
\node (D) at (0,1.5) {$ A \ox A$};
\path[->,font=\scriptsize]
(C.south) edge node[descr] {$m$} (B.north)
(D.south) edge node[descr] {$m$} (A.north)
(A.east) edge node[above] {$ f$} (B.west)
(D.east) edge node[above] {$ f \ox f$} (C.west)
;
\end{tikzpicture} 
\qquad
\begin{tikzpicture}[descr/.style={fill=white},baseline=(A.base)] 
\node (A) at (0,0) {$ A $};
\node (C) at (-2,0) {$I $};
\node (D) at (0,-1.5) {$ B.$};
\path[->,font=\scriptsize]
(A.south) edge node[descr] {$f $} (D.north)
(C.south east) edge node[descr] {$u_B $} (D.north west)
(C.east) edge node[above] {$ u_A $} (A.west)
;
\end{tikzpicture}
% textidote: ignore end
\end{center}
A coalgebra is the dual notion of the notion of an algebra. In other words, a coalgebra over $(\mathcal{C}, \otimes, I, \sigma)$ is an object $C \in \mathcal{C}$ with a comultiplication $\Delta \colon C \rightarrow C \ox C$. {\blue From now on, the coalgebras will always be coassociative, i.e. the following equality holds
\begin{equation}\label{coass comultiplication}
(\Delta \ox 1_C) \cdot \Delta = (1_C \ox \Delta) \cdot \Delta
\end{equation}
\begin{center}
\begin{tikzpicture}[descr/.style={fill=white},baseline=(D.base)] 
\node (A) at (0,0) {$ C \ox C $};
\node (B) at (3,0) {$C \ox C \ox C.$};
\node (C) at (3,1.5) {$C \ox C $};
\node (D) at (0,1.5) {$ C$};
\path[->,font=\scriptsize]
(C.south) edge node[descr] {$1_C \ox \Delta$} (B.north)
(D.south) edge node[descr] {$\Delta$} (A.north)
(A.east) edge node[above] {$ \Delta \ox 1_C$} (B.west)
(D.east) edge node[above] {$ \Delta$} (C.west)
;
\end{tikzpicture}
% textidote: ignore end
\end{center}
We will also assume that the coalgebras are counital, meaning that there exists a morphism $ {\blue\epsilon_C} \colon C \rightarrow I$, called counit, satisfying the condition:
\begin{equation}\label{counit comultiplication}
( {\blue\epsilon_C} \ox 1_C) \cdot \Delta = 1_C = (1_C  \ox  {\blue\epsilon_C}) \cdot \Delta,
\end{equation}
as expressed  by the commutativity of the following diagram
% textidote: ignore begin
\begin{center}
\begin{tikzpicture}[descr/.style={fill=white},baseline=(A.base)] 
\node (A) at (0,0) {$ C \ox C$};
\node (B) at (2,0) {$C$};
\node (C) at (-2,0) {$C $};
\node (D) at (0,-1.5) {$ C$};
\draw[commutative diagrams/.cd, ,font=\scriptsize]
(B.south west) edge[commutative diagrams/equal] (D.north east)
(C.south east) edge[commutative diagrams/equal] (D.north west);
\path[->,font=\scriptsize]
(D.north) edge node[descr] {$\Delta$} (A.south)
(A.west) edge node[above] {$  {\blue\epsilon_C} \ox 1_C$} (C.east)
(A.east) edge node[above] {$1_C \ox {\blue\epsilon_C}$} (B.west)
;
\end{tikzpicture} 
\end{center}
}
Similarly, a morphism of coalgebras $g \colon C \rightarrow D$ is  a morphism in $\mathcal{C}$ such that the following two diagrams commute
%\begin{equation*}
%\Delta \cdot g  = (g \ox g) \cdot \Delta
%\end{equation*}
%\begin{equation*}
%\epsilon \cdot g = \epsilon
%\end{equation*}
\begin{center}
% textidote: ignore begin
\begin{tikzpicture}[descr/.style={fill=white},baseline=(D.base)] 
\node (A) at (0,0) {$ D $};
\node (B) at (3,0) {$C$};
\node (C) at (3,1.5) {$C \ox C $};
\node (D) at (0,1.5) {$ D \ox D$};
\path[->,font=\scriptsize]
(B.north) edge node[descr] {$\Delta$} (C.south) 
 (A.north) edge node[descr] {$\Delta$} (D.south)
 (B.west) edge node[above] {$ g$} (A.east)
(C.west) edge node[above] {$ g\ox g$} (D.east)
;
\end{tikzpicture} 
\qquad
\begin{tikzpicture}[descr/.style={fill=white},baseline=(A.base)] 
\node (A) at (0,0) {$ D $};
\node (C) at (-2,0) {$I $};
\node (D) at (0,-1.5) {$ C.$};
\path[->,font=\scriptsize]
(D.north)edge node[descr] {$g $} (A.south) 
(D.north west) edge node[descr] {$ {\blue\epsilon_C} $} (C.south east)
(A.west) edge node[above] {$  {\blue\epsilon_D} $} (C.east)
;
\end{tikzpicture} 
% textidote: ignore end
\end{center}

We also recall that a bialgebra is a 5-tuple $(B,m,{\blue u_B},\Delta,{\blue \epsilon_B})$ where $(B,m,{\blue u_B})$ is an algebra, $(B,\Delta, {\blue \epsilon_B})$ is a coalgebra and $\Delta, {\blue \epsilon_B}$ are algebra morphisms (which is equivalent to ask{\ro ing} that $m$ and ${\blue u_B}$ are coalgebra morphisms) i.e.\ the following conditions hold
% textidote: ignore begin
\begin{equation}\label{m et delta}
 \Delta \cdot m = (m \ox m) \cdot (1_B \ox {\blue \sigma_{B,B}} \ox 1_B) \cdot (\Delta \ox \Delta) ,
 \end{equation} 
% textidote: ignore end
 \begin{center}
% textidote: ignore begin
\begin{tikzpicture}[descr/.style={fill=white},baseline=(D.base),xscale=1.5] 

\node (A) at (0,0) {$ B \ox B \ox B \ox B $};
\node (B) at (3.2,0) {$ B \ox B \ox B \ox B$};
\node (E) at (6,0) {$ B \ox B $};
\node (C) at (6,1.5) {$B $};
\node (D) at (0,1.5) {$ B \ox B$};
\path[->,font=\scriptsize]
(C.south)  edge node[descr] {$\Delta$} (E.north)
  (D.south) edge node[descr] {$\Delta \ox \Delta$}(A.north)
 (A.east)edge node[above] {$ 1_B \ox {\blue \sigma_{B,B}} \ox 1_B$} (B.west) 
 (B.east)edge node[above] {$ m \ox m$} (E.west) 
(D.east) edge node[above] {$ m$} (C.west)
;
\end{tikzpicture} 
% textidote: ignore end
\end{center}
% textidote: ignore begin
 \begin{equation}\label{u et delta}
 \Delta \cdot u_B = u_B \ox u_B,
  \end{equation} 
   \begin{center}
  \begin{tikzpicture}[descr/.style={fill=white},baseline=(A.base)] 
\node (A) at (0,0) {$ B $};
\node (C) at (-2,0) {$I $};
\node (D) at (0,-1.5) {$ B \ox B$};
\path[->,font=\scriptsize]
(A.south) edge node[descr] {$\Delta $} (D.north)
(C.south east) edge node[descr] {$u_B \ox u_B $} (D.north west)
(C.east) edge node[above] {$ u_B $} (A.west)
;
\end{tikzpicture} 
\end{center}
 \begin{equation}\label{m et epsilon}
 {\blue \epsilon_B} \cdot m = {\blue \epsilon_B} \ox {\blue \epsilon_B},
  \end{equation} 
     \begin{center}
  \begin{tikzpicture}[descr/.style={fill=white},baseline=(A.base)] 
\node (A) at (0,0) {$ B  $};
\node (C) at (-2,0) {$B \ox B$};
\node (D) at (0,-1.5) {$ I$};
\path[->,font=\scriptsize]
(A.south) edge node[descr] {${\blue \epsilon_B} $} (D.north)
(C.south east) edge node[descr] {${\blue \epsilon_B} \ox {\blue \epsilon_B} $} (D.north west)
(C.east) edge node[above] {$ m $} (A.west)
;
\end{tikzpicture} 
\end{center}
 \begin{equation}\label{u et epsilon}
 {\blue \epsilon_B} \cdot u_B = 1_I.
\end{equation}
   \begin{center}
  \begin{tikzpicture}[descr/.style={fill=white},baseline=(A.base)] 
\node (A) at (0,0) {$ B $};
\node (C) at (-2,0) {$I $};
\node (D) at (0,-1.5) {$ I$};
\draw[commutative diagrams/.cd, ,font=\scriptsize]
(C.south east) edge[commutative diagrams/equal] (D.north west);
\path[->,font=\scriptsize]
(A.south) edge node[descr] {${\blue \epsilon_B} $} (D.north)
(C.east) edge node[above] {$ u_B $} (A.west)
;
\end{tikzpicture} 
% textidote: ignore end
\end{center}
Moreover, a morphism in $\mathcal{C}$ is a morphism of bialgebras if it is a morphism of algebras and coalgebras.

A non-associative Hopf algebra is a 7-tuple $(A,m,{\blue u_A},\Delta,{\blue \epsilon_A},{\rouge S_L, S_R})$ where $(A,m,u_A,\Delta,{\blue \epsilon_A})$ is a bialgebra and {\rouge $S_L$ and $S_R$ are antihomomorphisms of coalgebras and algebras, called the left and the right antipode, such that the following diagram commutes 
% textidote: ignore begin
\begin{equation}\label{antipode}
 \begin{tikzpicture}[descr/.style={fill=white},baseline=(current  bounding  box.center),xscale=0.7] 
\node (A) at (-4,0) {$ A  $};
\node (C) at (2,0) {$I$};
\node (D) at (8,0) {$ A.$};
\node (E) at (-1.5,1.5) {$ A \ox A$ };
\node (F) at (5.5,1.5) {$ A \ox A$ };
\path[->,font=\scriptsize]
(A.east) edge node[below] {${\blue \epsilon_A }$} (C.west)
(C.east) edge node[below] {${\blue u_A }$} (D.west)
(A.north east) edge node[left,xshift=-5pt] {$ \Delta $} (E.south west)
(F.south east) edge node[right,xshift=5pt] {$ m $} (D.north west)
([yshift=-4pt]E.east) edge node[below] {$S_L \ox 1_A $} ([yshift=-4pt]F.west)
([yshift=4pt]E.east) edge node[above] {$1_A \ox S_R $} ([yshift=4pt]F.west);
\end{tikzpicture} 
\end{equation}
{\green 
%The morphism $S$ is an antihomomorphism of coalgebras and algebras.
A morphism of Hopf algebras is a morphism of bialgebras preserving the antipodes. Note that in the case of associative Hopf algebras, the antipode is unique  {\rouge ($S_L = S = S_R)$} and then and $S$ is automatically an antihomomorphism of coalgebras and algebras. Moreover, a bialgebra morphism between associative Hopf algebras necessarily preserves the antipode.}
% textidote: ignore end
%The morphisms of Hopf algebras are morphisms of bialgebras, since the antipode is automatically preserved. 

\begin{example}
(1) In the symmetric monoidal category $(\mathsf{Set}, \times , \{\star\})$ of sets where $\sigma$ is the twist morphism (where $\sigma(x,y) = (y,x)$ for any element $x$ of a set $X$ and any element $y$ of a set  $Y$), every object has a coalgebra structure with $\Delta$ being the diagonal and ${\blue \epsilon_X}$ the morphism sending every element of $X$ to {\ro $\star$}. Hence, a non-associative bialgebra (or algebra) is an unital magma, an associative bialgebra (or algebra) is a monoid, an associative Hopf algebra is a group. {\blue The case of non-associative Hopf algebras in the category of sets will be treated in detail in Example \ref{exampleloop}.}

(2) In the symmetric monoidal category $({\blue {\sf Vect}_K},\ox,K)$ of vector spaces over a field $K$ where $\sigma$ is the twist morphism (defined by $\sigma(x \ox y) = y \ox x$ for any $x \ox y \in X \ox Y$), we recover the notion of $K$-algebra, $K$-coalgebra, $K$-bialgebra and Hopf $K$-algebra.

(3) In \cite{CG}, a symmetric monoidal category was introduced such that Hom-algebras, Hom-coalgebras and Hom-Hopf algebras (see \cite{MS}) coincide with the algebras, coalgebras and Hopf algebras in this symmetric monoidal category. 

(4) In \cite{CL}, the authors showed that Turaev's Hopf group-coalgebras (see \cite{Tuarev}) are Hopf algebras in a symmetric monoidal category which they called Turaev category.  

(5) Associative and non-coassociative bialgebras and Hopf algebras in any symmetric monoidal category $\mathcal{C}$ can be seen as non-associative bialgebras and Hopf algebras in $\mathcal{C}^{op}$, the opposite category, which is still a symmetric monoidal category. 

(6) The coquasi-bialgebras and quasi-bialgebras are respectively examples of non-associative bialgebras in ${\blue {\sf Vect}_K}$ and in ${\blue {\sf Vect}_K^{op}}$, see \cite{Majid3} for an introduction about these structures. The coquasi-Hopf algebras have different antipode conditions, but under some specific assumptions, it is possible to see them as non-associative Hopf algebras.  An example which is both a non-associative Hopf algebra, as we defined, and a coquasi-Hopf algebra is the structure of octonions see \cite{AM}.

 %In partiucalr coquasi-bialgebras generalize bialgebras, preserving the monoidality of the category of comodules.
\end{example}

These examples give us a glimpse of {\blue some} frameworks and cases in which the results of this paper can be applied. 

{\blue We also recall an important observation. The identity object $I$ in a symmetric monoidal category $\C$, is a Hopf algebra in $\C$. The comultiplication on $I$ is given by the natural isomorphism $\Delta \colon I \cong I\ox I$, the counit, the unit and the antipode are the identity maps. Moreover, for any Hopf algebra $A$, we have unique Hopf algebra morphisms $u_A \colon I \to A$ and $\epsilon_A : A \to I$. Hence, $I$ is the {\ro zero} object in the category of Hopf algebras in $\C$. Obviously, $I$ is also the zero object in the category
of bialgebras in $\C$.}

\begin{Convention} For the monoidal product of $n$ copies $A \ox \cdots \ox A$, the notation $A^n$ will be used. The same convention will be used for the morphisms, for example, we denote $\alpha \ox \alpha \colon A \ox A \rightarrow B \ox B$ by $\alpha^2 \colon A^2 \rightarrow B^2$. For the sake of simplicity, ``bialgebras'' will mean ``non-associative bialgebras'' (unless the associativity is explicitly mentioned). If it is not explicitly mentioned, the bialgebras and Hopf algebras are considered to be constructed in a general symmetric monoidal category $\C$.
\end{Convention}

\section{Split extensions of non-associative bialgebras}

In this section, we introduce a notion of split extension of bialgebras in a symmetric monoidal category $\C$. These extensions have several properties. Among them, the fact that they are equivalent to actions of bialgebras, a restricted version of the Split Short {\ro F}ive Lemma holds for these extensions. Moreover, a split extension of bialgebras is an exact sequence and the definition of split extension of bialgebras generalizes the one of split extension of magmas introduced in \cite{GJS}.

\begin{definition}\label{def action}
Let $X$ and $B$ be bialgebras in a symmetric monoidal category $(\mathcal{C}, \otimes, I, \sigma)$. An action {\blue of $B$ on $X$} is a morphism $\act \colon B \otimes X \rightarrow X $ {\blue in $\mathcal{C}$}, such that {\blue the diagrams}
% textidote: ignore begin
\begin{equation}\label{conditionaction1}
 \triangleright \cdot (u_B \otimes 1_X)  = 1_X,
\end{equation} 
   \begin{center}
  \begin{tikzpicture}[descr/.style={fill=white},baseline=(A.base)] 
\node (A) at (0,0) {$ B \ox X$};
\node (C) at (-2,0) {$X $};
\node (D) at (0,-1.5) {$ X$};
\draw[commutative diagrams/.cd, ,font=\scriptsize]
(C.south east) edge[commutative diagrams/equal] (D.north west);
\path[->,font=\scriptsize]
(A.south) edge node[descr] {$\act $} (D.north)
(C.east) edge node[above] {$ u_B \ox 1_X$} (A.west)
;
\end{tikzpicture} 
\end{center}
\begin{equation}\label{condition action = e}
 \triangleright \cdot (1_B \otimes u_X) =  u_X \cdot\epsilon_B,
\end{equation} 
  \begin{center}
  \begin{tikzpicture}[descr/.style={fill=white},baseline=(A.base)] 
\node (A) at (0,0) {$ B \ox X $};
\node (C) at (-2,0) {$B  $};
\node (D) at (0,-1.5) {$ X$};
\path[->,font=\scriptsize]
(A.south) edge node[descr] {$\act $} (D.north)
(C.south) edge node[descr] {$ u_X \cdot \epsilon_B $} (D.north west)
(C.east) edge node[above] {$ 1_B \ox u_X $} (A.west)
;
\end{tikzpicture} 
\end{center}
\begin{equation}\label{coco}
(1_B \otimes \triangleright) \cdot (\Delta \otimes 1_X) =  (1_B  \otimes \triangleright) \cdot ({\blue \sigma_{B,B}} \otimes 1_X) \cdot (\Delta \otimes 1_X),
\end{equation} 
 \begin{center}
\begin{tikzpicture}[descr/.style={fill=white},baseline=(D.base),xscale=1.5] 

\node (A) at (0,0) {$ B \ox B  \ox X $};
\node (B) at (3.2,0) {$ B \ox  B \ox X$};
\node (E) at (6,0) {$ B\ox X $};
\node (C) at (6,1.5) {$B \ox B \ox X $};
\node (D) at (0,1.5) {$ B \ox X$};
\path[->,font=\scriptsize]
(C.south)  edge node[descr] {$1_B \ox \act$} (E.north)
  (D.south) edge node[descr] {$\Delta \ox 1_X$}(A.north)
 (A.east)edge node[above] {$  {\blue \sigma_{B,B}} \ox 1_X$} (B.west) 
 (B.east)edge node[above] {$ 1_B
 \ox \act$} (E.west) 
(D.east) edge node[above] {$\Delta \ox 1_X$} (C.west)
;
\end{tikzpicture} 
\end{center}
\begin{equation}\label{eps et act}
\epsilon_X \cdot \triangleright = \epsilon_B \otimes \epsilon_X,
\end{equation} 
   \begin{center}
  \begin{tikzpicture}[descr/.style={fill=white},baseline=(A.base)] 
\node (A) at (0,0) {$ X $};
\node (C) at (-2,0) {$B \ox X $};
\node (D) at (0,-1.5) {$ I$};
\path[->,font=\scriptsize]
(A.south) edge node[descr] {$\epsilon_X $} (D.north)
(C.south) edge node[descr] {$ \epsilon_B \ox \epsilon_X $} (D.north west)
(C.east) edge node[above] {$ \act $} (A.west)
;
\end{tikzpicture} 
\end{center}
\begin{equation}\label{delta et act}
\Delta \cdot  \triangleright = (\triangleright \otimes \triangleright) \cdot (1_B \otimes {\blue \sigma_{B,X}} \otimes 1_X) \cdot (\Delta \otimes \Delta).
\end{equation} 
 \begin{center}
\begin{tikzpicture}[descr/.style={fill=white},baseline=(D.base),xscale=1.5] 

\node (A) at (0,0) {$ B \ox B \ox X \ox X $};
\node (B) at (3.2,0) {$ B \ox X \ox B \ox X$};
\node (E) at (6,0) {$ X \ox X $};
\node (C) at (6,1.5) {$X $};
\node (D) at (0,1.5) {$ B \ox X$};
\path[->,font=\scriptsize]
(C.south)  edge node[descr] {$\Delta$} (E.north)
  (D.south) edge node[descr] {$\Delta \ox \Delta$}(A.north)
 (A.east)edge node[above] {$ 1_B \ox {\blue \sigma_{B,X}} \ox 1_X$} (B.west) 
 (B.east)edge node[above] {$ \act \ox \act$} (E.west) 
(D.east) edge node[above] {$ \act$} (C.west)
;
\end{tikzpicture} 
% textidote: ignore end
\end{center}
{\blue  in $\C$ commute.}
\end{definition}
Let us note that the last two axioms mean that $\act$ is a morphism of coalgebras. The axiom \eqref{coco} is inspired by the condition (1) that Majid used in \cite{Majid} to define a Hopf algebra crossed module. It is also what we need to define a bialgebra via the semi-direct product construction (see Theorem 3.3 in \cite{Majid2} and \cite{Molnar} for the construction of the semi-direct product also called smash-product).  In particular, we are interested in diagrams of the form 
\begin{equation}\label{semi-direct product}
% textidote: ignore begin
\begin{tikzpicture}[descr/.style={fill=white},baseline=(A.base)] 
\node (A) at (0,0) {$X \rtimes B$};
\node (B) at (2.5,0) {$B$};
\node (C) at (-2.5,0) {$X$};
\path[dashed,->,font=\scriptsize]
([yshift=2pt]A.west) edge node[above] {$\pi_1$} ([yshift=2pt]C.east);
\path[->,font=\scriptsize]
([yshift=-4pt]C.east) edge node[below] {$i_1$} ([yshift=-4pt]A.west)
([yshift=-4pt]A.east) edge node[below] {$\pi_2$} ([yshift=-4pt]B.west)
([yshift=2pt]B.west) edge node[above] {$i_2$} ([yshift=2pt]A.east);
\end{tikzpicture}
% textidote: ignore end 
\end{equation}
{\blue in $\C$}, where $i_1 = 1_X \otimes u_B$, $i_2 = u_X \otimes 1_B $, $\pi_1 = 1_X \otimes {\blue \epsilon_B}$, $\pi_2= {\blue \epsilon_X} \otimes 1_B$ and $X \rtimes B$ is the object $X \ox B$, where the bialgebra structure is given by the following morphisms in $\mathcal{C}$,
% textidote: ignore begin
\begin{eqnarray*}
m_{X \rtimes B} &=& (m \otimes m) \cdot (1_X \otimes \triangleright \otimes 1_B \otimes 1_B) \cdot (1_X \otimes 1_B \otimes {\blue \sigma_{B,X}} \otimes 1_B) \cdot (1_X \otimes \Delta \otimes 1_X \otimes 1_B)\\
u_{X \rtimes B} &=&u_X \otimes u_B,\\
\Delta_{X \rtimes B} &=& (1_X \otimes {\blue \sigma_{X,B}} \otimes 1_B) \cdot (\Delta \otimes \Delta),\\
\epsilon_{X \rtimes B}&=& \epsilon_X \otimes \epsilon_B.
\end{eqnarray*}
% textidote: ignore end
By combining Figure \ref{sm bialgebra 1} and Figure \ref{sm bialgebra 2} in the appendix, we check that this definition provides a bialgebra structure on $X \ox B$ making the morphisms $i_1$, $i_2$ and $\pi_2$ bialgebra morphisms, {\rouge and $\pi_1$ a coalgebra morphism}.
Note that we can check that $u_X \otimes u_B$ is the neutral element for the multiplication thanks to the first two axioms \eqref{conditionaction1} and \eqref{condition action = e} of the definition of action. The comultiplication of this structure is given as the usual comultiplication of the product of two coalgebras. In particular, the coassociativity of the comultiplication of this semi-direct product is obvious. It is also interesting to remark that the associativity of this structure is not automatic, even if the bialgebras $X$ and $B$ are associative, see {\ro Example \ref{ex ass bialgebra}.}

Let us make some observations about the graph \eqref{semi-direct product}, which are analogous to the ones made in \cite{GJS}{\ro :}
\begin{lemma}\label{obs semi direct bialg}
The graph \begin{center} 
% textidote: ignore begin
\begin{tikzpicture}[descr/.style={fill=white},baseline=(A.base)] 
\node (A) at (0,0) {$X \rtimes B$};
\node (B) at (2.5,0) {$B$};
\node (C) at (-2.5,0) {$X$};
\path[dashed,->,font=\scriptsize]
([yshift=2pt]A.west) edge node[above] {$\pi_1$} ([yshift=2pt]C.east);
\path[->,font=\scriptsize]
([yshift=-4pt]C.east) edge node[below] {$i_1$} ([yshift=-4pt]A.west)
([yshift=-4pt]A.east) edge node[below] {$\pi_2$} ([yshift=-4pt]B.west)
([yshift=2pt]B.west) edge node[above] {$i_2$} ([yshift=2pt]A.east);
\end{tikzpicture}
% textidote: ignore end 
\end{center} as defined in \eqref{semi-direct product}, where $i_1$, $i_2$, $\pi_2$ are morphisms of bialgebras, satisfies the following properties 

\begin{itemize}
\item[(1)] $\pi_1 \cdot i_1 = 1_X$, $\pi_2 \cdot i_2 =1_B$ 
\item[(2)] $\pi_1 \cdot i_2 = u_X \cdot \epsilon_B$, $\pi_2 \cdot i_1 = u_B \cdot \epsilon_X$
\item[(3)] $m_{X \rtimes B} \cdot (i_1 \cdot \pi_1 \otimes  i_2 \cdot \pi_2) (1_X \otimes  {\blue \sigma_{X,B}} \otimes 1_B) \cdot (\Delta \otimes \Delta)= 1_X \ox 1_B$

\item[(4)] $\pi_1 \cdot m_{X \rtimes B} \cdot (i_1 \otimes i_2) = 1_X \otimes {\blue \epsilon_B} $

\item[(5)]$(1_B \otimes \pi_1) \cdot (1_B \otimes m_{X\rtimes B}) \cdot (1_B \otimes i_2 \otimes i_1)  \cdot (\Delta \otimes 1_X)= (1_B \otimes \pi_1) \cdot (1_B \otimes m_{X\rtimes B}) \cdot (1_B \otimes i_2 \otimes i_1) \cdot ({\blue \sigma_{B,B}} \otimes 1_X) \cdot (\Delta \otimes 1_X)$

\item[(6)] $m_{X \rtimes B} \cdot (1_X \otimes 1_B \otimes m_{X \rtimes B}) \cdot (i_1 \otimes i_2 \otimes 1_X \otimes 1_B) =  m_{X \rtimes B} \cdot (m_{X \rtimes B} \otimes 1_X \otimes 1_B) \cdot (i_1 \otimes i_2 \otimes 1_X \otimes 1_B)$

\item[(7)]
$ m_{X \rtimes B} \cdot (1_X \otimes 1_B \otimes m_{X \rtimes B}) \cdot (i_1 \otimes 1_X \otimes 1_B \otimes i_2) =  m_{X \rtimes B} \cdot ( m_{X \rtimes B} \otimes 1_X \otimes 1_B) \cdot (i_1 \otimes 1_X \otimes 1_B \otimes i_2) $
%\begin{align*}
%i_1(x)\Big((x' \otimes b')i_2(b)\Big) &= (x \otimes 1)\Big((x' \otimes b')(1 \otimes b)\Big) \\
%&= (x \otimes 1)(x' \otimes b'b) \\
%&= xx' \otimes b'b\\
%&= (xx' \otimes  b')(1 \otimes b) \\
%&= \Big((x \otimes 1)(x' \otimes b')\Big)(1 \otimes b) \\
%&=\Big( i_1(x)(x' \otimes b')\Big) i_2(b)
%\end{align*}
\item[(8)]
$ m_{X \rtimes B} \cdot (1_X \otimes 1_B \otimes m_{X \rtimes B}) \cdot (1_X \otimes 1_B \otimes i_1 \otimes i_2) =  m_{X \rtimes B} \cdot ( m_{X \rtimes B} \otimes 1_X \otimes 1_B) \cdot (1_X \otimes 1_B \otimes i_1 \otimes i_2) $
%\begin{align*}
%(x' \otimes b')\Big(i_1(x)i_2(b)\Big) &= (x' \otimes b')\Big((x \otimes 1)(1 \otimes b)\Big) \\
%&= (x' \otimes b')(x \otimes b) \\
%&= x'\;^{b'_1}x \otimes b'_2b\\
%&= (x'\;^{b'_1}x \otimes  b'_2)(1 \otimes b) \\
%&= \Big((x' \otimes b')(x \otimes 1)\Big)(1 \otimes b) \\
%&=\Big( (x' \otimes b')i_1(x)\Big) i_2(b)
%\end{align*}
\item[(9)] $\pi_1$ is a morphism of coalgebras and $\pi_1$ preserves the unit. 

\end{itemize}

\end{lemma}

\begin{proof}
The properties (1) and (2) are trivial. The condition (3) is proven via the commutativity of the following diagram.
\begin{center} 
%\ref{Condition (3)}.
%\begin{figure}[H]  
%\centering 
%\captionsetup{justification=centering}
  % textidote: ignore begin
\begin{tikzpicture}[descr/.style={fill=white}] 
\node (A) at (-4,0) {$X \ox B$};
\node (B) at (3,0) {$X^2 \ox B^2$};
\node (A') at (-4,-1.5) {$X \ox B$};
\node (B') at (3,0-1.5) {$X \ox B$};
\node (C) at (3,1.5) {$X\ox B \ox X \ox B^2$};
\node (D) at (3,3) {$X \ox B^2  \ox X \ox B$};
\node (E) at (3,4.5) {$(X \ox B)^2$};
\node (F) at (3,6) {$X \ox B $};
\node (G) at (3,7.5) {$(X \ox B)^2$};
\node (H) at (-1,7.5) {$X^2 \ox B^2$};
\node (I) at (-4,7.5) {$X \ox B$};
\node (J) at (-4,3) {$X \ox B$};
\node (X) at (-4,6) {$X \ox B$};
\node (Y) at (-4,1.5) {$X \ox B$};
\node  at (-0.5,0.75) {$\eqref{conditionaction1}$};
\node  at (-0.5,4.5) {$\eqref{u et delta}$};
\node at (-0.5,-0.75) {$\eqref{unital multiplication}$};
\node at (-1.5,6.75) {$\eqref{counit comultiplication}$};

%\draw[equal] (J.south) -- (Y.north);
\draw[commutative diagrams/.cd, ,font=\scriptsize]
(I.south) edge[commutative diagrams/equal]  (X.north)
(A.south) edge[commutative diagrams/equal]  (A'.north)
(X.south) edge[commutative diagrams/equal]  (J.north)
(J.south) edge[commutative diagrams/equal] (Y.north)
(Y.south) edge[commutative diagrams/equal]  (A.north)
(A'.east) edge[commutative diagrams/equal]  (B'.west)
(X.east) edge[commutative diagrams/equal]  (F.west);

\path[->,font=\scriptsize]
(H.south east) edge node[descr] {$1_X  \ox {\blue \epsilon_X} \ox {\blue \epsilon_B} \ox 1_B \; \; \; \; \; \; \; \; \; \; \; \; \; \; \; \; \; \; \;$} (F.north west)
(J.east) edge node[above] {$1_X  \ox u_B \ox u_B \ox u_X \ox 1_B$} (D.west)
(Y.east) edge node[above] {$1_X  \ox u_B \ox u_X \ox u_B \ox 1_B$} (C.west)
(A.east) edge node[descr] {$1_X  \ox u_X \ox u_B \ox 1_B$} (B.west)
(H.east) edge node[above] {$1_X  \ox {\blue \sigma_{X,B}} \ox 1_B$} (G.west)
(I.east) edge node[above] {$\Delta \ox \Delta$} (H.west)
(C.south) edge node[descr] {$1_X \ox \act \ox 1_B \ox 1_B$} (B.north)
(B.south) edge node[descr] {$m \ox m$} (B'.north)
(D.south) edge node[descr] {$1_X \ox 1_B \ox {\blue \sigma_{B,X}} \ox 1_B$} (C.north)
(E.south) edge node[descr] {$1_X \ox \Delta \ox 1_X \ox 1_B$} (D.north)
(F.south) edge node[descr] {$1_X \ox u_B \ox u_X \ox 1_B \; \; \; \; \; \; \; \; $} (E.north)
(G.south) edge node[descr] {$1_X \ox {\blue \epsilon_B} \ox {\blue \epsilon_X} \ox 1_B$} (F.north);
\draw[->] (G.south east) to[bend left=55]node[right,scale=0.8] {${\blue (i_1 \cdot \pi_1) \otimes  (i_2 \cdot \pi_2)}$} (E.north east);
\draw[->] (E.south east) to[bend left=55]node[right,scale=0.8] {$ {\blue m_{X \rtimes B}}$} (B'.north east);
\end{tikzpicture} 
% textidote: ignore end
%\caption{Condition $(3)$}
%\label{Condition (3)}
%\end{figure}
\end{center}
The equality (4) holds since this diagram commutes
% textidote: ignore begin
\begin{center}
\begin{tikzpicture}[descr/.style={fill=white},yscale=0.9,baseline=(A.base),scale=0.75] 
\node (A) at (-5,0) {$X$};
\node (B) at (4,0) {$X \ox B.$};
\node (C) at (4,2) {$X^2 \ox B^2$};
\node (D) at (4,4) {$X \ox B \ox X \ox B^2$};
\node (E) at (4,6) {$X \ox B^2 \ox X \ox B$};
\node (F) at (4,8) {$(X \ox B)^2$};
\node (G) at (-5,8) {$X \ox B $};
\node (H) at (-5,2) {$X \ox B$};
\node (X) at (-1,4) {$\eqref{conditionaction1}$};
\node (X) at (-1,7) {$\eqref{u et delta}$};
\node at (2,1.2) {$\eqref{unital multiplication}$};
\draw[->] (F.south east) to[bend left=40]node[right,scale=0.8] {${\blue m_{X \rtimes B}}$} (B.north east);
\draw[commutative diagrams/.cd, ,font=\scriptsize]
(G.south) edge[commutative diagrams/equal]  (H.north)
(H.south east) edge[commutative diagrams/equal]  (B.north west);
\path[->,font=\scriptsize]
(G.south east) edge node[descr] {$1_X  \ox u_B \ox u_X \ox u_B \ox 1_B$} (D.north west)
(B.west) edge node[below] {$1_X  \ox {\blue \epsilon_B}$} (A.east)
(H.east) edge node[above] {$1_X  \ox u_X \ox u_B \ox 1_B$} (C.west)
(G.east) edge node[above] {$1_X  \ox u_B \ox u_X \ox 1_B$} (F.west)
(F.south) edge node[descr] {$1_X \ox \Delta \ox 1_X \ox 1_B$} (E.north)
(E.south) edge node[descr] {$1_X \ox 1_B \ox {\blue \sigma_{B,X}} \ox 1_B$} (D.north)
(D.south) edge node[descr] {$1_X \ox \act \ox (1_B)^2$} (C.north)
(C.south) edge node[descr] {$m  \ox m$} (B.north)
(H.south) edge node[descr] {$1_X \ox {\blue \epsilon_B}$} (A.north);
\end{tikzpicture} 
% textidote: ignore end
\end{center}

The property (5) is due to \eqref{coco} and the commutativity of the squares denoted $(A)$ as we can see in the following commutative diagram
\begin{center}
%\ref{Condition (5)}.
%\begin{figure}
%\centering % textidote: ignore begin
\begin{tikzpicture}[descr/.style={fill=white},yscale=0.8] 
\node (A) at (0,-3) {$ B \ox (X \ox B)^2  $};
\node (B) at (4,-3) {$B \ox X \ox B.$};
\node (C) at (4,0) {$B \ox X$};
\node (D) at (8,0) {$ B \ox X \ox B$};
\node (E) at (8,3) {$ B \ox (X \ox B)^2$};
\node (F) at (4,3) {$B \ox B  \ox X$};
\node (G) at (0,3) {$B \ox X$};
\node (H) at (0,1.5) {$B \ox B \ox X$};
\node (I) at (0,0) {$B \ox B \ox X$};
\node at (2,1.5) {$\eqref{coco}$};
\node at (6,1.5) {$(A)$};
\node at (2,-1.5) {$(A)$};
\path[->,font=\scriptsize]
(E.south) edge node[descr] {$1_B \ox m_{X \rtimes B} $} (D.north)
(D.west) edge node[above] {$1_B \ox 1_X \ox {\blue \epsilon_B}$} (C.east)
(B.north) edge node[descr] {$1_B \ox 1_X \ox {\blue \epsilon_B}$} (C.south)
(G.south) edge node[descr] {$\Delta \ox 1_X$} (H.north)
(H.south) edge node[descr] {${\blue \sigma_{B,B}} \ox 1_X$} (I.north)
(F.east) edge node[above] {$ 1_B \ox i_2 \ox i_1$} (E.west)
(I.south) edge node[descr] {$ 1_B \ox i_2 \ox i_1$} (A.north)
(A.east) edge node[above] {$1_B \ox  m_{X\rtimes B}$} (B.west)
(G.east) edge node[above] {$\Delta  \ox 1_X$} (F.west)
(F.south) edge node[descr] {$1_B \ox \act$} (C.north)
(I.east) edge node[above] {$1_B \ox \act$} (C.west)
;
\end{tikzpicture} 
%\caption{Condition $(5)$}
%\label{Condition (5)}
%\end{figure}
\end{center}
% textidote: ignore end
{\red We observe that} the squares {\blue $(A)$} commute thanks to the unitality of the multiplication and the counitality of the comultiplication, as expressed with commutativity of this diagram  %\ref{diagram (A)}. 
\begin{center}
%\begin{figure}
%\centering
%% textidote: ignore begin
\begin{tikzpicture}[descr/.style={fill=white},yscale=0.6,xscale=0.75] 
\node (A) at (-6,-3) {$B \ox X$};
\node (B) at (10,-3) {$B \ox X \ox B.$};
\node (C) at (10,0) {$B \ox X^2 \ox B^2 $};
\node (D) at (10,3) {$ (B \ox X)^2 \ox B^2$};
\node (E) at (10,6) {$ B \ox X \ox B^2 \ox X \ox B$};
\node (F) at (2,6) {$B \ox (X \ox B)^2 $};
\node (G) at (-6,6) {$B^2 \ox X$};
\node (H) at (-6,3) {$B^2 \ox X$};
\node (I) at (-2,3) {$B^3 \ox X$};
\node (J) at (-2,0) {$B^2 \ox X \ox B $};
\node (K) at (2.5,0) {$B \ox X  \ox B$};
\node at (0,-1.5) {$\eqref{counit comultiplication}$};
\node at (7,-1) {$\eqref{unital multiplication}$};
\draw[commutative diagrams/.cd, ,font=\scriptsize]
(K.south) edge[commutative diagrams/equal]  (B.north west)
(G.south) edge[commutative diagrams/equal]  (H.north);
\path[->,font=\scriptsize]
(E.south) edge node[descr] {$1_B \ox 1_X \ox 1_B \ox {\blue \sigma_{B,X}} \ox 1_B$} (D.north)
(D.south) edge node[descr] {$1_B \ox 1_X \ox \act \ox 1_B \ox 1_B$} (C.north)
(C.south) edge node[descr] {$1_B \ox m \ox m$} (B.north)
(B.west) edge node[above] {$1_B \ox 1_X \ox {\blue \epsilon_B}$} (A.east)
(G.east) edge node[above] {$1_B \ox u_X \ox 1_B \ox 1_X \ox u_B$} (F.west)
(K.east) edge node[above] {$1_B \ox u_X \ox 1_X \ox 1_B \ox u_B$} (C.west)
(F.east) edge node[above] {$1_B \ox 1_X \ox \Delta \ox 1_X \ox 1_B$} (E.west)
(H.east) edge node[above] {$1_B \ox \Delta \ox 1_X $} (I.west)
(I.south) edge node[descr] {$1_B \ox 1_B \ox {\blue \sigma_{B,X}}$} (J.north)
(J.east) edge node[above] {$ 1_B \ox \act \ox 1_B$} (K.west)
(H.south) edge node[descr] {$1_B \ox \act$} (A.north)
;
\end{tikzpicture} % textidote: ignore end
%\caption{Diagram $(A)$}
%\label{diagram (A)}
%\end{figure}
\end{center}

We prove the ``partial associativity" condition (6) thanks to Figure \ref{Condition (6)}.
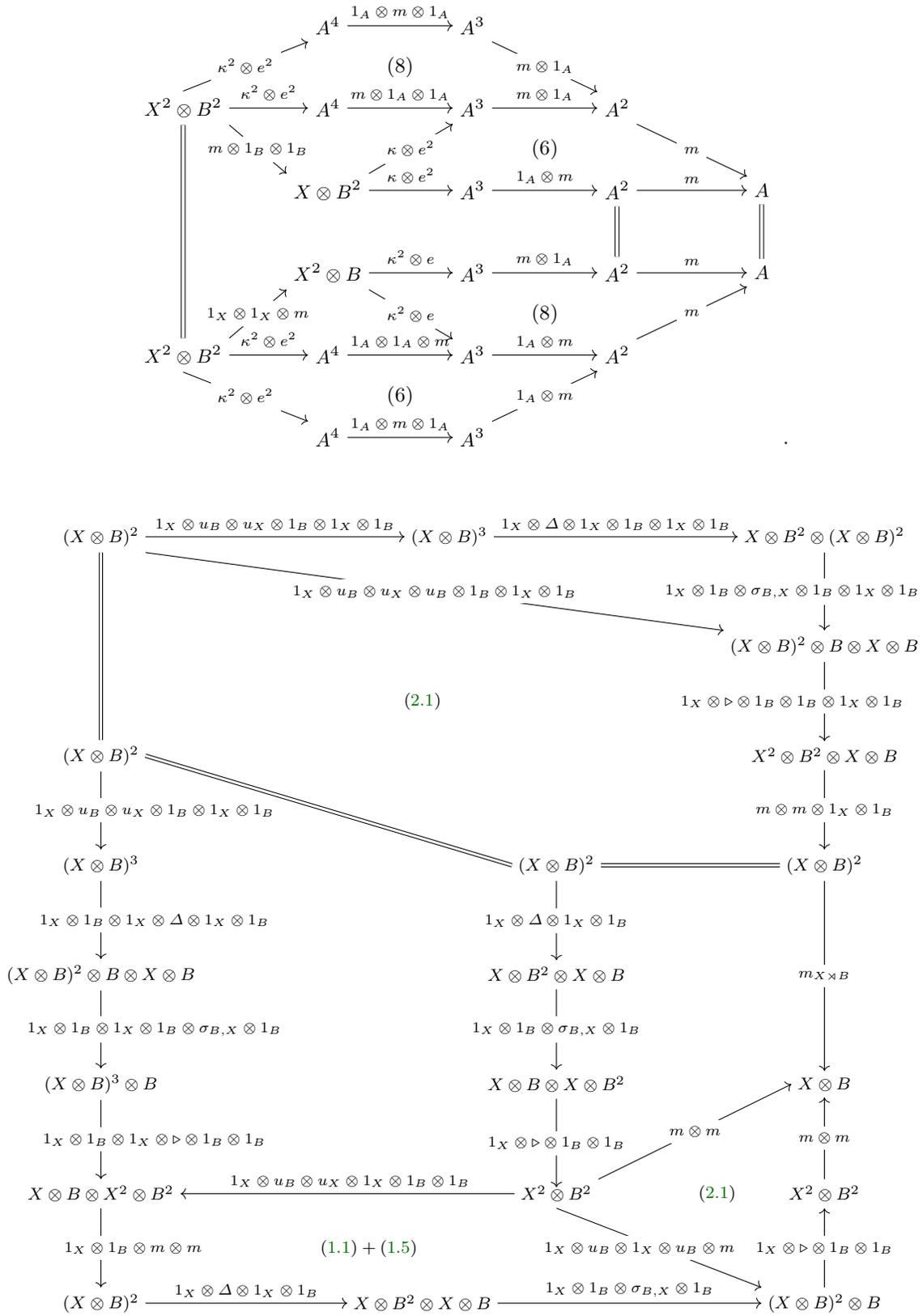
\begin{figure}[b]
\centering
% textidote: ignore begin
\begin{tikzpicture}[descr/.style={fill=white},xscale=0.88,yscale=0.9] 
\node (A1) at (-1,4) {$(X \ox B)^3 \ox B$};
\node (A2) at (-1,2) {$X \ox B \ox X^2 \ox B^2$};
\node (A3) at (-1,0) {$(X \ox B)^2$};
\node (A4) at (5,0) {$X \ox B^2 \ox X \ox B$};
\node (A5) at (12.5,0) {$(X \ox B)^2 \ox B$};
\node (A6) at (12.5,2) {$X^2 \ox B^2$};
\node (B1) at (-1,6) {$(X \ox B)^2 \ox B \ox X \ox B$};
\node (B2) at (7.5,2) {$ $};
\node (B5) at (7.5,2) {$X^2 \ox B^2$};
\node (B6) at (12.5,4) {$X \ox B$};
\node (C2) at (7.5,4) {$X \ox B \ox X \ox B^2$};
\node (D1) at (-1,8) {$(X \ox B)^3$};
\node (D2) at (7.5,6) {$X \ox B^2 \ox X \ox B$};
\node (E1) at (-1,10) {$(X \ox B)^2$};
\node (E2) at (7.5,8) {$(X \ox B)^2$};
\node (E6) at (12.5,8) {$(X \ox B)^2$};
\node (F1) at (-1,14) {$(X \ox B)^2$};
\node (F2) at (5.5,14) {$(X \ox B)^3$};
\node (F3) at (12.5,14) {$X \ox B^2 \ox (X \ox B)^2$};
\node (F4) at (12.5,12) {$(X \ox B)^2 \ox  B \ox X \ox B$};
\node (F6) at (12.5,10) {$X^2 \ox B^2 \ox X \ox B$};
\node at (10.5,2) {$\eqref{conditionaction1}$};
\node at (4,1) {$\eqref{unital multiplication}+\eqref{u et delta}$};
\node at (5,11) {$\eqref{conditionaction1}$};
\draw[commutative diagrams/.cd, ,font=\scriptsize]
(F1.south) edge[commutative diagrams/equal]  (E1.north)
(E1.east) edge[commutative diagrams/equal]  (E2.west)
(E2.east) edge[commutative diagrams/equal]  (E6.west)
;
\path[->,font=\scriptsize]
(B5.south) edge node[descr] {$1_X \ox u_B \ox 1_X \ox u_B \ox m \; \; \; \; \; \; \; \;$} (A5.north west)
(F1.south east) edge node[descr] {$ 1_X \ox u_B \ox u_X \ox u_B \ox 1_B \ox 1_X \ox 1_B$} (F4.north west)
(A1.south) edge node[descr,xshift=25pt] {$1_X \ox 1_B \ox 1_X \ox  \act \ox 1_B \ox 1_B $} (A2.north)
(A2.south) edge node[descr,xshift=15pt] {$1_X \ox 1_B \ox m \ox m$} (A3.north)
(A3.east) edge node[above] {$1_X \ox \Delta \ox 1_X \ox 1_B$} (A4.west)
(A4.east) edge node[above] {$1_X \ox 1_B \ox {\blue \sigma_{B,X}} \ox 1_B$} (A5.west)
(A5.north) edge node[descr] {$1_X \ox \act \ox 1_B \ox 1_B$} (A6.south)
(B1.south) edge node[descr,xshift=25pt] {$1_X \ox 1_B \ox 1_X \ox 1_B \ox {\blue \sigma_{B,X}} \ox 1_B $} (A1.north)
(B5.west) edge node[above] {$\; 1_X \ox u_B \ox u_X \ox 1_X \ox 1_B \ox 1_B$} (A2.east)
(A6.north) edge node[descr] {$ m \ox m$} (B6.south)
(B5.north east) edge node[descr] {$m \ox m$} (B6.west)
(C2.south) edge node[descr] {$1_X \ox \act \ox 1_B \ox 1_B$} (B2.north)
(D2.south) edge node[descr] {$1_X \ox 1_B \ox {\blue \sigma_{B,X}} \ox 1_B$} (C2.north)
(E2.south) edge node[descr] {$1_X \ox \Delta \ox 1_X \ox 1_B$} (D2.north)
(E1.south) edge node[descr,xshift=25pt] {$1_X \ox u_B \ox u_X \ox 1_B \ox 1_X \ox 1_B$} (D1.north)
(D1.south) edge node[descr,xshift=25pt] {$1_X \ox 1_B \ox 1_X \ox \Delta \ox 1_X \ox 1_B$} (B1.north)
(E6.south) edge node[descr] {$m_{X \rtimes B}$} (B6.north)
(F6.south) edge node[descr] {$m \ox m \ox 1_X \ox 1_B$} (E6.north)
(F4.south) edge node[descr,xshift=-15pt] {$1_X \ox \act \ox 1_B \ox 1_B \ox 1_X \ox 1_B $} (F6.north)
(F2.east) edge node[above] {$1_X \ox \Delta \ox 1_X \ox 1_B \ox 1_X \ox 1_B $} (F3.west)
(F3.south) edge node[descr,xshift=-15pt] {$1_X  \ox 1_B \ox {\blue \sigma_{B,X}} \ox 1_B \ox 1_X \ox 1_B $} (F4.north)
(F1.east) edge node[above] {$1_X \ox u_B \ox u_X \ox 1_B \ox 1_X \ox 1_B $} (F2.west)
;
\end{tikzpicture} 
% textidote: ignore end
\caption{Condition (6)}
\label{Condition (6)}
\end{figure}
With similar computations we can show the ``partial associativities" (7) and (8), moreover (9) is clear.
\qed
 \end{proof}

We define a split extension of bialgebras by taking inspiration of the above Lemma.

\begin{definition}\label{definition split extension}
A split extension of bialgebras {\blue in $\C$} is given by a diagram {\blue in $\C$} \begin{equation}\label{split extension}
% textidote: ignore begin
\begin{tikzpicture}[descr/.style={fill=white},baseline=(A.base)] 
\node (A) at (0,0) {$A$};
\node (B) at (2.5,0) {$B,$};
\node (C) at (-2.5,0) {$X$};
\path[dashed,->,font=\scriptsize]
([yshift=2pt]A.west) edge node[above] {$\lambda$} ([yshift=2pt]C.east);
\path[->,font=\scriptsize]
([yshift=-4pt]C.east) edge node[below] {$\kappa$} ([yshift=-4pt]A.west)
([yshift=-4pt]A.east) edge node[below] {$\alpha$} ([yshift=-4pt]B.west)
([yshift=2pt]B.west) edge node[above] {$e$} ([yshift=2pt]A.east);
\end{tikzpicture} 
% textidote: ignore end 
\end{equation}
where $X$, $A$, $B$ are bialgebras, $\kappa$, $\alpha$, $e$ are morphisms of bialgebras, such that
\begin{itemize}
\item[(1)] $\lambda \cdot \kappa = 1_X$, $\alpha \cdot e =1_B$ 
\item[(2)] $\lambda \cdot e = u_X\cdot \epsilon_B$, $\alpha \cdot \kappa = u_B \cdot\epsilon_X$
\item[(3)] $m \cdot ((\kappa\cdot \lambda) \otimes (e \cdot \alpha)) \cdot \Delta = 1_A$
\item[(4)]$ \lambda \cdot m \cdot (\kappa \otimes e) =  1_X \otimes {\blue \epsilon_B}$
\item[(5)] $(1_B \otimes \lambda) \cdot (1_B \otimes m) \cdot ( 1_B \otimes e\otimes \kappa)  \cdot (\Delta \otimes 1_X)= (1_B \otimes \lambda) \cdot (1_B \otimes m) \cdot ( 1_B \otimes e\otimes \kappa) \cdot ({\blue \sigma_{B,B}} \otimes 1_X) \cdot (\Delta \otimes 1_X)$
\item[(6)] 
$m \cdot( m \otimes 1_A) \cdot (\kappa \otimes e \otimes 1_A) = m \cdot( 1_A \otimes m) \cdot (\kappa \otimes e \otimes 1_A)$
\item[(7)]
$m \cdot( m \otimes 1_A) \cdot (\kappa \otimes 1_A \otimes e) = m \cdot( 1_A \otimes m) \cdot (\kappa \otimes 1_A \otimes e)$ 
\item[(8)]
$m \cdot( m \otimes 1_A) \cdot (1_A \otimes \kappa \otimes e) = m \cdot( 1_A \otimes m) \cdot (1_A \otimes \kappa \otimes e)$
\item[(9)] $\lambda$ is a morphism of coalgebras preserving the unit.
\end{itemize} 
\end{definition}
\begin{remark}\label{useless axioms}We notice that the conditions $\lambda \cdot \kappa = 1_X$, $\lambda \cdot e = u_X \cdot\epsilon_B$ and $\lambda$ preserving the unit are consequences of the axiom $(4)$. The condition $(7)$ follows from $(3)$, $(6)$ and $(8)$, as we show in the following diagrams, where we use that $e$ and $\kappa$ are (bi)algebra morphisms. \begin{center}
% textidote: ignore begin
\begin{tikzpicture}[descr/.style={fill=white},
yscale=0.95,baseline=(A.base),scale=0.95] 
\node (A1) at (2.5,0) {$ A^4$};
\node (A2) at (5,0) {$ A^3$};
\node (B0) at (0,1.5) {$ X^2 \ox B^2$};
\node (B1) at (2.5,1.5) {$ A^4$};
\node (B2) at (5,1.5) {$ A^3$};
\node (B3) at (7.5,1.5) {$ A^2$};
\node (C1) at (2.5,3) {$ X^2 \ox B$};
\node (C2) at (5,3) {$ A^3$};
\node (C3) at (7.5,3) {$ A^2$};
\node (C4) at (10,3) {$ A$};
\node (D1) at (2.5,4.5) {$ X \ox B^2$};
\node (D2) at (5,4.5) {$ A^3$};
\node (D3) at (7.5,4.5) {$ A^2$};
\node (D4) at (10,4.5) {$ A$};
\node (E0) at (0,6) {$ X^2 \ox B^2$};
\node (E1) at (2.5,6) {$ A^4$};
\node (E2) at (5,6) {$ A^3$};
\node (E3) at (7.5,6) {$ A^2$};
\node (F1) at (2.5,7.5) {$ A^4$};
\node (F2) at (5,7.5) {$ A^3$};
\node at (3.75,0.75) {$(6)$};
\node at (3.75,6.75) {$(8)$};
\node at (6.25,2.25) {$(8)$};
\node at (6.25,5.25) {$(6)$};
\draw[commutative diagrams/.cd, ,font=\scriptsize]
(E0.south) edge[commutative diagrams/equal]  (B0.north)
(D3.south) edge[commutative diagrams/equal]  (C3.north)
(D4.south) edge[commutative diagrams/equal]  (C4.north);
%\draw[double distance = 2pt] (G.north) to[bend left=15]  (C.north west);
\path[->,font=\scriptsize]
(B0.south) edge node[descr] {$\kappa^2 \ox e^2$} (A1.north west)
(A1.east) edge node[above] {$ 1_A \ox m \ox 1_A $} (A2.west)
(B0.east) edge node[above] {$ \kappa^2 \ox e^2 $} (B1.west)
(B1.east) edge node[above] {$ 1_A \ox 1_A \ox m $} (B2.west)
(B2.east) edge node[above] {$  1_A \ox m $} (B3.west)
(C1.east) edge node[above] {$ \kappa^2 \ox e $} (C2.west)
(C2.east) edge node[above] {$ m \ox 1_A$} (C3.west)
(C3.east) edge node[above] {$ m $} (C4.west)
(D3.east) edge node[above] {$ m $} (D4.west)
(E0.east) edge node[above] {$ \kappa^2 \ox e^2 $} (E1.west)
(E1.east) edge node[above] {$ m \ox 1_A \ox 1_A  $} (E2.west)
(E2.east) edge node[above] {$  m \ox 1_A $} (E3.west)
(F1.east) edge node[above] {$ 1_A \ox m \ox 1_A $} (F2.west)
(D1.east) edge node[above] {$ \kappa \ox e^2 $} (D2.west)
(D2.east) edge node[above] {$ 1_A \ox m$} (D3.west)
(E0.south east) edge node[descr] {$m \ox 1_B \ox 1_B$} (D1.north west)
(C1.south east) edge node[descr] {$\kappa^2 \ox e$} (B2.north west)
(F2.south east) edge node[descr] {$m \ox 1_A$} (E3.north west)
(E3.south east) edge node[descr] {$m$} (D4.north west)
(B3.north east) edge node[descr] {$m$} (C4.south west)
(A2.north east) edge node[descr] {$1_A \ox m$} (B3.south west)
(B0.north east) edge node[descr] {$1_X \ox 1_X \ox m$} (C1.south west)
(D1.north east) edge node[descr] {$\kappa \ox e^2$} (E2.south west)
(E0.north) edge node[descr] {$\kappa^2 \ox e^2$} (F1.south west)
;
\end{tikzpicture} .
% textidote: ignore end
\end{center} Hence, by pre-composing with $(1_X \ox \lambda \ox \alpha \ox 1_B) \cdot (1_X \ox \Delta \ox 1_B)$ and using the condition (3), we obtain the following diagram
\begin{center}
% textidote: ignore begin
\begin{tikzpicture}[descr/.style={fill=white},baseline=(A0.base),
yscale=0.6,xscale=0.9] 
\node (A0) at (0,0) {$ X \ox A \ox B$};
\node (A1) at (4,0) {$ A^3$};
\node (A2) at (8,0) {$ A^2$};
\node (A3) at (12,0) {$ A.$};
\node (B0) at (4,4) {$ X \ox A^2 \ox B$};
\node (B3) at (12,4) {$ A^2$};
\node (C3) at (12,8) {$ A^3$};
\node (D0) at (0,12) {$ X \ox A \ox B$};
\node (D1) at (2,10) {$ X \ox A^2 \ox B$};
\node (D2) at (9,8) {$ X \ox A^2 \ox B$};
\node (D3) at (12,12) {$ X \ox A \ox B$};
\node (E) at (4,8) {$ X^2 \ox B^2$};
\node  at (1.5,6) {$(3)$};
\node  at (6,10) {$(3)$};
\draw[commutative diagrams/.cd, ,font=\scriptsize]
(D0.east) edge[commutative diagrams/equal]  (D3.west)
(D0.south) edge[commutative diagrams/equal]  (A0.north);
\path[->,font=\scriptsize]
(A0.east) edge node[above] {$ \kappa \ox 1_A \ox e$} (A1.west)
(A1.east) edge node[above] {$ 1_A \ox m  $} (A2.west)
(A2.east) edge node[above] {$ m  $} (A3.west)
(E.east) edge node[above] {$ 1_X \ox \kappa \ox e  \ox 1_B $} (D2.west)
(D2.north) edge node[descr] {$ 1_X \ox m \ox 1_B $} (D3.south west)
(D0.south east) edge node[descr] {$ 1_X \ox \Delta \ox 1_B $} (D1.north)
(D1.south east) edge node[descr] {$ 1_X \ox \lambda \ox \alpha \ox 1_B $} (E.north)
(E.south) edge node[descr] {$ 1_X \ox \kappa \ox e  \ox 1_B $} (B0.north)
(B0.south) edge node[descr] {$ 1_X \ox m \ox 1_B $} (A0.north east)
(D3.south) edge node[descr] {$  \kappa \ox 1_A \ox e $} (C3.north)
(C3.south) edge node[descr] {$ m \ox 1_A$} (B3.north)
(B3.south) edge node[descr] {$ m $} (A3.north)
;
\end{tikzpicture}
% textidote: ignore end
\end{center}Then, we conclude that condition (7) holds. 

{\blue In the associative case, the conditions $(6)$, $(7)$ and $(8)$ in Definition \ref{definition split extension} trivially become redundant.
Note that these ``partial associativity conditions'' are versions of the conditions considered in the article \cite{GJS}. More precisely, this definition generalizes the notion of split extension of unitary magmas in \cite{GJS}. Indeed, by taking $\C = \sf Set$, we obtain exactly the definition introduced in \cite{GJS}.  }

\end{remark}
{\blue 
\begin{proposition}
If there exists a morphism $\lambda$ satisfying the conditions of Definition \ref{definition split extension}, it has to be unique. 
\end{proposition}
}
\begin{proof}
Let us suppose that {\blue there exist two morphisms $\lambda$ and $\lambda'$} satisfying the conditions of Definition \ref{definition split extension}. Then the commutativity of the following diagram shows that $\lambda'$ has to be equal to $\lambda$.
\begin{center}
% textidote: ignore begin
\begin{tikzpicture}[descr/.style={fill=white},baseline=(A.base),scale=0.75] 
\node (A) at (0,0) {$ A$};
\node (B) at (10,0) {$X$};
\node (C) at (10,3) {$A  $};
\node (D) at (7.5,3) {$ A^2$};
\node (E) at (5,3) {$ X \ox B$};
\node (F) at (2.5,3) {$ A^2$};
\node (G) at (0,3) {$ A $};
\node at (8,2) {$ (4) $};
\node at (5,3.7) {$ (3) $};
\node at (4,1.5) {$ \eqref{counit comultiplication} $};
\draw[commutative diagrams/.cd, ,font=\scriptsize]
%(L.south east) edge[commutative diagrams/equal]  (C.north west)
%(I.south) edge[commutative diagrams/equal]  (A.north)
(G.south) edge[commutative diagrams/equal]  (A.north);
\draw[double distance = 2pt] (G.north) to[bend left=20]  (C.north west);
\path[->,font=\scriptsize]
%(F.south) edge node[descr] {$1_X \ox 1_B \ox \sigma \ox 1_B$} (E.north)
%(E.south) edge node[descr] {$1_X \ox  \act \ox 1_B \ox 1_B$} (D.north)
(E.south) edge node[descr] {$ 1_X \ox {\blue \epsilon_B}$} (B.north west)
(C.south) edge node[descr] {$\lambda$} (B.north)
(A.east) edge node[above] {$ \lambda' $} (B.west)
(D.east) edge node[above] {$m$} (C.west)
(E.east) edge node[above] {$\kappa \ox e$} (D.west)
(F.east) edge node[above] {$\lambda'  \ox \alpha$} (E.west)
(G.east) edge node[above] {$\Delta$} (F.west)
%(K.east) edge node[above] {$\act \ox 1_B$} (L.west)
%(L.east) edge node[above] {$u_X \ox 1_X \ox 1_B \ox u_B$} (D.west)
;
\end{tikzpicture}
 % textidote: ignore end
\end{center}
%Note that the uniqueness of $\lambda$ also follows from the fact that the arrows $\kappa$ and $e$ are jointly epimorphic {\red in the category of bialgebras} (see Lemma \ref{jointly epic}).
\qed
\end{proof}

Another important property of the split extensions of bialgebras is given by the following proposition:

\begin{proposition}\label{jointly epic}
Let 
% textidote: ignore begin
\begin{tikzpicture}[descr/.style={fill=white},baseline=(A.base)] 
\node (A) at (0,0) {$A$};
\node (B) at (2.5,0) {$B$};
\node (C) at (-2.5,0) {$X$};
\path[dashed,->,font=\scriptsize]
([yshift=2pt]A.west) edge node[above] {$\lambda$} ([yshift=2pt]C.east);
\path[->,font=\scriptsize]
([yshift=-4pt]C.east) edge node[below] {$\kappa$} ([yshift=-4pt]A.west)
([yshift=-4pt]A.east) edge node[below] {$\alpha$} ([yshift=-4pt]B.west)
([yshift=2pt]B.west) edge node[above] {$e$} ([yshift=2pt]A.east);
\end{tikzpicture} 
% textidote: ignore end
be a split extension of bialgebras, then $\kappa$ and $e$ are jointly epimorphic in the category of bialgebras (and in the category of algebras)
\end{proposition}

\begin{proof}
Let $v,w \colon A \rightarrow Y$ be {\blue two} morphisms of bialgebras such that $ v \cdot \kappa = w \cdot \kappa$ and $ v \cdot e = w \cdot e$. 
\begin{center}
% textidote: ignore begin
\begin{tikzpicture}[descr/.style={fill=white},baseline=(A.base),scale=0.8,yscale=0.6] 
\node (A) at (0,0) {$ A $};
\node (B) at (9,0) {$Y $};
\node (C) at (9,9) {$A $};
\node (D) at (0,9) {$A$};
\node (A') at (3,3) {$ A^2 $};
\node (B') at (6,3) {$Y^2 $};
\node (C') at (6,6) {$A^2 $};
\node (D') at (3,6) {$A^2$};
\node (X) at (1.5,5.5) {$(3)$};
\node (Y) at (4.5,7.75) {$(3)$};
\draw[commutative diagrams/.cd, ,font=\scriptsize]
(D.south) edge[commutative diagrams/equal]  (A.north)
(D.east) edge[commutative diagrams/equal]  (C.west);
\path[->,font=\scriptsize]
(A.east) edge node[above] {$w$} (B.west)
(A'.east) edge node[above] {$w \ox w$} (B'.west)
(D'.east) edge node[above] {$(\kappa \cdot \lambda) \ox (e \cdot \alpha)$} (C'.west)
(C.south) edge node[descr] {$v$} (B.north)
(C'.south) edge node[descr] {$v \ox v$} (B'.north)
(D'.south) edge node[descr] {$(\kappa \cdot \lambda) \ox (e \cdot \alpha)$} (A'.north)
(D.south east) edge node[descr] {$\Delta$} (D'.north west)
(A'.south west) edge node[descr] {$m$} (A.north east)
(C'.north east) edge node[descr] {$m$} (C.south west)
(B'.south east) edge node[descr] {$m$} (B.north west)
;
\end{tikzpicture}
% textidote: ignore end
\end{center}
The above diagram allows us to conclude that $v$ and $w$ are equal. Hence, $\kappa$ and $e$ are jointly epimorphic. Note that we only use that $v$ and $w$ are algebra morphisms (we do not need them to be coalgebra morphisms).
\qed
\end{proof}

\begin{definition}\label{morph split ext}
A morphism of split extensions from  
% textidote: ignore begin
\begin{tikzpicture}[descr/.style={fill=white},baseline=(A.base),xscale=0.6] 
\node (A) at (0,0) {$A$};
\node (B) at (2.5,0) {$B$};
\node (C) at (-2.5,0) {$X$};
\path[dashed,->,font=\scriptsize]
([yshift=2pt]A.west) edge node[above] {$\lambda$} ([yshift=2pt]C.east);
\path[->,font=\scriptsize]
([yshift=-4pt]C.east) edge node[below] {$\kappa$} ([yshift=-4pt]A.west)
([yshift=-4pt]A.east) edge node[below] {$\alpha$} ([yshift=-4pt]B.west)
([yshift=2pt]B.west) edge node[above] {$e$} ([yshift=2pt]A.east);
\end{tikzpicture} % textidote: ignore end
  to 
  % textidote: ignore begin
\begin{tikzpicture}[descr/.style={fill=white},baseline=(A.base),xscale=0.6] 
\node (A) at (0,0) {$A'$};
\node (B) at (2.5,0) {$B'$};
\node (C) at (-2.5,0) {$X'$};
\path[dashed,->,font=\scriptsize]
([yshift=2pt]A.west) edge node[above] {$\lambda'$} ([yshift=2pt]C.east);
\path[->,font=\scriptsize]
([yshift=-4pt]C.east) edge node[below] {$\kappa'$} ([yshift=-4pt]A.west)
([yshift=-4pt]A.east) edge node[below] {$\alpha'$} ([yshift=-4pt]B.west)
([yshift=2pt]B.west) edge node[above] {$e'$} ([yshift=2pt]A.east);
\end{tikzpicture} 
% textidote: ignore end
is given by {\blue three} morphisms of bialgebras $g\colon B \rightarrow B'$, $v \colon X \rightarrow X'$ and $p \colon A \rightarrow A'$ such {\blue that the diagram}  \begin{equation}\label{commute} 
% textidote: ignore begin
\begin{tikzpicture}[descr/.style={fill=white},baseline=(current  bounding  box.center),xscale=0.7] 
\node (A) at (0,0) {$A$};
\node (B) at (2.5,0) {$B$};
\node (C) at (-2.5,0) {$X$};
\node (A') at (0,-2) {$A'$};
\node (B') at (2.5,-2) {$B'$};
\node (C') at (-2.5,-2) {$X'$};
\path[dashed,->,font=\scriptsize]
([yshift=2pt]A.west) edge node[above] {$\lambda$} ([yshift=2pt]C.east)
([yshift=2pt]A'.west) edge node[above] {$\lambda'$} ([yshift=2pt]C'.east);
\path[->,font=\scriptsize]
(B.south) edge node[right] {$ g$}  (B'.north)
 (C.south) edge node[left] {$ v $}  (C'.north)
(A.south) edge node[left] {$ p$} (A'.north)
([yshift=-4pt]C'.east) edge node[below] {$\kappa'$} ([yshift=-4pt]A'.west)
([yshift=-4pt]A'.east) edge node[below] {$\alpha'$} ([yshift=-4pt]B'.west)
([yshift=2pt]B'.west) edge node[above] {$e'$} ([yshift=2pt]A'.east)
([yshift=-4pt]C.east) edge node[below] {$\kappa$} ([yshift=-4pt]A.west)
([yshift=-4pt]A.east) edge node[below] {$\alpha$} ([yshift=-4pt]B.west)
([yshift=2pt]B.west) edge node[above] {$e$} ([yshift=2pt]A.east);
\end{tikzpicture}
% textidote: ignore end 
\end{equation}
{\blue commutes in $\C$.}
\end{definition}
We do not need to ask the commutativity of all the squares thanks to this corollary:
\begin{corollary} \label{comm}
Let $(g,v,p)$ be a morphism of split extensions of bialgebras 
\begin{center}
% textidote: ignore begin
\begin{tikzpicture}[descr/.style={fill=white},baseline=(A'.base),xscale=0.7] 
\node (A) at (0,0) {$A$};
\node (B) at (2.5,0) {$B$};
\node (C) at (-2.5,0) {$X$};
\node (A') at (0,-2) {$A'$};
\node (B') at (2.5,-2) {$B'.$};
\node (C') at (-2.5,-2) {$X'$};
\path[dashed,->,font=\scriptsize]
([yshift=2pt]A.west) edge node[above] {$\lambda$} ([yshift=2pt]C.east)
([yshift=2pt]A'.west) edge node[above] {$\lambda'$} ([yshift=2pt]C'.east);
\path[->,font=\scriptsize]
(B.south) edge node[right] {$ g$}  (B'.north)
 (C.south) edge node[left] {$ v $}  (C'.north)
(A.south) edge node[left] {$ p$} (A'.north)
([yshift=-4pt]C'.east) edge node[below] {$\kappa'$} ([yshift=-4pt]A'.west)
([yshift=-4pt]A'.east) edge node[below] {$\alpha'$} ([yshift=-4pt]B'.west)
([yshift=2pt]B'.west) edge node[above] {$e'$} ([yshift=2pt]A'.east)
([yshift=-4pt]C.east) edge node[below] {$\kappa$} ([yshift=-4pt]A.west)
([yshift=-4pt]A.east) edge node[below] {$\alpha$} ([yshift=-4pt]B.west)
([yshift=2pt]B.west) edge node[above] {$e$} ([yshift=2pt]A.east);
\end{tikzpicture}
% textidote: ignore end 
\end{center}
If $ p \cdot \kappa = \kappa' \cdot v$ and $ p \cdot e =  e' \cdot g$ hold, then the identities $ \lambda' \cdot p = v \cdot \lambda $ and $\alpha' \cdot p = g \cdot \alpha  $ follow, and conversely. 
\end{corollary}
\begin{proof}
Let us suppose that $p \cdot \kappa = \kappa' \cdot v$ and $p \cdot e = e' \cdot g$, then by Proposition \ref{jointly epic}, we can prove that $\lambda' \cdot p = v \cdot \lambda$ by checking that $ (\lambda' \cdot p)  \cdot \kappa = (v \cdot \lambda) \cdot \kappa$ and $ (\lambda' \cdot p) \cdot e = (v \cdot \lambda) \cdot e$, {\red which is done in the following identities}:
\[  (\lambda' \cdot p)  \cdot \kappa = \lambda' \cdot \kappa' \cdot v = v =(v \cdot \lambda) \cdot \kappa, \]
\[  (\lambda' \cdot p) \cdot e = \lambda' \cdot e' \cdot g = {\blue \epsilon_B} = (v \cdot \lambda) \cdot e. \] Similarly, one can check that $\alpha' \cdot p = g \cdot \alpha$.
\qed
\end{proof}

\begin{proposition}\label{prop lamda morph}
Let 
% textidote: ignore begin
\begin{tikzpicture}[descr/.style={fill=white},baseline=(A.base)] 
\node (A) at (0,0) {$A$};
\node (B) at (2.5,0) {$B$};
\node (C) at (-2.5,0) {$X$};
\path[dashed,->,font=\scriptsize]
([yshift=2pt]A.west) edge node[above] {$\lambda$} ([yshift=2pt]C.east);
\path[->,font=\scriptsize]
([yshift=-4pt]C.east) edge node[below] {$\kappa$} ([yshift=-4pt]A.west)
([yshift=-4pt]A.east) edge node[below] {$\alpha$} ([yshift=-4pt]B.west)
([yshift=2pt]B.west) edge node[above] {$e$} ([yshift=2pt]A.east);
\end{tikzpicture} 
% textidote: ignore end
  be a split extension of bialgebras, then the following diagram 
 % textidote: ignore begin
\begin{equation}\label{lambda morph of alg}
\begin{tikzpicture}[descr/.style={fill=white},baseline=(current  bounding  box.center),yscale=0.9] 
\node (A) at (0,0) {$ A$};
\node (B) at (9,0) {$X$};
\node (C) at (9,1.5) {$X^2 $};
\node (D) at (9,3) {$ A^2$};
\node (E) at (6.5,3) {$A^3$};
\node (F) at (2.5,3) {$ A^3$};
\node (G) at (0,3) {$ A^2 $};
\path[->,font=\scriptsize]
(G.south) edge node[descr] {$m$} (A.north)
(D.south) edge node[descr] {$\lambda \ox \lambda$} (C.north)
(C.south) edge node[descr] {$ m$} (B.north)
(A.east) edge node[above] {$ \lambda$} (B.west)
(E.east) edge node[above] {$1_A \ox m$} (D.west)
(F.east) edge node[above] {$1_A \ox (e \cdot \alpha) \ox (\kappa \cdot \lambda)$} (E.west)
(G.east) edge node[above] {$\Delta \ox 1_A$} (F.west)
%(K.east) edge node[above] {$\act \ox 1_B$} (L.west)
%(L.east) edge node[above] {$u_X \ox 1_X \ox 1_B \ox u_B$} (D.west)
;
\end{tikzpicture}
\end{equation}% textidote: ignore end
commutes in $\C$.
\end{proposition}
\begin{proof} 
In order to prove the proposition, we make the three Figures \ref{DiagramA2}, \ref{DiagramB2} and \ref{DiagramC2} commute. %Then we compose them to obtain the conclusion.
 The Figure \ref{DiagramA2} is commutative since $e$ and $\alpha$ are morphisms of (bi)algebras.
\begin{figure}[b]
\centering
% textidote: ignore begin
\begin{tikzpicture}[descr/.style={fill=white},baseline=(A.base),xscale=1.6,yscale=0.75] 
\node (A) at (0,0) {$ A^4$};
\node (B) at (3,0) {$A^3$};
\node (C) at (6,0) {$A^2 $};
\node (D) at (8,0) {$ A$};
\node (E) at (8,3) {$A^2$};
\node (F) at (8,6) {$ A^3$};
\node (G) at (8,9) {$ A^3 $};
\node (H) at (8,12) {$ A^4$};
\node (I) at (0,12) {$A^5$};
\node (J) at (0,9) {$A^5 $};
\node (K) at (0,6) {$ A^5$};
\node (L) at (0,3) {$A^4$};
\node (M) at (4,9) {$ A^5$};
\node (N) at (6,9) {$ A^3 $};
\node (O) at (4,6) {$ A^4$};
\node (P) at (4,3) {$A^4$};
\node (Q) at (6,3) {$A^3 $};
\node at (7.5,4.5) {$(8)$};
\node at (4,1.5) {$(7)$};
\draw[commutative diagrams/.cd, ,font=\scriptsize]
(I.south) edge[commutative diagrams/equal]  (J.north);
\path[->,font=\scriptsize]
(A.east) edge node[above] {$ 1_A \ox m \ox 1_A$} (B.west)
(B.east) edge node[above] {$ 1_A \ox m $} (C.west)
(C.east) edge node[above] {$  m $} (D.west)
(P.east) edge node[above] {$ 1_A \ox 1_A \ox m $} (Q.west)
(M.east) edge node[below] {$ 1_A \ox m \ox m $} (N.west)
(I.east) edge node[above] {$ 1_A \ox (e \cdot \alpha) \ox (\kappa \cdot \lambda) \ox m $} (H.west)
(G.west) edge node[above] {$ 1_A \ox 1_A \ox (e \cdot \alpha) $} (N.east)
(J.east) edge node[above] {$ 1_A \ox (e \cdot \alpha) \ox (\kappa \cdot \lambda)  \ox (e \cdot \alpha) \ox (e \cdot \alpha) $} (M.west)
(H.south) edge node[descr] {$1_A \ox m \ox 1_A$} (G.north)
(G.south) edge node[descr] {$ (\kappa \cdot \lambda)^2   \ox (e \cdot \alpha)$} (F.north)
(F.south) edge node[descr] {$ m \ox 1_A$} (E.north)
(E.south) edge node[descr] {$ m$} (D.north)
(L.south) edge node[descr] {$ (\kappa \cdot \lambda) ^2  \ox 1_A \ox 1_A$} (A.north)
(N.south) edge node[descr] {$ (\kappa \cdot \lambda)^2 \ox 1_A$} (Q.north)
(O.south) edge node[descr] {$ (\kappa \cdot \lambda)^2  \ox 1_A \ox 1_A$} (P.north)
(K.south) edge node[descr] {$1_A \ox m \ox 1_A \ox  1_A$} (L.north)

(Q.south) edge node[descr] {$1_A \ox m$} (C.north)
(M.south) edge node[descr] {$  1_A \ox m \ox 1_A \ox  1_A$} (O.north)
(J.south) edge node[descr] {$ 1_A \ox (e \cdot \alpha) \ox (\kappa \cdot \lambda)  \ox (e \cdot \alpha)^2$} (K.north)
;
\end{tikzpicture} 
% textidote: ignore end
\caption{Diagram $(A)$}
\label{DiagramA2}
\end{figure}
The Figure \ref{DiagramB2} commutes thanks to the counitality of the comultiplication, the unitality of the multiplication and the fact that $(e \cdot \alpha)$ is a morphism of (bi)algebras.
\begin{figure}
\centering
% textidote: ignore begin
\begin{tikzpicture}[descr/.style={fill=white},xscale=1.1,yscale=0.65] 
\node (A) at (0,10) {$ A^2$};
\node (B) at (4,10) {$A^4$};
\node (C) at (8,10) {$A^5 $};
\node (D) at (12,10) {$ A^5$};
\node (E) at (12,8) {$A^5$};
\node (F) at (12,6) {$ A^4$};
\node (G) at (12,4) {$ A^4 $};
\node (H) at (12,2) {$ A^4$};
\node (I) at (12,0) {$A^4$};
\node (J) at (6,0) {$A^4 $};
\node (K) at (0,0) {$ A^6$};
\node (L) at (0,2) {$A^6$};
\node (M) at (0,4) {$ A^6$};
\node (N) at (0,6) {$ A^6 $};
\node (O) at (0,8) {$ A^4$};
\node (P) at (5,4) {$A^6$};
\node (Q) at (8.5,4) {$A^5 $};
\node (R) at (4,2) {$A^5$};
\node at (6,6) {$\eqref{unital multiplication}+\eqref{counit comultiplication}$};
\node at (6,3) {$(1)+(2)$};
\draw[commutative diagrams/.cd, ,font=\scriptsize]
(H.south) edge[commutative diagrams/equal]  (I.north)
(G.south) edge[commutative diagrams/equal]  (H.north)
(L.south) edge[commutative diagrams/equal]  (K.north);
\path[->,font=\scriptsize]
(A.east) edge node[above] {$ \Delta \ox \Delta $} (B.west)
(B.east) edge node[above] {$ 1_A \ox \Delta \ox 1_A \ox 1_A $} (C.west)
(C.east) edge node[above] {$ 1_A \ox 1_A \ox {\blue \sigma_{A,A}} \ox 1_A $} (D.west)
(P.east) edge node[above] {$ 1_A \ox m \ox 1_A \ox 1_A \ox 1_A $} (Q.west)
(K.east) edge node[above] {$ 1_A \ox m \ox m \ox 1_A $} (J.west)
(J.east) edge node[above] {$ 1_A \ox(\kappa \cdot \lambda) \ox  (e \cdot \alpha) \ox 1_A $} (I.west)
(M.south) edge node[descr] {$  ((\kappa \cdot \lambda) \ox  (e \cdot \alpha))^3$} (L.north)
(M.east) edge node[above] {$ ((\kappa \cdot \lambda) \ox (e \cdot \alpha))^2 \ox (u_A \cdot {\blue \epsilon_A}) \ox (e \cdot \alpha) $} (P.west)
(Q.east) edge node[above] {$ 1_A  \ox (\kappa \cdot \lambda) \ox m \ox 1_A$} (G.west)
(L.east) edge node[above] {$ 1_A \ox m \ox (e \cdot \alpha)^2 \ox 1_A $} (R.west)
(R.east) edge node[above] {$ 1_A \ox  (\kappa \cdot \lambda)  \ox m \ox 1_A $} (H.west)
(D.south) edge node[descr,xshift=-5mm] {$  ((\kappa \cdot \lambda) \ox  (e \cdot \alpha))^2 \ox  (e \cdot \alpha) $} (E.north)
(F.south) edge node[descr,xshift=-2mm] {$  1_A \ox (\kappa \cdot \lambda) \ox 1_A \ox 1_A $} (G.north)
(A.south) edge node[descr] {$ \Delta \ox \Delta$} (O.north)
(O.south) edge node[descr] {$  1_A \ox \Delta \ox \Delta \ox 1_A$} (N.north)
(N.south) edge node[descr] {$1_{A^2} \ox {\blue \sigma_{A,A}} \ox 1_{A^2}$} (M.north)
(E.south) edge node[descr] {$1_A \ox m \ox 1_A \ox 1_A$} (F.north)
;
\end{tikzpicture} % textidote: ignore end
\caption{Diagram $(B)$}
\label{DiagramB2}
\end{figure}
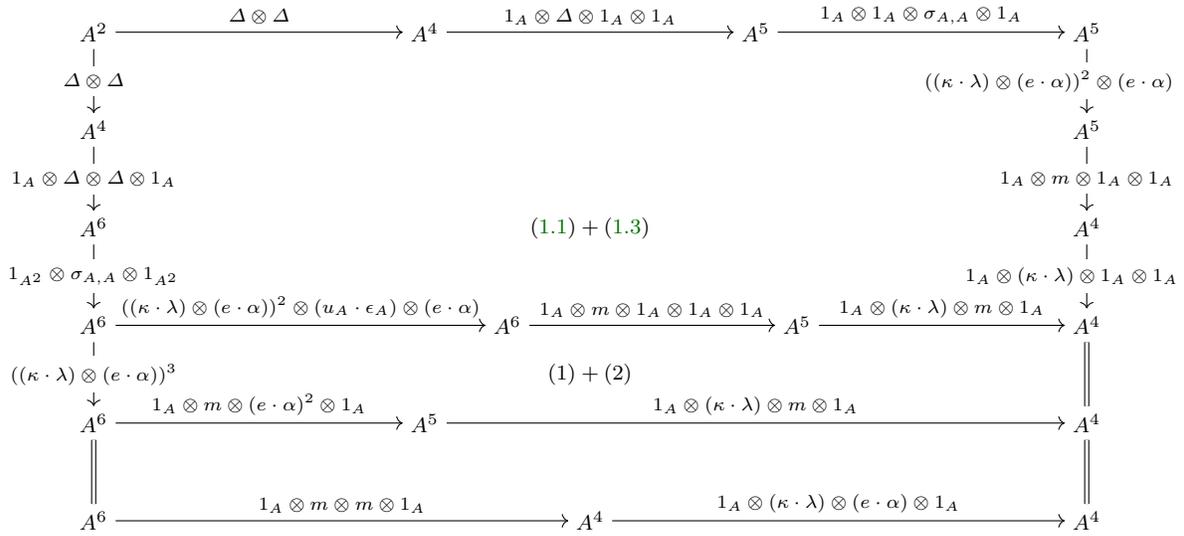
In Figure \ref{DiagramC2} we use the fact that we work with coalgebra morphisms.

\begin{figure}
\centering
% textidote: ignore begin
\begin{tikzpicture}[descr/.style={fill=white},yscale=0.8] 
%\node at (-0,14.5) {$(C)$};
\node (A) at (1,1.25) {$ A^2$};
\node (B) at (7.5,1.25) {$A^2$};
\node (C) at (15,1.25) {$A $};
\node (D) at (15,2.5) {$ A$};
\node (E) at (15,4) {$A^2$};
\node (F) at (15,6) {$ A^3$};
\node (G) at (15,8) {$ A^4 $};
\node (H) at (15,10) {$ A^4$};
\node (I) at (15,12) {$A^6$};
\node (J) at (15,14) {$A^6 $};
\node (K) at (12,14) {$ A^6$};
\node (L) at (9,14) {$A^4$};
\node (M) at (6,14) {$ A^4$};
\node (N) at (3,14) {$ A^2 $};
\node (O) at (1,14) {$ A^2$};
\node (P) at (3,8) {$A^4$};
\node (Q) at (6,12) {$A^4 $};
\node (R) at (9,12) {$A^4$};
\node (S) at (12,12) {$A^6$};
\node (T) at (9,10) {$ A^3$};
\node (U) at (6,4) {$ A^3 $};
\node (O') at (3,12) {$ A^4$};
\node (X) at (3,2.5) {$A^3$};
%\node (X) at (6,15) {$A^5 $};
\node (Y) at (7.5,2.5) {$A^2$};
\node at (13,8) {$(3)$};
\node at (9,6) {$(8)$};
\node at (4.5,7) {$(6)$};
\node at (2,7) {$(3)$};
\draw[commutative diagrams/.cd, ,font=\scriptsize]
(T.south east) edge[commutative diagrams/equal]  (F.west)
(M.east) edge[commutative diagrams/equal]  (L.west)
(O.east) edge[commutative diagrams/equal]  (N.west)
(A.east) edge[commutative diagrams/equal]  (B.west)
(D.south) edge[commutative diagrams/equal]  (C.north)
(Y.south) edge[commutative diagrams/equal]  (B.north)
(O.south) edge[commutative diagrams/equal]  (A.north);
\path[->,font=\scriptsize]
(N.east) edge node[above] {$ \Delta \ox \Delta $} (M.west)
(L.east) edge node[above] {$ 1_A \ox \Delta \ox \Delta \ox 1_A $} (K.west)
(R.east) edge node[above] {$ 1_A \ox \Delta \ox \Delta \ox 1_A $} (S.west)
(T.east) edge node[above] {$ 1_A \ox \Delta \ox 1_A $} (H.west)
(U.east) edge node[above] {$ 1_A \ox m $} (E.west)
(K.east) edge node[above] {$ 1_{A^2} \ox  {\blue \sigma_{A,A}} \ox 1_{A^2} $} (J.west)
(S.east) edge node[above] {$ 1_{A^2} \ox  {\blue \sigma_{A,A}} \ox 1_{A^2}  $} (I.west)
(X.east) edge node[above] {$  m \ox 1_A$} (Y.west)
(Y.east) edge node[above] {$  m $} (D.west)
(B.east) edge node[above] {$  m $} (C.west)
(L.south) edge node[descr] {$     \; \; \; \; \;\; \; \; \; \;  ((\kappa \cdot \lambda) \ox  (e \cdot \alpha))^2$} (R.north)
(O'.south) edge node[descr] {$((\kappa \cdot \lambda) \ox  (e \cdot \alpha))^2$} (P.north)
(M.south) edge node[descr] {$((\kappa \cdot \lambda) \ox  (e \cdot \alpha))^2 $} (Q.north)
(J.south) edge node[descr,xshift=-4mm] {$  ((\kappa \cdot \lambda) \ox  (e \cdot \alpha))^3 $} (I.north)
%(M.east) edge node[above] {$ (\kappa \cdot \lambda) \ox (e \cdot \alpha) \ox (\kappa \cdot \lambda) \ox (e \cdot \alpha) \ox (u_A \cdot \epsilon) \ox (e \cdot \alpha) $} (P.west)
%(Q.east) edge node[above] {$ 1_A  \ox (\kappa \cdot \lambda) \ox m \ox 1_A$} (G.west)
%(L.east) edge node[above] {$ 1_A \ox m \ox (e \cdot \alpha)\ox (e \cdot \alpha)\ox 1_A $} (R.west)
%(R.east) edge node[above] {$ 1_A \ox  (\kappa \cdot \lambda)  \ox m \ox 1_A $} (H.west)
%(D.south) edge node[descr] {$  (\kappa \cdot \lambda) \ox  (e \cdot \alpha) \ox(\kappa \cdot \lambda) \ox  (e \cdot \alpha)\ox  (e \cdot \alpha) $} (E.north)
(H.south) edge node[descr,xshift=-8mm] {$  1_A \ox (\kappa \cdot \lambda) \ox (e \cdot \alpha) \ox 1_A $} (G.north)
(N.south) edge node[descr] {$ \Delta \ox \Delta$} (O'.north)
%(O.south) edge node[descr] {$  1_A \ox \Delta \ox \Delta \ox 1_A$} (N.north)
%(N.south) edge node[descr] {$1_A \ox 1_A \ox \sigma \ox 1_A \ox  1_A$} (M.north)
(P.south) edge node[descr] {$1_A \ox 1_A \ox m $} (X.north)
(Q.south) edge node[descr] {$1_A \ox 1_A \ox m $} (U.north)
(R.south) edge node[descr] {$ 1_A  \ox m \ox 1_A$} (T.north)
(I.south) edge node[descr,xshift=-15pt] {$1_A  \ox m^2 \ox 1_A$} (H.north)
(G.south) edge node[descr] {$1_A  \ox m \ox 1_A$} (F.north)
(F.south) edge node[descr] {$1_A  \ox m $} (E.north)
(E.south) edge node[descr] {$m $} (D.north)
;
\end{tikzpicture} % textidote: ignore end
\caption{Diagram $(C)$}
\label{DiagramC2}
\end{figure}
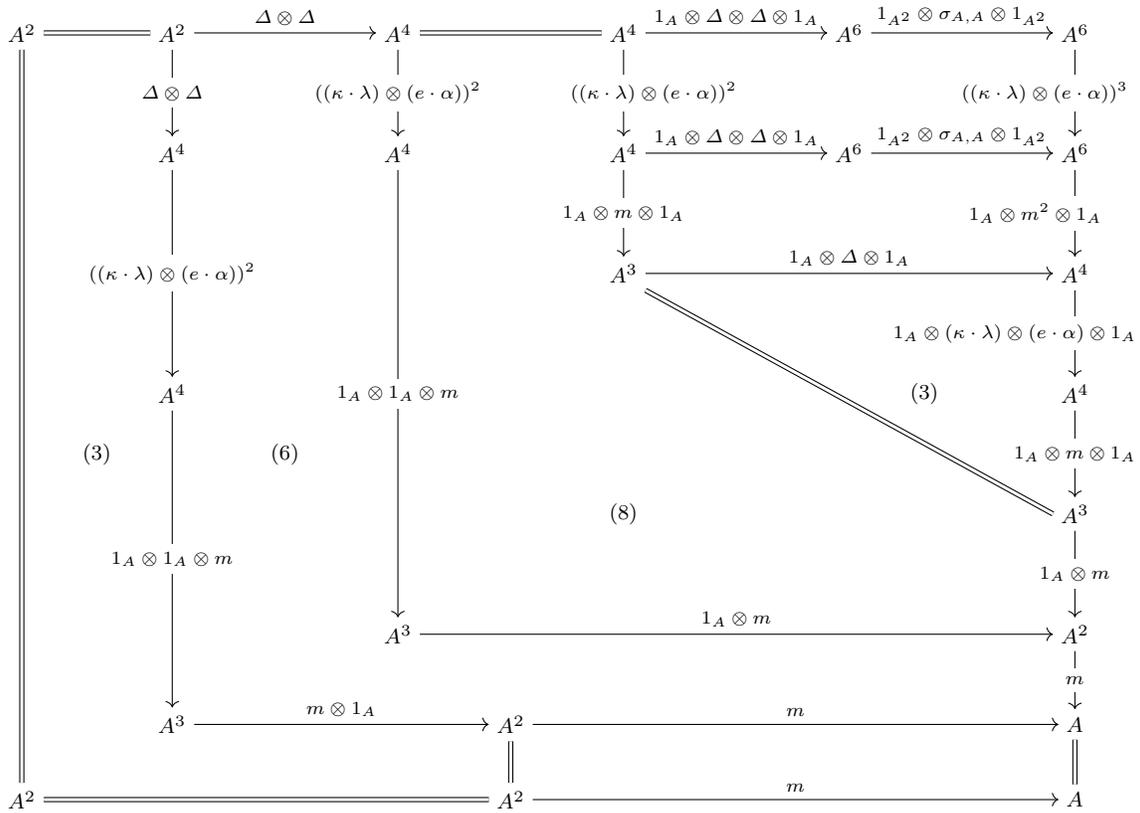
Finally, we combine the Figures \ref{DiagramA2}, \ref{DiagramB2} and \ref{DiagramC2} in Figure \ref{combinationABC} (in the appendix) and we obtain the following equality \begin{eqnarray*}
m =  m \cdot (\kappa \ox e) \cdot (m \ox 1_B) \cdot (\lambda \ox \lambda \ox \alpha) \cdot(1_A \ox m \ox 1_A) \cdot (1_A \otimes (e\cdot \alpha) \otimes (\kappa \cdot \lambda) \ox m)  \\  \cdot (1_A \ox 1_A \ox {\blue \sigma_{A,A}} \ox 1_A) \cdot (1_A \ox \Delta \ox 1_A \ox 1_A) \cdot (\Delta \otimes \Delta).
\end{eqnarray*} 

By composing with $\lambda$ and using the condition $(4)$, we can conclude that 
\[ \lambda \cdot m = m \cdot (\lambda \otimes \lambda) \cdot (1_A \otimes m) \cdot (1_A \otimes (e\cdot \alpha) \otimes (\kappa \cdot \lambda)) \cdot (\Delta \otimes 1_A). \]\qed 
\end{proof}

\begin{proposition}\label{Action bialgebra}
Given a split extension 
% textidote: ignore begin
\begin{tikzpicture}[descr/.style={fill=white},baseline=(A.base)] 
\node (A) at (0,0) {$A$};
\node (B) at (2.5,0) {$B$};
\node (C) at (-2.5,0) {$X$};
\path[dashed,->,font=\scriptsize]
([yshift=2pt]A.west) edge node[above] {$\lambda$} ([yshift=2pt]C.east);
\path[->,font=\scriptsize]
([yshift=-4pt]C.east) edge node[below] {$\kappa$} ([yshift=-4pt]A.west)
([yshift=-4pt]A.east) edge node[below] {$\alpha$} ([yshift=-4pt]B.west)
([yshift=2pt]B.west) edge node[above] {$e$} ([yshift=2pt]A.east);
\end{tikzpicture} 
% textidote: ignore end
of bialgebras, we  can construct an action of bialgebras, $\act \colon B \ox X \rightarrow X$ defined by \[\triangleright  = \lambda \cdot m \cdot (e \otimes \kappa).\]
\end{proposition}
\begin{proof}
We check all the axioms of the definition of actions of bialgebras (Definition \ref{def action}).
\begin{center}
% textidote: ignore begin
\begin{tikzpicture}[descr/.style={fill=white},baseline=(A.base),yscale=0.8] 
\node (A) at (-2,0) {$X $};
\node (B) at (0,0) {$X$};
\node (C) at (0,2) {$A$};
\node (D) at (0,4) {$A^2$};
\node (E) at (0,6) {$B\ox X$};
\node (F) at (-2,6) {$X $};
\draw[->] (E.south east) to[bend left=20]node[left,scale=0.8] {${\blue \act} $} (B.north east);
%\draw[equal] (J.south) -- (Y.north);
\draw[commutative diagrams/.cd, ,font=\scriptsize]
%(I.south) edge[commutative diagrams/equal]  (X.north)
%(X.south) edge[commutative diagrams/equal]  (J.north)
%(J.south) edge[commutative diagrams/equal] (Y.north)
(F.south) edge[commutative diagrams/equal]  (A.north)
(A.east) edge[commutative diagrams/equal]  (B.west);
\path[->,font=\scriptsize]
(F.east) edge node[above] {$u_B  \ox 1_X$} (E.west)
(C.south) edge node[descr] {$\lambda$} (B.north)
(D.south) edge node[descr] {$m$} (C.north)
(F.south east) edge node[descr] {$\kappa$} (C.north west)
(E.south) edge node[descr] {$e \ox \kappa$} (D.north)
;
\end{tikzpicture} 
\qquad
\begin{tikzpicture}[descr/.style={fill=white},baseline=(A.base),yscale=0.8] 
\node (A) at (-2,0) {$B $};
\node (B) at (0,0) {$X$};
\node (C) at (0,2) {$A$};
\node (D) at (0,4) {$A^2$};
\node (E) at (0,6) {$B\ox X$};
\node (F) at (-2,6) {$B $};
\draw[->] (E.south east) to[bend left=20]node[left,scale=0.8] {${\blue \act} $} (B.north east);
%\draw[equal] (J.south) -- (Y.north);
\draw[commutative diagrams/.cd, ,font=\scriptsize]
%(I.south) edge[commutative diagrams/equal]  (X.north)
%(X.south) edge[commutative diagrams/equal]  (J.north)
%(J.south) edge[commutative diagrams/equal] (Y.north)
(F.south) edge[commutative diagrams/equal]  (A.north);
\path[->,font=\scriptsize]
(A.east) edge node[above] {$u_X \cdot {\blue \epsilon_B}$} (B.west)
(F.east) edge node[above] {$1_B  \ox u_X$} (E.west)
(C.south) edge node[descr] {$\lambda$} (B.north)
(D.south) edge node[descr] {$m$} (C.north)
(F.south east) edge node[descr] {$e$} (C.north west)
(E.south) edge node[descr] {$e \ox \kappa$} (D.north)
;
\end{tikzpicture} 
\qquad
\begin{tikzpicture}[descr/.style={fill=white},baseline=(A.base)] 
\node (A) at (0,0) {$B \ox X $};
\node (B) at (6,0) {$I$};
\node (C) at (6,2) {$X$};
\node (D) at (4,2) {$A$};
\node (E) at (2,2) {$A^2 $};
\node (F) at (0,2) {$B\ox X$};
\draw[->] (F.north) to[bend left=20]node[above,scale=0.8] {${\blue \act }$} (C.north west);
\draw[commutative diagrams/.cd, ,font=\scriptsize]
(F.south) edge[commutative diagrams/equal]  (A.north);
\path[->,font=\scriptsize]
(A.east) edge node[above] {${\blue \epsilon_B} \ox {\blue \epsilon_X}$} (B.west)
(F.east) edge node[above] {$e  \ox \kappa$} (E.west)
(E.east) edge node[above] {$m$} (D.west)
(D.east) edge node[above] {$\lambda$} (C.west)
(D.south east) edge node[descr] {${\blue \epsilon_A}$} (B.north west)
(C.south) edge node[descr] {${\blue \epsilon_X}$} (B.north)
;
\end{tikzpicture}
\end{center}
\begin{center}
\begin{tikzpicture}[descr/.style={fill=white},yscale=0.8] 
\node (A) at (-1,0) {$(B \ox X)^2$};
\node (B) at (2.5,0) {$A^4$};
\node (C) at (5,0) {$A^2$};
\node (D) at (7.5,0) {$X^2.$};
\node (E) at (7.5,5) {$X$};
\node (F) at (5,5) {$A $};
\node (G) at (2.5,5) {$A^2$};
\node (H) at (-1,5) {$ B \ox X$};
\node (I) at (-1,2.5) {$B^2 \ox X^2$};
\node (J) at (2.5,2.5) {$A^4$};
\node at (3.75,2.5) {$\eqref{m et delta}$};
\draw[->] (H.north) to[bend left=15]node[above,scale=0.8] {${\blue \act }$} (E.north west);
cla\draw[->] (A.south) to[bend right=15]node[below,scale=0.8] {${\blue \act  \ox \act}$} (D.south west);
\path[->,font=\scriptsize]
(A.east) edge node[above] {$e  \ox k \ox e  \ox k $} (B.west)
(B.east) edge node[above] {$m  \ox m$} (C.west)
(C.east) edge node[above] {$\lambda  \ox \lambda$} (D.west)
(H.east) edge node[above] {$e  \ox k$} (G.west)
(G.east) edge node[above] {$m  $} (F.west)
(F.east) edge node[above] {$\lambda$} (E.west)
(I.east) edge node[above] {$e \ox e \ox k \ox k $} (J.west)
(E.south) edge node[descr] {$\Delta$} (D.north)
(F.south) edge node[descr] {$\Delta$} (C.north)
(H.south) edge node[descr] {$\Delta \ox \Delta$} (I.north)
(G.south) edge node[descr] {$\Delta \ox \Delta$} (J.north)
(J.south) edge node[descr] {$1_A \ox {\blue \sigma_{A,A}} \ox 1_A$} (B.north)
(I.south) edge node[descr] {$1_B \ox {\blue \sigma_{B,X}} \ox 1_X$} (A.north);
\end{tikzpicture} 
% textidote: ignore end
\end{center}
Moreover, the condition \eqref{coco} for the particular action $\act = \lambda \cdot m \cdot (e \ox \kappa)$ is exactly the condition (5) in the definition of split extension of bialgebras (Definition \ref{definition split extension}).
\qed
\end{proof}

Consider the diagram 
\begin{equation}\label{diag iso}
% textidote: ignore begin
\begin{tikzpicture}[descr/.style={fill=white},baseline=(current  bounding  box.center)] 
\node (A) at (0,0) {$A$};
\node (B) at (2.5,0) {$B$};
\node (C) at (-2.5,0) {$X$};
\node (A') at (0,-2) {$X \rtimes B$};
\node (B') at (2.5,-2) {$B$};
\node (C') at (-2.5,-2) {$X$};
\draw [double distance = 3pt] (B.south) -- (B'.north);
\draw [double distance = 3pt] (C.south) -- (C'.north);
\path[dashed,->,font=\scriptsize]
([yshift=2pt]A'.west) edge node[above] {$\pi_1$} ([yshift=2pt]C'.east)
([yshift=2pt]A.west) edge node[above] {$\lambda$} ([yshift=2pt]C.east);
\path[->,font=\scriptsize]
([xshift=-3pt]A.south) edge node[left] {$ \psi$} ([xshift=-3pt]A'.north)
([xshift=3pt]A'.north) edge node[right] {$ \phi$} ([xshift=3pt]A.south)
([yshift=-4pt]C'.east) edge node[below] {$i_1$} ([yshift=-4pt]A'.west)
([yshift=-4pt]A'.east) edge node[below] {$\pi_2$} ([yshift=-4pt]B'.west)
([yshift=2pt]B'.west) edge node[above] {$i_2$} ([yshift=2pt]A'.east)
([yshift=-4pt]C.east) edge node[below] {$\kappa$} ([yshift=-4pt]A.west)
([yshift=-4pt]A.east) edge node[below] {$\alpha$} ([yshift=-4pt]B.west)
([yshift=2pt]B.west) edge node[above] {$e$} ([yshift=2pt]A.east);
\end{tikzpicture}
% textidote: ignore end 
\end{equation}
where the two rows are split extensions, the top row by definition and the bottom one by Lemma \ref{obs semi direct bialg}, where the action of bialgebras is given by Proposition \ref{Action bialgebra}. The morphisms $\phi$ and $\psi$ {\ro in $\C$} are defined by 
\[ \psi = (\lambda \otimes \alpha) \cdot \Delta, \]
 and
\[\phi = m \cdot (\kappa \otimes e) .\]
In the following lemmas, we prove step by step that $\psi$ and $\phi$ are isomorphisms of split extensions. First, we prove that they are inverse to each other {\ro in $\C$}.
\begin{lemma}\label{inverse to each other}
The {\blue morphisms} $\phi$ and $\psi$ {\blue in $\C$} are inverse to each other. 
\end{lemma}
\begin{proof}
We prove by means of the two following diagrams that $\psi \cdot \phi = 1_X \ox 1_B$ and $\phi \cdot \psi = 1_A$  by using the properties of the split extensions. 
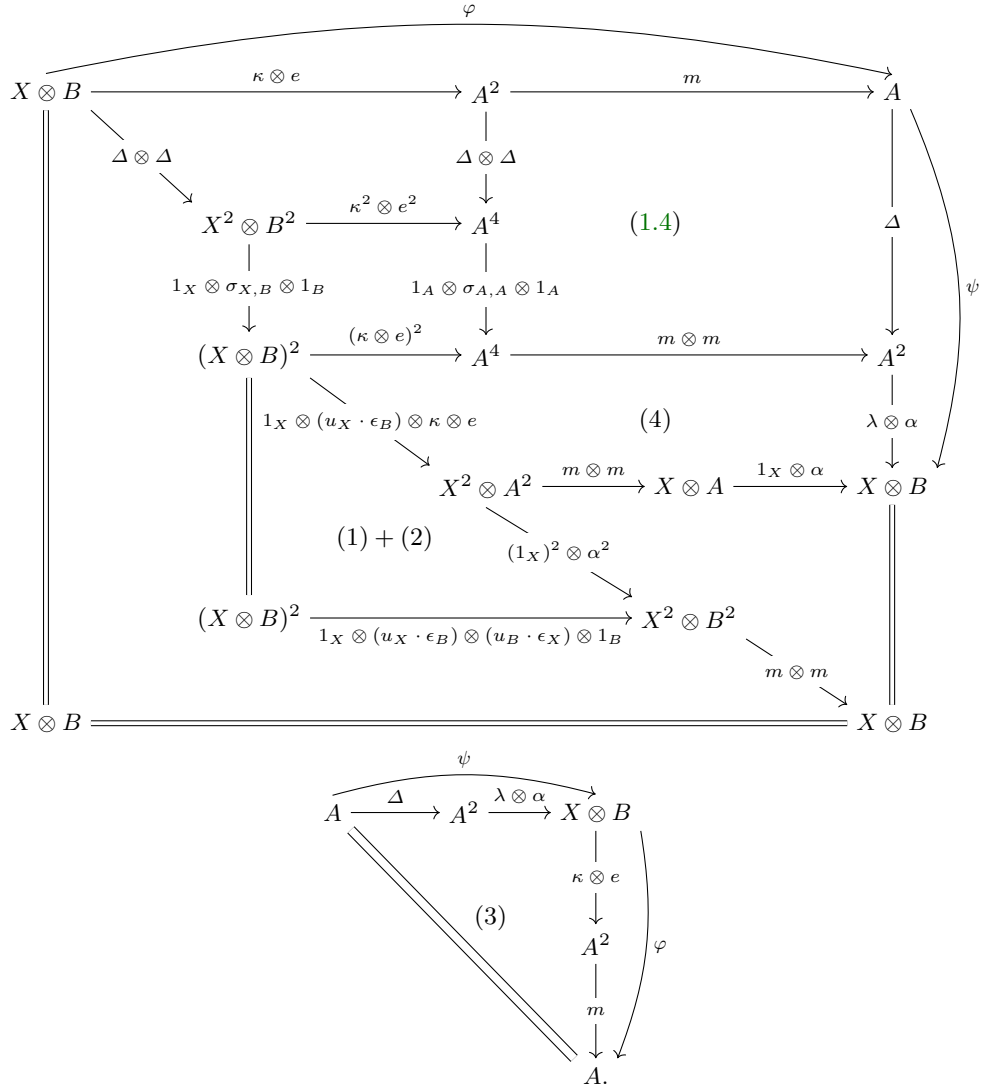
\begin{center}
%\begin{figure}[b]
%\centering
%% textidote: ignore begin
\begin{tikzpicture}[descr/.style={fill=white},yscale=0.7,xscale=0.9] 
\node (A) at (0,0.5) {$ X \ox B $};
\node (B) at (12.5,0.5) {$X \ox B$};
\node (C) at (12.5,5) {$X \ox B$};
\node (D) at (12.5,7.5) {$A^2$};
\node (E) at (12.5,12.5) {$A$};
\node (F) at (6.5,12.5) {$A^2$};
\node (G) at (0,12.5) {$ X \ox B$};
\node (H) at (3,10) {$ X^2 \ox B^2$};
\node (I) at (6.5,10) {$A^4$};
\node (J) at (3,7.5) {$(X \ox B)^2$};
\node (K) at (6.5,7.5) {$A^4$};
\node (L) at (6.5,5) {$X^2 \ox A^2$};
\node (M) at (9.5,5) {$X \ox A$};
\node (N) at (3,2.5) {$(X \ox B)^2$};
\node (O) at (9.5,2.5) {$X^2 \ox B^2$};
\node (X) at (5,4) {$(1)+(2)$};
\node (Y) at (9,6.25) {$(4)$};
\node (Y) at (9,10) {$\eqref{m et delta}$};
%\draw[equal] (J.south) -- (Y.north);
\draw[commutative diagrams/.cd, ,font=\scriptsize]
(G.south) edge[commutative diagrams/equal]  (A.north)
(J.south) edge[commutative diagrams/equal]  (N.north)
(C.south) edge[commutative diagrams/equal] (B.north)
(A.east) edge[commutative diagrams/equal]  (B.west);
%\draw[=] (H.south west) to[bend right=90]node[left] {$1_{G_1} $} (D.south west);
\draw[->] (E.south east) to[bend left=15]node[right,scale=0.8] {${\blue \psi }$} (C.north east);
\draw[->] (G.north) to[bend left=15]node[above,scale=0.8] {${\blue \phi }$} (E.north);
\path[->,font=\scriptsize]
(G.east) edge node[above] {$\kappa \ox e $} (F.west)
(F.east) edge node[above] {$m$} (E.west)
(H.east) edge node[above] {$\kappa^2 \ox e^2    $} (I.west)
(J.east) edge node[above] {$(\kappa \ox e)^2$} (K.west)
(K.east) edge node[above] {$m \ox m$} (D.west)
(L.east) edge node[above] {$m \ox m$} (M.west)
(M.east) edge node[above] {$1_X \ox \alpha$} (C.west)
(N.east) edge node[below] {$1_X \ox (u_X \cdot {\blue \epsilon_B}) \ox (u_B \cdot {\blue \epsilon_X}) \ox 1_B$} (O.west)
(D.south) edge node[descr] {$\lambda \ox \alpha$} (C.north)
(E.south) edge node[descr] {$\Delta$} (D.north)
(F.south) edge node[descr] {$\Delta \ox \Delta$} (I.north)
(I.south) edge node[descr] {$1_A \ox {\blue \sigma_{A,A}} \ox 1_A$} (K.north)
(H.south) edge node[descr] {$1_X \ox {\blue \sigma_{X,B}} \ox 1_B$} (J.north)
(G.south east) edge node[descr] {$\Delta  \ox \Delta $} (H.north west)
(J.south east) edge node[descr] {$1_X \ox (u_X \cdot {\blue \epsilon_B}) \ox \kappa \ox e$} (L.north west)
(L.south) edge node[descr] {$(1_X)^2 \ox \alpha^2 $} (O.north west)
(O.south east) edge node[descr] {$m \ox m$} (B.north west)
;
\end{tikzpicture} 
\quad
\begin{tikzpicture}[descr/.style={fill=white},scale=0.7] 
\node (A) at (0,0) {$A$};
\node (B) at (2.5,0) {$A^2$};
\node (C) at (5,0) {$X \ox B$};
\node (D) at (5,-2.5) {$A^2$};
\node (E) at (5,-5) {$A.$};
\node at (3,-2) {$(3)$};
\draw[->] (C.south east) to[bend left=15]node[right,scale=0.8] {${\blue \phi} $} (E.north east);
\draw[->] (A.north) to[bend left=15]node[above,scale=0.8] {${\blue \psi }$} (C.north);
\draw [double distance = 3pt] (A.south east) -- (E.north west);
\path[->,font=\scriptsize]
(D.south) edge node[descr] {$m$} (E.north)
(C.south) edge node[descr] {$\kappa \ox e$} (D.north)
(A.east) edge node[above] {$\Delta$} (B.west)
(B.east) edge node[above] {$\lambda \ox \alpha$} (C.west)
;
\end{tikzpicture} 
% textidote: ignore end
%\caption{$\psi$ and $\phi$ inverse to each other}
%\label{psi-phi}
%\end{figure}
\end{center}
\qed
\end{proof}
In order to prove that $\psi$ and $\phi$ of \eqref{diag iso} are morphisms of bialgebras, we need the following technical Lemma. Notice that this Lemma will also be convenient to express the action $\act {\blue \coloneqq} \lambda \cdot m \cdot (e \ox \kappa)$ in the particular case of Hopf algebras ({\blue s}ee Remark \ref{action hopf reexpressed}).
\begin{lemma}\label{tech lemma}
Given a split extension 
% textidote: ignore begin
\begin{tikzpicture}[descr/.style={fill=white},baseline=(A.base)] 
\node (A) at (0,0) {$A$};
\node (B) at (2.5,0) {$B$};
\node (C) at (-2.5,0) {$X$};
\path[dashed,->,font=\scriptsize]
([yshift=2pt]A.west) edge node[above] {$\lambda$} ([yshift=2pt]C.east);
\path[->,font=\scriptsize]
([yshift=-4pt]C.east) edge node[below] {$\kappa$} ([yshift=-4pt]A.west)
([yshift=-4pt]A.east) edge node[below] {$\alpha$} ([yshift=-4pt]B.west)
([yshift=2pt]B.west) edge node[above] {$e$} ([yshift=2pt]A.east);
\end{tikzpicture} 
% textidote: ignore end
of bialgebras, we have
% textidote: ignore begin
 \begin{equation}\label{lemma1.7}
 m \cdot (e \otimes \kappa) = m \cdot (\kappa \otimes e) \cdot (\act \otimes 1_B) \cdot (1_B \otimes {\blue \sigma_{B,X}}) \cdot (\Delta \otimes 1_X),
 \end{equation}
% textidote: ignore end
 where $\act = \lambda \cdot m \cdot (e \ox \kappa)$ .
\end{lemma} 
\begin{proof} 
The equality of the lemma is proven thanks to the commutativity of the diagram below, where we use that $\alpha$, $\kappa$ and $e$ are morphisms of bialgebras. 
\begin{center}
%\begin{figure}[H]
%\centering
\begin{tikzpicture}[descr/.style={fill=white},yscale=0.8,xscale=1.3] 
\node (A) at (0,0) {$ A \ox B $};
\node (B) at (3,0) {$X \ox B$};
\node (C) at (6,0) {$A^2$};
\node (D) at (9,0) {$A$};
\node (E) at (9,2.5) {$A^2$};
\node (F) at (9,5) {$X \ox B$};
\node (G) at (9,7.5) {$A^2$};
\node (H) at (9,10) {$A$};
\node (I) at (6,10) {$A^2$};
\node (J) at (3,10) {$B \ox X$};
\node (K) at (0,10) {$B \ox X$};
\node (L) at (0,7.5) {$B^2 \ox X$};
\node (M) at (0,5) {$B \ox X \ox B$};
\node (N) at (0,2.5) {$A^2 \ox B$};
\node (O) at (3,7.5) {$B \ox X \ox B \ox X$};
\node (P) at (6,7.5) {$A^4$};
\node (Q) at (3,5) {$A^2 \ox B \ox X$};
\node (R) at (6,5) {$A \ox B^2$};
\node at (10,6) {$(3)$};
\node at (4.5,2.5) {$\eqref{unital multiplication}+\eqref{counit comultiplication}$};
\node at (4.5,6.75) {$(1) + (2)$};
\node at (8,9) {$\eqref{m et delta}$};
\draw[double distance = 2pt] (H.south east) to[bend right=330] (D.north east);
%%\draw[equal] (J.south) -- (Y.north);
\draw[commutative diagrams/.cd, ,font=\scriptsize]
(K.east) edge[commutative diagrams/equal]  (J.west);
%%\draw[=] (H.south west) to[bend right=90]node[left] {$1_{G_1} $} (D.south west);
\path[->,font=\scriptsize]
(A.east) edge node[above] {$\lambda \ox 1_B $} (B.west)
(B.east) edge node[above] {$\kappa \ox e$} (C.west)
(C.east) edge node[above] {$m   $} (D.west)
(J.east) edge node[above] {$e \ox \kappa$} (I.west)
(I.east) edge node[above] {$m $} (H.west)
(O.east) edge node[above] {$(e \ox \kappa)^2$} (P.west)
(P.east) edge node[above] {$m\ox m$} (G.west)
(Q.east) edge node[above] {$m \ox 1_B \ox (u_B \cdot {\blue \epsilon_X}) $} (R.west)
(R.east) edge node[above] {$\lambda \ox m$} (F.west)
(E.south) edge node[descr] {$m$} (D.north)
(F.south) edge node[descr] {$\kappa \ox e$} (E.north)
(G.south) edge node[descr] {$\lambda \ox \alpha$} (F.north)
(H.south) edge node[descr] {$\Delta$} (G.north)
(K.south) edge node[left] {$\Delta \ox 1_X$} (L.north)
(L.south) edge node[descr] {$1_B \ox {\blue \sigma_{B,X}}$} (M.north)
(M.south) edge node[descr] {$e \ox \kappa \ox 1_B $} (N.north)
(N.south) edge node[descr] {$m \ox 1_B $} (A.north)
(J.south) edge node[descr] {$(1_B \ox {\blue \sigma_{B,X}} \ox 1_X) \cdot(\Delta \ox \Delta) \; \; \; \; \; \; \; \; \; \; \; \; \; \; \; \; \; \; \; $} (O.north)
(I.south) edge node[descr] {$(1_A \ox {\blue \sigma_{A,A}} \ox 1_A)\cdot(\Delta \ox \Delta)$} (P.north)
(O.south) edge node[descr] {$e \ox \kappa \ox 1_B\ox 1_X$} (Q.north)
(P.south) edge node[descr] {$m \ox \alpha \ox \alpha$} (R.north)
%(G.south east) edge node[descr] {$\Delta  \ox \Delta $} (H.north west)
%(J.south east) edge node[descr] {$1_X \ox \epsilon \ox \kappa \ox e$} (L.north west)
%(L.south east) edge node[descr] {$1_X \ox 1_X \ox \alpha \ox \alpha$} (O.north west)
%(O.south east) edge node[descr] {$m \ox m$} (B.north west)
;
\end{tikzpicture} 
% textidote: ignore end
%\caption{Proof of \eqref{lemma1.7}}
%\label{proof lemma 1.7}
\end{center}
\qed
\end{proof}
\begin{lemma}
$\phi {\blue \coloneqq} m \cdot (\kappa \ox e) \colon X \ox B \rightarrow A$ is a morphism of bialgebras.
\end{lemma}
\begin{proof}
%First we establish two diagrams denoted by $(A)$ and $(B)$. We will need them to conclude that $\phi$ is a morphism of algebras.
We use the partial associativity of the split extensions to check that Figure \ref{DiagramA'} commutes.
\begin{figure}[b]
\centering
% textidote: ignore begin
\begin{tikzpicture}[descr/.style={fill=white},baseline=(A.base),yscale=0.7,xscale=0.9] 
\node (A) at (-2,0) {$ X \ox B \ox X \ox B $};
\node (B) at (2.5,0) {$A^4$};
\node (C) at (5,0) {$A^3$};
\node (X) at (7.5,0) {$A^2$};
\node (D) at (10,0) {$A$};
\node (I) at (-2,5) {$ X \ox B \ox X \ox B $};
\node (H) at (3,5) {$A^4$};
\node (G) at (6,5) {$A^3$};
\node (F) at (10,5) {$A^2$};
\node (E) at (10,2.5) {$A$};
\node (J) at (-2,2.5) {$ X \ox B \ox X \ox B $};
\node (K) at (2.5,2.5) {$A^4$};
\node (L) at (5,2.5) {$A^3$};
\node (M) at (7.5,2.5) {$A^2$};
\node (Y) at (5,3.75) {$(6)$};
\node (Z) at (5,1.25) {$(8)$};
\draw[commutative diagrams/.cd, ,font=\scriptsize]
(I.south) edge[commutative diagrams/equal]  (J.north)
(J.south) edge[commutative diagrams/equal]  (A.north)
(E.south) edge[commutative diagrams/equal]  (D.north)
;
\path[->,font=\scriptsize]
(A.east) edge node[above] {$(\kappa \ox e)^2 $} (B.west)
(B.east) edge node[above] {$m \ox 1_A \ox 1_A$} (C.west)
(C.east) edge node[above] {$1_A \ox m$} (X.west)
(X.east) edge node[above] {$m$} (D.west)
(I.east) edge node[above] {$(\kappa \ox e)^2 $} (H.west)
(H.east) edge node[above] {$1_A \ox m \ox 1_A$} (G.west)
(G.east) edge node[above] {$m \ox 1_A$} (F.west)
(J.east) edge node[above] {$(\kappa \ox e)^2 $} (K.west)
(K.east) edge node[above] {$m \ox 1_A \ox 1_A$} (L.west)
(L.east) edge node[above] {$ m\ox 1_A $} (M.west)
(M.east) edge node[above] {$m$} (E.west)
(F.south) edge node[descr] {$m$} (E.north)
;
\end{tikzpicture} 
% textidote: ignore end
\caption{First sub-diagram in the proof that $\phi$ is a morphism of algebras}
\label{DiagramA'}
\end{figure}
Moreover, Figure \ref{DiagramB'} commutes since $e$ and $\kappa$ are (bi)algebra morphisms. \begin{figure}
\centering
% textidote: ignore begin
\begin{tikzpicture}[descr/.style={fill=white},baseline=(A.base),yscale=0.8] 
\node (A) at (0,0) {$ X \ox B$};
\node (B) at (3,0) {$A^2$};
\node (C) at (9,0) {$A$};
\node (D) at (9,5) {$A^2$};
\node (E) at (9,10) {$A^2 $};
\node (F) at (6,10) {$A^3$};
\node (G) at (3,10) {$A^4$};
\node (H) at (0,10) {$X^2 \ox B^2$};
\node (I) at (0,7.5) {$X^2 \ox B^2$};
\node (J) at (0,5) {$X \ox B^2$};
\node (K) at (0,2.5) {$X \ox B^2$};
\node (L) at (3,7.5) {$A^4$};
\node (M) at (6,7.5) {$A^3$};
\node (N) at (3,5) {$A^3$};
\node (O) at (3,2.5) {$A^3$};
\node (Y) at (6,2.5) {$(6)$};
\node (Z) at (3,8.75) {$(7)$};
\draw[commutative diagrams/.cd, ,font=\scriptsize]
(H.south) edge[commutative diagrams/equal]  (I.north)
(J.south) edge[commutative diagrams/equal]  (K.north)
(E.south) edge[commutative diagrams/equal]  (D.north)
;
\draw[->] (H.south west) to[bend left=345] node[descr,scale=0.8] {${\blue m \ox m}$} (A.west);
\path[->,font=\scriptsize]
(A.east) edge node[above] {$\kappa \ox e$} (B.west)
(B.east) edge node[above] {$m$} (C.west)
(K.east) edge node[above] {$\kappa \ox e^2$} (O.west)
(J.east) edge node[above] {$\kappa \ox e^2$} (N.west)
(I.east) edge node[above] {$\kappa^2  \ox e^2$} (L.west)
(H.east) edge node[above] {$\kappa^2 \ox e^2$} (G.west)
(G.east) edge node[above] {$1_A \ox m \ox 1_A$} (F.west)
(F.east) edge node[above] {$m \ox 1_A$} (E.west)
(N.east) edge node[above] {$m \ox 1_A$} (D.west)
(L.east) edge node[above] {$m \ox 1_A \ox 1_A$} (M.west)
(D.south) edge node[descr] {$m$} (C.north)
(K.south) edge node[descr] {$1_X \ox m$} (A.north)
(I.south) edge node[descr] {$m \ox 1_B \ox 1_B$} (J.north)
(L.south) edge node[descr] {$m \ox 1_A \ox 1_A$} (N.north)
(O.south) edge node[descr] {$1_A \ox m$} (B.north)
(M.north east) edge node[descr] {$m \ox 1_A$} (E.south west)
;
\end{tikzpicture}
% textidote: ignore end
\caption{Second sub-diagram in the proof that $\phi$ is a morphism of algebras}
\label{DiagramB'}
\end{figure}
Thanks to Figure \ref{DiagramA'} and \ref{DiagramB'} and Lemma \ref{tech lemma} we have Figure \ref{phi-alg} showing that $\phi {\blue \coloneqq} m \cdot (\kappa \ox e)$ is a morphism of algebras.
\begin{figure}
\centering
% textidote: ignore begin
\begin{tikzpicture}[descr/.style={fill=white},xscale=1.2, yscale=0.6] 
\node (A) at (0,0) {$ A^4 $};
\node (B) at (6,0) {$A^2$};
\node (C) at (12,0) {$A$};
\node (D) at (12,3) {$A^2$};
\node (E) at (12,6) {$X \ox B$};
\node (F) at (12,9) {$X^2 \ox B^2$};
\node (G) at (8,9) {$X \ox B \ox X \ox B^2$};
\node (H) at (4,9) {$X \ox B^2 \ox X \ox B$};
\node (I) at (0,9) {$(X \ox B)^2$};
\node (J) at (4,6) {$A^4$};
\node (K) at (6,3) {$A^3$};
\node (L) at (8,6) {$A^4$};
\node (M) at (8,3) {$A^2$};
\node (X) at (6,7) {$\eqref{lemma1.7}$};
\node (Y) at (3,3) {$Fig. \; \ref{DiagramA'}$};
\node (Z) at (9.5,5) {$Fig. \; \ref{DiagramB'}$};
%\draw[commutative diagrams/.cd, ,font=\scriptsize]
%(K.east) edge[commutative diagrams/equal]  (J.west);
%\draw[=] (H.south west) to[bend right=90]node[left] {$1_{G_1} $} (D.south west);
\path[->,font=\scriptsize]
(A.east) edge node[above] {$m \ox m $} (B.west)
(B.east) edge node[above] {$m $} (C.west)
(I.east) edge node[above] {$1_X \ox \Delta \ox 1_X \ox 1_B $} (H.west)
(H.east) edge node[above] {$1_X \ox 1_B \ox {\blue \sigma_{B,X}} \ox 1_B $} (G.west)
(G.east) edge node[above] {$1_X \ox \act \ox 1_B \ox 1_B $} (F.west)
(K.east) edge node[above] {$m \ox 1_A $} (M.west)
(E.south) edge node[descr] {$\kappa \ox e$} (D.north)
(F.south) edge node[descr] {$m \ox m $} (E.north)
(D.south) edge node[descr] {$m$} (C.north)
(I.south) edge node[descr] {$(\kappa \ox e)^2$} (A.north)
(I.south east) edge node[descr] {$(\kappa \ox e)^2 $} (J.west)
(J.south east) edge node[descr] {$1_A \ox m \ox 1_A$} (K.north west)
(M.south east) edge node[descr] {$m$} (C.north west)
(F.south west) edge node[descr] {$\kappa^2 \ox e^2$} (L.east)
(L.south west) edge node[descr] {$1_A \ox m \ox 1_A$} (K.north east)
;
\end{tikzpicture} 
% textidote: ignore end
\caption{$\phi$ is a morphism of algebras}
\label{phi-alg}
\end{figure}
%$$ m \cdot (\kappa \otimes e) \cdot (u_X \otimes u_B) = u_A$$

Finally, $\phi$ is a morphism of bialgebras since it is also a morphism of coalgebras:
%\begin{align*}
%\epsilon \cdot  m \cdot (\kappa \otimes e) &=  (\epsilon \otimes \epsilon) \cdot (\kappa \otimes  e)\\
%&= \epsilon \ox \epsilon
%\end{align*}
%\begin{align*}
%\Delta \cdot  m \cdot (\kappa \otimes e) &= (m \ox m) \cdot (1_A \otimes \sigma \otimes 1_A)\cdot (\Delta \otimes \Delta) \cdot (\kappa \otimes e)\\
%&= (m \ox m) \cdot (\kappa \otimes e \ox \kappa \otimes e) \cdot (1_X \otimes \sigma \otimes 1_B)\cdot (\Delta \otimes \Delta) \\ 
%&= (\phi \ox \phi) \cdot \Delta_{X \rtimes B}
%\end{align*}
 \begin{equation}
 % textidote: ignore begin
\begin{tikzpicture}[descr/.style={fill=white},baseline=(current  bounding  box.center),scale=0.8] 
\node (A) at (-2,0) {$ X \ox B \ox X \ox B $};
\node (B) at (2.5,0) {$A^4$};
\node (C) at (5,0) {$A^2$};
\node (D) at (5,5) {$A$};
\node (E) at (2.5,5) {$A^2$};
\node (F) at (-2,5) {$ X \ox B$};
\node (G) at (-2,2.5) {$ X\ox X \ox B  \ox B $};
\node (H) at (2.5,2.5) {$A^4$};
\node at (3.75,2.5) {$\eqref{m et delta}$};
\path[->,font=\scriptsize]
(A.east) edge node[above] {$(\kappa \ox e)^2$} (B.west)
(B.east) edge node[above] {$m \ox m$} (C.west)
(F.east) edge node[above] {$\kappa \ox e$} (E.west)
(E.east) edge node[above] {$ m $} (D.west)
(G.east) edge node[above] {$\kappa^2 \ox e^2$} (H.west)
(F.south) edge node[descr] {$\Delta \ox \Delta$} (G.north)
(G.south) edge node[descr] {$1_X \ox {\blue \sigma_{X,B}} \ox 1_B $} (A.north)
(E.south) edge node[descr] {$\Delta \ox \Delta$} (H.north)
(H.south) edge node[descr] {$1_A \ox {\blue \sigma_{A,A}} \ox 1_A $} (B.north)
(D.south) edge node[descr] {$\Delta$} (C.north)
;
\end{tikzpicture}
% textidote: ignore end
 \end{equation}
 This diagram commutes since $e$ and $\kappa$ are coalgebras morphisms.
\qed
\end{proof}
From this Lemma and Lemma \ref{inverse to each other}, it is straightforward that $\psi$ is also a morphism of bialgebras, and it brings us to the following useful proposition.
\begin{proposition}\label{canonical}
The diagram \eqref{diag iso} 
% textidote: ignore begin
\begin{center}
\begin{tikzpicture}[descr/.style={fill=white},baseline=(A'.base)] 
\node (A) at (0,0) {$A$};
\node (B) at (2.5,0) {$B$};
\node (C) at (-2.5,0) {$X$};
\node (A') at (0,-2) {$X \rtimes B$};
\node (B') at (2.5,-2) {$B$};
\node (C') at (-2.5,-2) {$X$};
\draw [double distance = 3pt] (B.south) -- (B'.north);
\draw [double distance = 3pt] (C.south) -- (C'.north);
\path[dashed,->,font=\scriptsize]
([yshift=2pt]A'.west) edge node[above] {$\pi_1$} ([yshift=2pt]C'.east)
([yshift=2pt]A.west) edge node[above] {$\lambda$} ([yshift=2pt]C.east);
\path[->,font=\scriptsize]
([xshift=-3pt]A.south) edge node[left] {$ \overset{\psi}{\cong}$} ([xshift=-3pt]A'.north)
([xshift=3pt]A'.north) edge node[right] {$\overset{\phi}{\cong}$} ([xshift=3pt]A.south)
([yshift=-4pt]C'.east) edge node[below] {$i_1$} ([yshift=-4pt]A'.west)
([yshift=-4pt]A'.east) edge node[below] {$\pi_2$} ([yshift=-4pt]B'.west)
([yshift=2pt]B'.west) edge node[above] {$i_2$} ([yshift=2pt]A'.east)
([yshift=-4pt]C.east) edge node[below] {$\kappa$} ([yshift=-4pt]A.west)
([yshift=-4pt]A.east) edge node[below] {$\alpha$} ([yshift=-4pt]B.west)
([yshift=2pt]B.west) edge node[above] {$e$} ([yshift=2pt]A.east);
\end{tikzpicture} 
% textidote: ignore end 
\end{center}
{\blue in the category of bialgebras in $\C$ commutes}. Accordingly, $(1_B, 1_X, \psi)$ and $(1_B, 1_X, \phi)$ are {\green isomorphisms of split extensions of bialgebras.}
\end{proposition}

\begin{proof}
{\green Thanks to Proposition \ref{comm}, it is enough to check that the four following diagrams of \eqref{diag iso} commute to conclude that $\phi$ and $\psi$ are morphisms of split extensions of bialgebras.} We recall that $\psi = (\lambda \ox \alpha) \cdot \Delta$ and $\phi = m \cdot (\kappa \ox e)$.
\begin{center}
% textidote: ignore begin
\begin{tikzpicture}[descr/.style={fill=white},baseline=(A.base),scale=0.8] 
\node (A) at (0,0) {$ B  $};
\node (B) at (2.5,0) {$A$};
\node (C) at (2.5,2) {$A^2$};
\node (D) at (2.5,4) {$X \ox B$};
\node (E) at (0,4) {$B$};
%\draw[->] (D.south east) to[bend right=340] node[descr] {$\phi$} (B.east);
\draw[commutative diagrams/.cd, ,font=\scriptsize]
(E.south) edge[commutative diagrams/equal]  (A.north);
\path[->,font=\scriptsize]
(A.east) edge node[above] {$e$} (B.west)
(E.east) edge node[above] {$ u_X \ox 1_B $} (D.west)
(D.south) edge node[descr] {$\kappa \ox e$} (C.north)
(C.south) edge node[descr] {$m $} (B.north)
(E.south east) edge node[descr] {$u_A \ox e$} (C.west)
;
\end{tikzpicture} 
\qquad
\begin{tikzpicture}[descr/.style={fill=white},baseline=(A.base),scale=0.8] 
\node (A) at (0,0) {$ X  $};
\node (B) at (2.5,0) {$A$};
\node (C) at (2.5,2) {$A^2$};
\node (D) at (2.5,4) {$X \ox B$};
\node (E) at (0,4) {$X$};
%\draw[->] (D.south east) to[bend right=340] node[descr] {$\phi$} (B.east);
\draw[commutative diagrams/.cd, ,font=\scriptsize]
(E.south) edge[commutative diagrams/equal]  (A.north);
\path[->,font=\scriptsize]
(A.east) edge node[above] {$\kappa$} (B.west)
(E.east) edge node[above] {$ 1_X \ox u_B $} (D.west)
(D.south) edge node[descr] {$\kappa \ox e$} (C.north)
(C.south) edge node[descr] {$m $} (B.north)
(E.south east) edge node[descr] {$\kappa \ox u_A$} (C.west)
;
\end{tikzpicture}
\qquad
%\begin{tikzpicture}[descr/.style={fill=white},baseline=(B.base),scale=0.8] 
%\node (B) at (2.5,0) {$X$};
%\node (C) at (2.5,2) {$A$};
%\node (D) at (2.5,4) {$A^2$};
%\node (E) at (0,4) {$X \ox B$};
%\node (X) at (1.25,3) {$(4)$};
%%\draw[->] (D.south east) to[bend right=340] node[descr] {$\psi$} (B.east);
%\path[->,font=\scriptsize]
%(E.east) edge node[above] {$ \kappa \ox e $} (D.west)
%(D.south) edge node[descr] {$m$} (C.north)
%(C.south) edge node[descr] {$\lambda $} (B.north)
%(E.south) edge node[descr] {$1_X \ox \epsilon$} (B.north west)
%;
%\end{tikzpicture} 
%\begin{tikzpicture}[descr/.style={fill=white},baseline=(A.base),scale=0.7] 
%\node (A) at (0,0) {$ X \ox B  $};
%\node (B) at (5,0) {$B$};
%\node (C) at (5,2) {$B$};
%\node (D) at (5,4) {$A$};
%\node (E) at (2.5,4) {$A^2$};
%\node (F) at (0,4) {$ X \ox B  $};
%\node (G) at (2.5,2) {$B^2$};
%\draw[commutative diagrams/.cd, ,font=\scriptsize]
%(F.south) edge[commutative diagrams/equal]  (A.north)
%(C.south) edge[commutative diagrams/equal]  (B.north);
%\path[->,font=\scriptsize]
%(A.east) edge node[above] {$\epsilon \ox 1_B$} (B.west)
%(E.east) edge node[above] {$ m $} (D.west)
%(G.east) edge node[above] {$ m $} (C.west)
%(F.east) edge node[above] {$ \kappa \ox e $} (E.west)
%(D.south) edge node[descr] {$\alpha$} (C.north)
%(E.south) edge node[descr] {$\alpha \ox \alpha $} (G.north)
%;
%\end{tikzpicture} 
%\end{center} 
%\begin{center}
\begin{tikzpicture}[descr/.style={fill=white},scale=0.7,xscale=0.8] 
\node (A) at (0,0) {$ X $};
\node (B) at (5,0) {$X \ox B $};
\node (D) at (5,4) {$A^2$};
\node (E) at (2.5,4) {$A$};
\node (F) at (0,4) {$ X  $};
\node (G) at (2.5,2) {$X^2$};
\draw[commutative diagrams/.cd, ,font=\scriptsize]
(F.south) edge[commutative diagrams/equal]  (A.north);
\path[->,font=\scriptsize]
(A.east) edge node[above] {$1_X \ox u_B$} (B.west)
(E.east) edge node[above] {$ \Delta $} (D.west)
(G.north east) edge node[descr] {$ \kappa \ox \kappa $} (D.south west)
(F.east) edge node[above] {$ \kappa$} (E.west)
(D.south) edge node[descr] {$\lambda \ox \alpha$} (B.north)
(F.south east) edge node[descr] {$\Delta $} (G.north west)
(G) edge node[descr,xshift=-15pt] {$1_X \ox (u_B \cdot {\blue \epsilon_X}) $} (B.north west)
;
\end{tikzpicture}
\qquad
\begin{tikzpicture}[descr/.style={fill=white},
xscale=0.8,scale=0.7] 
\node (A) at (0,0) {$ B $};
\node (B) at (5,0) {$X \ox B $};
\node (D) at (5,4) {$A^2$};
\node (E) at (2.5,4) {$A$};
\node (F) at (0,4) {$ B  $};
\node (G) at (2.5,2) {$B^2$};
\draw[commutative diagrams/.cd, ,font=\scriptsize]
(F.south) edge[commutative diagrams/equal]  (A.north);
\path[->,font=\scriptsize]
(A.east) edge node[above] {$u_X \ox 1_B$} (B.west)
(E.east) edge node[above] {$ \Delta $} (D.west)
(G.north east) edge node[descr] {$ e \ox e $} (D.south west)
(F.east) edge node[above] {$ e$} (E.west)
(D.south) edge node[descr] {$\lambda \ox \alpha$} (B.north)
(F.south east) edge node[descr] {$\Delta $} (G.north west)
(G) edge node[descr,xshift=-15pt] {$(u_X \cdot {\blue \epsilon_B}) \ox 1_B $} (B.north west)
;
\end{tikzpicture} 
%\begin{tikzpicture}[descr/.style={fill=white},baseline=(A.base),scale=0.7] 
%\node (A) at (1,0) {$ A $};
%\node (B) at (5,0) {$X $};
%\node (D) at (5,2.5) {$X \ox B $};
%\node (E) at (2.5,2.5) {$A^2$};
%\node (F) at (1,2.5) {$ A  $};
%\draw[commutative diagrams/.cd, ,font=\scriptsize]
%(F.south) edge[commutative diagrams/equal]  (A.north);
%\path[->,font=\scriptsize]
%(A.east) edge node[above] {$\lambda$} (B.west)
%(E.east) edge node[above] {$ \lambda \ox \alpha $} (D.west)
%(F.east) edge node[above] {$ \Delta$} (E.west)
%(D.south) edge node[descr] {$1_X \ox \epsilon$} (B.north)
%(E.south) edge node[descr] {$\lambda \ox \epsilon $} (B.north west)
%;
%\end{tikzpicture} 
%\begin{tikzpicture}[descr/.style={fill=white},baseline=(A.base),scale=0.7] 
%\node (A) at (1,0) {$ A $};
%\node (B) at (5,0) {$B $};
%\node (D) at (5,2.5) {$X \ox B $};
%\node (E) at (2.5,2.5) {$A^2$};
%\node (F) at (1,2.5) {$ A  $};
%\draw[commutative diagrams/.cd, ,font=\scriptsize]
%(F.south) edge[commutative diagrams/equal]  (A.north);
%\path[->,font=\scriptsize]
%(A.east) edge node[above] {$\alpha$} (B.west)
%(E.east) edge node[above] {$ \lambda \ox \alpha $} (D.west)
%(F.east) edge node[above] {$ \Delta$} (E.west)
%(D.south) edge node[descr] {$\epsilon \ox 1_B$} (B.north)
%(E.south) edge node[descr] {$\epsilon \ox \alpha$} (B.north west)
%;
%\end{tikzpicture} 
% textidote: ignore end

\end{center}
\qed
\end{proof}
%\begin{remark}
%For any $x \in X$ and any $b \in B$, 
%\begin{equation}\label{remark psi}
%(\kappa(x)e(b))(\kappa(x')e(b')) = \kappa(x\;^{b_1}x')e(b_2b')
%\end{equation}
%This is just the fact that $\phi$ is a morphism of algebras.
%\end{remark}

\begin{proposition}
Given a split extension 
% textidote: ignore begin
\begin{tikzpicture}[descr/.style={fill=white},baseline=(A.base)] 
\node (A) at (0,0) {$A$};
\node (B) at (2.5,0) {$B$};
\node (C) at (-2.5,0) {$X$};
\path[dashed,->,font=\scriptsize]
([yshift=2pt]A.west) edge node[above] {$\lambda$} ([yshift=2pt]C.east);
\path[->,font=\scriptsize]
([yshift=-4pt]C.east) edge node[below] {$\kappa$} ([yshift=-4pt]A.west)
([yshift=-4pt]A.east) edge node[below] {$\alpha$} ([yshift=-4pt]B.west)
([yshift=2pt]B.west) edge node[above] {$e$} ([yshift=2pt]A.east);
\end{tikzpicture} 
% textidote: ignore end
of bialgebras, the following properties hold:
\begin{itemize}
%\item[b)] $\kappa$ and $e$ are jointly strongly epic in the category of bialgebras
\item[a)] $\kappa$ is the kernel of $\alpha$ in the category of bialgebras {\blue in $\C$};
\item[b)] $\alpha$ is the cokernel of $\kappa$ in the category of bialgebras {\blue in $\C$};
%\item[d)] $\lambda$ and $\alpha$ form a categorical product in $\mathcal{C}$
{\blue \item[c)] $e$ is the kernel of $\lambda$ in the category of pointed coalgebras in $\C$. }
%\item[d)] $\lambda$ and $\alpha$ form a categorical product in $\mathcal{C}$
\end{itemize}
{\blue Hence, a split extension of bialgebras is a short exact sequence.}
\end{proposition}

\begin{proof}
a) Let $\omega : D \rightarrow A$ be a morphism of bialgebras such that $\alpha \cdot \omega = u_B \cdot {\blue \epsilon_D}$.
We build the morphism $\hat{\omega} : D \rightarrow X$ by setting \[\hat{\omega} = \lambda \cdot \omega.\]
First, we verify that $\kappa \cdot \hat{\omega} = w$ thanks to the commutativity of the following diagram
\begin{equation}\label{a)}
% textidote: ignore begin
\begin{tikzpicture}[descr/.style={fill=white},baseline=(current  bounding  box.center),xscale=1.85,yscale=2] 
\node (A3) at (1,0.4) {$A$};
\node (A0) at (5,0.4) {$A$};
\node (B2) at (2,1) {$A^2$};
\node (B1) at (4,1) {$A^2$};
\node (C2) at (2,2) {$ D^2$};
\node (D3) at (1,2.6) {$D$};
\node (D0) at (5,2.6) {$D$};
\node at (3,0.7) {$(3)$};
\node at (3,2.3) {$\eqref{unital multiplication}+ \eqref{counit comultiplication}$};

\draw[commutative diagrams/.cd, ,font=\scriptsize]
(D0.west) edge[commutative diagrams/equal]  (D3.east)
(A0.west) edge[commutative diagrams/equal]  (A3.east);
%\draw[->] (K.north) to[bend left=15]node[descr] {$ \Delta_{X \rtimes B}$} ([xshift=-1pt]I.north);
%\draw[->] ([xshift=1pt]I.north) to[bend left=15]node[descr] {$S_{X \rtimes B} \ox  1_X \ox 1_B$} (F.north);
\path[->,font=\scriptsize]
(B2.east) edge node[above] {$(\kappa \cdot  \lambda) \ox (e \cdot \alpha) $} (B1.west)
(C2.south east) edge node[descr] {$(\kappa \cdot  \lambda \cdot \omega) \ox (u_A \cdot {\blue \epsilon_D}) $} (B1.north west)
(D3.south east) edge node[descr] {$\Delta $} (C2.north west)
(B1.south east) edge node[descr] {$m $} (A0.north west)
(A3.north east) edge node[descr] {$\Delta $} (B2.south west)
(D0.south) edge node[descr] {$\kappa \cdot \lambda \cdot \omega $} (A0.north)
(D3.south) edge node[descr] {$ \omega $} (A3.north)
(C2.south) edge node[descr] {$ \omega \ox \omega $} (B2.north)
;
\end{tikzpicture}
% textidote: ignore end
 \end{equation}

Moreover, $\hat{\omega}$ is a coalgebra morphism by construction, and an algebra morphism since the following diagram commutes 

\begin{center}
% textidote: ignore begin
\begin{tikzpicture}[descr/.style={fill=white},baseline=(A0.base),xscale=1.2,yscale=1.2] 
\node (A0) at (0,0) {$D^2$};
\node (A2) at (4,0) {$A^2$};
\node (A3) at (6,0) {$X^2$};
\node (A4) at (8,0) {$X.$};
\node (B0) at (0,1) {$D^2$};
\node (B2) at (4,1) {$A^2$};
\node (C2) at (4,2) {$ A^3$};
\node (D0) at (0,3) {$D^2$};
\node (D1) at (2,3) {$D^3$};
\node (D2) at (4,3) {$A^3$};
\node (E2) at (4,4) {$A^2$};
\node (E4) at (8,4) {$A$};
\node (F0) at (0,5) {$D^2$};
\node (F4) at (8,5) {$D$};
\node at (3,0.5) {$(1)$};
\node at (2,2) {$\eqref{unital multiplication}+ \eqref{counit comultiplication}$};
\node at (6,2) {$\eqref{prop lamda morph}$};
\draw[commutative diagrams/.cd, ,font=\scriptsize]
(F0.south) edge[commutative diagrams/equal]  (D0.north)
(D0.south) edge[commutative diagrams/equal]  (B0.north)
(B0.south) edge[commutative diagrams/equal]  (A0.north)
;
%\draw[->] (K.north) to[bend left=15]node[descr] {$ \Delta_{X \rtimes B}$} ([xshift=-1pt]I.north);
%\draw[->] ([xshift=1pt]I.north) to[bend left=15]node[descr] {$S_{X \rtimes B} \ox  1_X \ox 1_B$} (F.north);
\path[->,font=\scriptsize]
(A0.east) edge node[above] {$\omega^2 $} (A2.west)
(A2.east) edge node[above] {$\lambda^2 $} (A3.west)
(A3.east) edge node[above] {$m $} (A4.west)
(F0.east) edge node[above] {$m $} (F4.west)
(E2.east) edge node[above] {$m $} (E4.west)
(D0.east) edge node[above] {$\Delta \ox 1_D $} (D1.west)
(D1.east) edge node[above] {$\omega^3 $} (D2.west)
(B0.east) edge node[above] {$\omega \ox (\kappa \cdot \lambda \cdot \omega) $} (B2.west)
(B2.south east) edge node[descr] {$\lambda^2 $} (A3.north west)
(D1.south east) edge node[descr] {$\omega \ox ( u_A \cdot {\blue \epsilon_D}) \ox (\kappa \cdot \lambda \cdot \omega) \; \; \; \; \; \; \; \; \; \; \;$} (C2.north west)
(D2.south) edge node[right] {$1_A \ox (e \cdot \alpha) \ox (\kappa \cdot \lambda) $} (C2.north)
(F0.south east) edge node[descr] {$\omega^2 $} (E2.north west)
(E2.south) edge node[descr] {$\Delta \ox 1_A$} (D2.north)
(C2.south) edge node[descr] {$ 1_A \ox m $} (B2.north)
(E4.south) edge node[descr] {$ \lambda $} (A4.north)
(F4.south) edge node[descr] {$ \omega $} (E4.north)
;
\end{tikzpicture}
% textidote: ignore end
 \end{center}
Finally, if there exists another morphism $\omega'$ such that $\kappa \cdot \omega' = w$, then by (1)
\[\omega' = \lambda \cdot \kappa \cdot \omega' = \lambda \cdot \omega = \lambda \cdot \kappa \cdot \hat{\omega} =  \hat{\omega}. \]
Hence, $\kappa$ is the kernel of $\alpha$.

b) Let $ \beta : A \rightarrow C$ be a bialgebra morphism such that $\beta \cdot \kappa = u_C \cdot {\blue \epsilon_X}$. We define $\tilde{\beta} : B \rightarrow C$ by \[\tilde{\beta} = \beta \cdot e.\] 
This morphism is a bialgebra morphism, and thanks to Proposition \ref{jointly epic}, it is enough to remark that $\tilde{\beta} \cdot \alpha \cdot \kappa = u_C \cdot {\blue \epsilon_X} = \beta \cdot \kappa$ and $\tilde{\beta} \cdot \alpha \cdot e = \beta \cdot e$, to conclude that $\tilde{\beta} \cdot \alpha = \beta$. Moreover, if there exists another $\beta'$ such that $\beta' \cdot \alpha = \beta$ then, thanks to (1) we have
\[\beta' = \beta' \cdot \alpha \cdot e = \beta \cdot e = \tilde{\beta} \cdot \alpha \cdot e = \tilde{\beta}\] and $\alpha$ is the cokernel of $\kappa$ 

{\blue c) Let $\gamma \colon Y \to A $ be a morphism of coalgebras such that  $\lambda \cdot \gamma = u_X \cdot \epsilon_Y$. We can verify that the equality $e \cdot \hat{\gamma} = \gamma$ holds for $\hat{\gamma} \coloneqq \alpha \cdot \gamma \colon Y \to B$, using the condition (3) in Definition \ref{definition split extension} as in diagram \eqref{a)}. Moreover, it is clear that $\hat{\gamma}$ is a coalgebra morphism and it is the unique morphism such that $e \cdot \hat{\gamma} = \gamma$.}

\qed
\end{proof}

\begin{lemma} \label{lemma morphism of actions}
Let $(g,v,p)$ be a morphism of split extensions as in Definition \ref{morph split ext}, then 
\[ v \cdot \act = \act \cdot (g \ox v),\] where the actions are induced by the split extensions. 
\end{lemma}
\begin{proof}
This follows {\red from the fact that} $(g,v,p)$ is a morphism of split extensions and $p$ is a morphism of bialgebras, as we can see in the following diagram
\begin{center}
% textidote: ignore begin
\begin{tikzpicture}[descr/.style={fill=white},scale=0.7] 
\node (A) at (0,0) {$ B \ox X  $};
\node (B) at (3,0) {$A^2 $};
\node (C) at (6,0) {$A $};
\node (D) at (9,0) {$X$};
\node (A') at (0,-3) {$ B' \ox X'  $};
\node (B') at (3,-3) {$A'^2 $};
\node (C') at (6,-3) {$A' $};
\node (D') at (9,-3) {$X'.$};
\draw[->] (A.north east) to[bend right=340] node[above,scale=0.8] {${\blue \act}$} (D.north west);
\draw[->] (A'.south east) to[bend left=340] node[below,scale=0.8] {${\blue \act}$} (D'.south west);
%\draw[commutative diagrams/.cd, ,font=\scriptsize]
%(D.south) edge[commutative diagrams/equal]  (A.north)
%(D.east) edge[commutative diagrams/equal]  (C.west);
\path[->,font=\scriptsize]
(A.east) edge node[above] {$e \ox \kappa$} (B.west)
(B.east) edge node[above] {$m$} (C.west)
(C.east) edge node[above] {$\lambda$} (D.west)
(A'.east) edge node[above] {$e' \ox \kappa'$} (B'.west)
(B'.east) edge node[above] {$m'$} (C'.west)
(C'.east) edge node[above] {$\lambda'$} (D'.west)
(A.south) edge node[descr] {$g \ox v$} (A'.north)
(B.south) edge node[descr] {$p \ox p$} (B'.north)
(C.south) edge node[descr] {$p$} (C'.north)
(D.south) edge node[descr] {$v$} (D'.north)
;
\end{tikzpicture}
% textidote: ignore end
\end{center}
\qed
\end{proof}

\begin{lemma}\label{lemma morphism of split ext}
Let 
% textidote: ignore begin
\begin{tikzpicture}[descr/.style={fill=white},baseline=(A.base),
xscale=0.9] 
\node (A) at (0,0) {$A$};
\node (B) at (2.5,0) {$B$};
\node (C) at (-2.5,0) {$X$};
\path[dashed,->,font=\scriptsize]
([yshift=2pt]A.west) edge node[above] {$\lambda$} ([yshift=2pt]C.east);
\path[->,font=\scriptsize]
([yshift=-4pt]C.east) edge node[below] {$\kappa$} ([yshift=-4pt]A.west)
([yshift=-4pt]A.east) edge node[below] {$\alpha$} ([yshift=-4pt]B.west)
([yshift=2pt]B.west) edge node[above] {$e$} ([yshift=2pt]A.east);
\end{tikzpicture} 
  and 
\begin{tikzpicture}[descr/.style={fill=white},baseline=(A.base),xscale=0.9] 
\node (A) at (0,0) {$A'$};
\node (B) at (2.5,0) {$B'$};
\node (C) at (-2.5,0) {$X'$};
\path[dashed,->,font=\scriptsize]
([yshift=2pt]A.west) edge node[above] {$\lambda'$} ([yshift=2pt]C.east)
;
\path[->,font=\scriptsize]
([yshift=-4pt]C.east) edge node[below] {$\kappa'$} ([yshift=-4pt]A.west)
([yshift=-4pt]A.east) edge node[below] {$\alpha'$} ([yshift=-4pt]B.west)
([yshift=2pt]B.west) edge node[above] {$e'$} ([yshift=2pt]A.east);
\end{tikzpicture} 
% textidote: ignore end
 be two split extensions of bialgebras, and $g \colon B \rightarrow B'$ and $v \colon X \rightarrow X'$ {\red two morphisms of bialgebras}. Then the following conditions are equivalent:
\begin{itemize}
\item[1)] there exists $p \colon A \rightarrow A'$ such that $(g,v,p)$ is a morphism of split extensions;
\item[2)] there exists a unique $p \colon A \rightarrow A'$ such that $(g,v,p)$ is a morphism of split extensions;
\item[3)] $ v \cdot \act = \act \cdot (g \otimes v)$.
\end{itemize}
\end{lemma}

\begin{proof}
Thanks to Proposition \ref{jointly epic} and Lemma \ref{lemma morphism of actions}, we just need to check that $3) \Rightarrow 1)$. Let us define $\tilde{p} \colon X \rtimes B \rightarrow X' \rtimes B'$ as \[ \tilde{p} = v \otimes g. \]
It is clear that this morphism is a morphism of coalgebras. Moreover, $\tilde{p}$ is a morphism of algebras as we can see in the following diagram
\begin{center}
% textidote: ignore begin
\begin{tikzpicture}[descr/.style={fill=white},yscale=0.8,
xscale=0.8] 
\node (A) at (-3.2,0) {$ (X \ox B)^2 $};
\node (B) at (2,0) {$X \ox B^2 \ox X \ox B  $};
\node (C) at (8,0) {$(X \ox B)^2 \ox B  $};
\node (D) at (12.5,0) {$X^2 \ox B^2 $};
\node (E) at (14.6,0) {$X \ox B$};
\node (A') at (-3.2,-3) {$ (X' \ox B' )^2  $};
\node (B') at (2,-3) {$X' \ox B'^2 \ox X' \ox B' $};
\node (C') at (8,-3) {$(X' \ox B')^2 \ox B' $};
\node (D') at (12.5,-3) {$X'^2 \ox B'^2$};
\node (E') at (14.6,-3) {$X' \ox B'.$};
\node (X) at (11,-1.5) {$3)$};
\draw[->] (A.north) to[bend right=350] node[above,scale=0.8] { ${\blue m_{X \rtimes B}}$} (E.north west);
\draw[->] (A'.south) to[bend left=350] node[below,scale=0.8] {${\blue m_{X' \rtimes B'}}$} (E'.south west);
%\draw[commutative diagrams/.cd, ,font=\scriptsize]
%(D.south) edge[commutative diagrams/equal]  (A.north)
%(D.east) edge[commutative diagrams/equal]  (C.west);
\path[->,font=\scriptsize]
(A.east) edge node[below] {$1_X \ox \Delta \ox 1_X  \ox 1_B$} (B.west)
(B.east) edge node[below]  {$1_X \ox 1_B \ox {\blue \sigma_{B,X}}  \ox 1_B$} (C.west)
(C.east) edge node[below]  {$1_X \ox \act \ox (1_B)^2$} (D.west)
(D.east) edge node[below]  {$m^2$} (E.west)
(A'.east) edge node[below] {$1_{X'} \ox \Delta \ox 1_{X'}  \ox 1_{B'}$} (B'.west)
(B'.east) edge node[below]  {$1_{X'} \ox 1_{B'} \ox {\blue \sigma_{B',X'}}  \ox 1_{B'}$} (C'.west)
(C'.east) edge node[below]  {$1_{X'} \ox \act \ox (1_{B'})^2$} (D'.west)
(D'.east) edge node[above] {$m^2$} (E'.west)
(A.south) edge node[descr] {$(v \ox g )^2$} (A'.north)
(B.south) edge node[descr] {$v \ox g^2 \ox v \ox g$} (B'.north)
(C.south) edge node[descr] {$(v \ox g)^2 \ox g$} (C'.north)
(D.south) edge node[descr] {$v^2 \ox g^2$} (D'.north)
(E.south) edge node[descr] {$v  \ox g$} (E'.north)
;
\end{tikzpicture} 
% textidote: ignore end
\end{center}
Then we can define a morphism $p \colon A \rightarrow A'$ as $\phi_{A'} \cdot \tilde{p} \cdot \psi_A$ where $\phi$ and $\psi$ are the isomorphisms  in Proposition \ref{canonical}. In particular, that gives us
\[p= m \cdot (\kappa' \ox e') \cdot (v \ox g) \cdot (\lambda \ox \alpha) \cdot \Delta. \]
Finally, $(g, v, p)$ is a morphism of split extensions of bialgebras. Indeed, thanks to Corollary \ref{comm} the commutativity of these two diagrams  
\begin{center} 
% textidote: ignore begin
\begin{tikzpicture}[descr/.style={fill=white},baseline=(A.base),yscale=0.7] 
\node (A) at (-2,0) {$ X'  $};
\node (B) at (2,0) {$A'$};
\node (C) at (2,2) {$A'^2  $};
\node (D) at (2,4) {$X' \ox B' $};
\node (E) at (2,6) {$X \ox B$};
\node (F) at (2,8) {$A^2 $};
\node (G) at (2,10) {$A $};
\node (H) at (-2,10) {$X$};
\node (I) at (0,8) {$X^2$};
\node at (1,7.5) {$(2)$};
\draw[->] (G.east) to[bend right=340] node[right,scale=0.8] {${\blue p}$} (B.north east);
\path[->,font=\scriptsize]
(A.east) edge node[above] {$\kappa'$} (B.west)
(H.east) edge node[above] {$\kappa$} (G.west)
(I.east) edge node[above] {$\kappa \ox \kappa$} (F.west)
(G.south) edge node[descr] {$\Delta $} (F.north)
(E.south) edge node[descr] {$v \ox g$} (D.north)
(D.south) edge node[descr] {$\kappa' \ox e'$} (C.north)
(C.south) edge node[descr] {$m$} (B.north)
(H.south) edge node[descr] {$v$} (A.north)
(H.south east) edge node[descr] {$\Delta$} (I.north west)
(I.south) edge node[left] {$1_X \ox (u_B \cdot {\blue \epsilon_X}) \; $} (E.north west)
(F.south) edge node[descr] {$\lambda \ox \alpha$} (E.north)
;
\end{tikzpicture} 
\qquad
\begin{tikzpicture}[descr/.style={fill=white},baseline=(A.base),yscale=0.7] 
\node (A) at (-2,0) {$ B'  $};
\node (B) at (2,0) {$A'$};
\node (C) at (2,2) {$A'^2  $};
\node (D) at (2,4) {$X' \ox B' $};
\node (E) at (2,6) {$X \ox B$};
\node (F) at (2,8) {$A^2 $};
\node (G) at (2,10) {$A $};
\node (H) at (-2,10) {$B$};
\node (I) at (0,8) {$B^2$};
\node at (1,7.5) {$(2)$};
\draw[->] (G.east) to[bend right=340] node[right,scale=0.8] {${\blue p}$} (B.north east);
\path[->,font=\scriptsize]
(A.east) edge node[above] {$e'$} (B.west)
(H.east) edge node[above] {$e$} (G.west)
(I.east) edge node[above] {$e \ox e$} (F.west)
(G.south) edge node[descr] {$\Delta $} (F.north)
(E.south) edge node[descr] {$v \ox g$} (D.north)
(D.south) edge node[descr] {$\kappa' \ox e'$} (C.north)
(C.south) edge node[descr] {$m$} (B.north)
(H.south) edge node[descr] {$g$} (A.north)
(H.south east) edge node[descr] {$\Delta$} (I.north west)
(I.south) edge node[left] {$(u_X \cdot {\blue \epsilon_B}) \ox 1_B\; $} (E.north west)
(F.south) edge node[descr] {$\lambda \ox \alpha$} (E.north)
;
\end{tikzpicture} 
% textidote: ignore end
\end{center}
suffices to conclude that diagram \eqref{commute} commutes. \qed
\end{proof}

\begin{definition}\label{morph action}
Let $\act \colon B \otimes X \rightarrow X$ and $\act' \colon B' \otimes X' \rightarrow X'$ be two actions of bialgebras. A morphism between them is defined as a pair of morphisms of bialgebras $g \colon B \rightarrow B'$ and $v \colon X \rightarrow X'$ such that  \[v \cdot \act = \act' \cdot (g \ox v).\]
\end{definition}
The split extensions of bialgebras (Definition \ref{definition split extension}) endowed with the morphisms of split extensions of bialgebras (Definition \ref{morph split ext}) form the category of split extensions of bialgebras denoted by $\mathsf{SplitExt({\blue BiAlg_\C})}$. The actions of bialgebras (Definition \ref{def action}) with the morphisms of actions (Definition \ref{morph action}) form the category of actions of bialgebras, denoted by $\mathsf{Act({\blue BiAlg_\C})}$.
\begin{theorem}\label{equi bialg}
{\blue Let $\C$ be a symmetric monoidal category}. The category $\mathsf{SplitExt({\blue BiAlg_\C})}$ of split extensions of bialgebras {\blue in $\C$} and the category $\mathsf{Act({\blue BiAlg_\C})}$ of actions of bialgebras {\blue in $\C$} are equivalent.
\end{theorem}

\begin{proof}
The functor $ F \colon \mathsf{SplitExt({\blue BiAlg_\C})} \rightarrow \mathsf{Act({\blue BiAlg_\C})}$ is defined as

% textidote: ignore begin
\[ F \left( 
\begin{tikzpicture}[descr/.style={fill=white},baseline=(O.base)] 
\node (O) at (1,-1) {$ $};
\node (A) at (0,0) {$A$};
\node (B) at (2.5,0) {$B$};
\node (C) at (-2.5,0) {$X$};
\node (A') at (0,-2) {$A'$};
\node (B') at (2.5,-2) {$B'$};
\node (C') at (-2.5,-2) {$X'$};
\path[dashed,->,font=\scriptsize]
([yshift=2pt]A'.west) edge node[above] {$\lambda'$} ([yshift=2pt]C'.east)
([yshift=2pt]A.west) edge node[above] {$\lambda$} ([yshift=2pt]C.east)
;
\path[->,font=\scriptsize]
(C.south) edge node[left] {$ v$} (C'.north)
(B.south) edge node[left] {$ g$} (B'.north)
(A.south) edge node[left] {$ p$} (A'.north)
([yshift=-4pt]C'.east) edge node[below] {$\kappa'$} ([yshift=-4pt]A'.west)
([yshift=-4pt]A'.east) edge node[below] {$\alpha'$} ([yshift=-4pt]B'.west)
([yshift=2pt]B'.west) edge node[above] {$e'$} ([yshift=2pt]A'.east)
([yshift=-4pt]C.east) edge node[below] {$\kappa$} ([yshift=-4pt]A.west)
([yshift=-4pt]A.east) edge node[below] {$\alpha$} ([yshift=-4pt]B.west)
([yshift=2pt]B.west) edge node[above] {$e$} ([yshift=2pt]A.east);
\end{tikzpicture} 
\right) =
\begin{tikzpicture}[descr/.style={fill=white},baseline=(O.base)] 
\node (E) at (4,0) {$B \otimes X$};
\node (E') at (4,-2) {$B' \otimes X'$};
\node (F) at (6,0) {$X$};
\node (F') at (6,-2) {$X'$};
\path[->,font=\scriptsize]
(E.east) edge node[above] {$\act$} (F.west)
(E'.east) edge node[above] {$\act'$} (F'.west)
(E.south) edge node[descr] {$ g \otimes v$} (E'.north)
(F.south) edge node[left] {$ v$} (F'.north)
;
\end{tikzpicture} 
  \]
% textidote: ignore end
where $\act= \lambda \cdot m \cdot (e \ox \kappa)$ as in Proposition \ref{Action bialgebra}, and we have a morphism of actions thanks to Lemma \ref{lemma morphism of actions}. The functor $G \colon \mathsf{Act({\blue BiAlg_\C})} \rightarrow \mathsf{SplitExt({\blue BiAlg_\C})}$ is defined as
% textidote: ignore begin
\[ G \left( 
\begin{tikzpicture}[descr/.style={fill=white},baseline=(O.base)] 
\node (O) at (-5,-1) {$ $};
\node (E) at (-6.5,0) {$B \otimes X$};
\node (E') at (-6.5,-2) {$B' \otimes X'$};
\node (F) at (-4.5,0) {$X$};
\node (F') at (-4.5,-2) {$X'$};
\path[->,font=\scriptsize]
(E.east) edge node[above] {$\act$} (F.west)
(E'.east) edge node[above] {$\act'$} (F'.west)
(E.south) edge node[descr] {$ g \otimes v$} (E'.north)
(F.south) edge node[left] {$ v$} (F'.north);
\end{tikzpicture} 
 \right) = 
\begin{tikzpicture}[descr/.style={fill=white},baseline=(O.base)] 
\node (O) at (-1,-1) {$ $};
\node (A) at (0,0) {$X \rtimes B$};
\node (B) at (2.5,0) {$B$};
\node (C) at (-2.5,0) {$X$};
\node (A') at (0,-2) {$X' \rtimes B'$};
\node (B') at (2.5,-2) {$B'$};
\node (C') at (-2.5,-2) {$X'$};
\path[dashed,->,font=\scriptsize]
([yshift=2pt]A'.west) edge node[above] {$\pi_1'$} ([yshift=2pt]C'.east)
([yshift=2pt]A.west) edge node[above] {$\pi_1$} ([yshift=2pt]C.east);
\path[->,font=\scriptsize]
(C.south) edge node[left] {$ v$} (C'.north)
(B.south) edge node[left] {$ g$} (B'.north)
(A.south) edge node[left] {$ p$} (A'.north)
([yshift=-4pt]C'.east) edge node[below] {$i_1'$} ([yshift=-4pt]A'.west)
([yshift=-4pt]A'.east) edge node[below] {$\pi_2'$} ([yshift=-4pt]B'.west)
([yshift=2pt]B'.west) edge node[above] {$i_2'$} ([yshift=2pt]A'.east)
([yshift=-4pt]C.east) edge node[below] {$i_1$} ([yshift=-4pt]A.west)
([yshift=-4pt]A.east) edge node[below] {$\pi_2$} ([yshift=-4pt]B.west)
([yshift=2pt]B.west) edge node[above] {$i_2$} ([yshift=2pt]A.east);
\end{tikzpicture} 
 \]
% textidote: ignore end
where $p{\blue \coloneqq} v \otimes g$ is given by Lemma \ref{lemma morphism of split ext} and the bialgebra structure{\blue s}  of the semi-direct product{\blue s}  $X \rtimes B$ and $X' \rtimes B'$ are defined as in \eqref{semi-direct product}.

{\red We observe that} \begin{equation}\label{act split}
{\blue ( } F \cdot G {\blue ) } ( \act) =   \pi_1 \cdot m_{X \rtimes B} \cdot (i_2 \otimes i_1) = \triangleright,\end{equation}
%\begin{align*}
%\pi_1(i_2(b)i_1(x)) &= \pi_1((1 \otimes b)(x \otimes 1))\\
%&= \pi_1(\;^{b_1}x \otimes b_2) \\
%&=\;^bx.
%\end{align*}
where the last equality holds thanks to the commutativity of the following diagram
\begin{center}
% textidote: ignore begin
\begin{tikzpicture}[descr/.style={fill=white},baseline=(A.base),yscale=0.85,xscale=0.85] 
\node (A) at (-6,0) {$B \ox X$};
\node (B) at (10,0) {$X.$};
\node (C) at (10,2) {$X \ox B  $};
\node (D) at (10,4) {$ X^2 \ox B^2$};
\node (E) at (10,6) {$ X \ox B \ox X \ox B^2$};
\node (F) at (4,6) {$ X \ox B^2 \ox X \ox B$};
\node (G) at (-1,6) {$ (X \ox B)^2 $};
\node (H) at (-6,6) {$B \ox X$};
\node (I) at (-6,4) {$B \ox X$};
\node (J) at (-2,4) {$B^2 \ox X$};
\node (K) at (1.5,4) {$B \ox X  \ox B$};
\node (L) at (5,4) {$ X \ox B$};
\node at (8.5,3) {$ \eqref{unital multiplication}$};
\node at (2,2) {$ \eqref{counit comultiplication}$};
\draw[commutative diagrams/.cd, ,font=\scriptsize]
(L.south) edge[commutative diagrams/equal]  (C.west)
(I.south) edge[commutative diagrams/equal]  (A.north)
(H.south) edge[commutative diagrams/equal]  (I.north);
\path[->,font=\scriptsize]
(F.east) edge node[above] {$1_X \ox 1_B \ox {\blue \sigma_{B,X}} \ox 1_B$} (E.west)
(E.south) edge node[descr] {$1_X \ox  \act \ox 1_B \ox 1_B$} (D.north)
(D.south) edge node[descr] {$ m \ox m$} (C.north)
(C.south) edge node[descr] {$1_X \ox {\blue \epsilon_B}$} (B.north)
(A.east) edge node[above] {$ \act $} (B.west)
(H.east) edge node[above] {$u_X \ox 1_B \ox 1_X \ox u_B$} (G.west)
(G.east) edge node[above] {$1_X \ox \Delta \ox 1_X \ox 1_B$} (F.west)
(I.east) edge node[above] {$\Delta  \ox 1_X$} (J.west)
(J.east) edge node[above] {$1_B \ox {\blue \sigma_{B,X}}$} (K.west)
(K.east) edge node[above] {$\act \ox 1_B$} (L.west)
(L.east) edge node[above] {$u_X \ox 1_X \ox 1_B \ox u_B$} (D.west);
\end{tikzpicture}
% textidote: ignore end
\end{center}
{\red Thanks to this observation and the isomorphisms $\phi$ and $\psi$ of \eqref{diag iso}, the functors $F$ and $G$ give rise to an equivalence of categories}.
\qed
\end{proof}

%{\blue 
%\begin{remark}
%Note that we directly obtain dual results for non-coassociative bialgebras, by considering non-associative bialgebras in the symmetric monoidal category $\C^{op}$, the dual of $\C$. The notion of action becomes coaction which satisfies the dual conditions of Definition \ref{def action}. In particular, \eqref{coco act} becomes the compatibility condition of a Yetter-Drinfield module with a trivial action. These coactions are equivalent to the split extensions of bialgebras in $\C^{op}$. In particular, these extensions are split monomorphisms such that the cokernel is split in the category of algebras in $\C$ (the category of coalgebras in $\C^{op}$). 
%\end{remark}
%}

To end this section we prove that a variation of the Split Short Five Lemma holds in $\mathsf{{\blue BiAlg_\C}}$. 

\begin{theorem}\label{Sec SSFL}
Let $(g,v,p)$ be a morphism of split extensions of bialgebras {\blue in a symmetric monoidal category $\C$} \begin{equation}\label{diag ssFL} 
\begin{tikzpicture}[descr/.style={fill=white},baseline=(current  bounding  box.center)] 
\node (A) at (0,0) {$A$};
\node (B) at (2.5,0) {$B$};
\node (C) at (-2.5,0) {$X$};
\node (A') at (0,-2) {$A'$};
\node (B',) at (2.5,-2) {$B'$};
\node (C') at (-2.5,-2) {$X'$};
\path[dashed,->,font=\scriptsize]
([yshift=2pt]A'.west) edge node[above] {$\lambda'$} ([yshift=2pt]C'.east)
([yshift=2pt]A.west) edge node[above] {$\lambda$} ([yshift=2pt]C.east)
;
\path[->,font=\scriptsize]
(B.south) edge node[right] {$ g$}  (B'.north)
 (C.south) edge node[left] {$ v $}  (C'.north)
(A.south) edge node[left] {$ p$} (A'.north)
([yshift=-4pt]C'.east) edge node[below] {$\kappa'$} ([yshift=-4pt]A'.west)
([yshift=-4pt]A'.east) edge node[below] {$\alpha'$} ([yshift=-4pt]B'.west)
([yshift=2pt]B'.west) edge node[above] {$e'$} ([yshift=2pt]A'.east)
([yshift=-4pt]C.east) edge node[below] {$\kappa$} ([yshift=-4pt]A.west)
([yshift=-4pt]A.east) edge node[below] {$\alpha$} ([yshift=-4pt]B.west)
([yshift=2pt]B.west) edge node[above] {$e$} ([yshift=2pt]A.east);
\end{tikzpicture} 
 \end{equation} then $p$ is an isomorphism whenever $v$ and $g$ are.
\end{theorem}

{\green \begin{proof}

Thanks to Theorem \ref{equi bialg}, the diagram \eqref{diag ssFL} is canonically isomorphic to 
% textidote: ignore begin
\[ 
\begin{tikzpicture}[descr/.style={fill=white},baseline=(A'.base)] 
\node (A) at (0,0) {$X \rtimes B$};
\node (B) at (2.5,0) {$B$};
\node (C) at (-2.5,0) {$X$};
\node (A') at (0,-2) {$X' \rtimes B'$};
\node (B') at (2.5,-2) {$B'.$};
\node (C') at (-2.5,-2) {$X'$};
\path[dashed,->,font=\scriptsize]
([yshift=2pt]A'.west) edge node[above] {$\pi_1'$} ([yshift=2pt]C'.east)
([yshift=2pt]A.west) edge node[above] {$\pi_1$} ([yshift=2pt]C.east)
;
\path[->,font=\scriptsize]
(B.south) edge node[right] {$ g$}  (B'.north)
 (C.south) edge node[left] {$ v $}  (C'.north)
(A.south) edge node[descr] {$ v \otimes g$} (A'.north)
([yshift=-4pt]C'.east) edge node[below] {$i_1'$} ([yshift=-4pt]A'.west)
([yshift=-4pt]A'.east) edge node[below] {$\pi_2'$} ([yshift=-4pt]B'.west)
([yshift=2pt]B'.west) edge node[above] {$i_2'$} ([yshift=2pt]A'.east)
([yshift=-4pt]C.east) edge node[below] {$i_1$} ([yshift=-4pt]A.west)
([yshift=-4pt]A.east) edge node[below] {$\pi_2$} ([yshift=-4pt]B.west)
([yshift=2pt]B.west) edge node[above] {$i_2$} ([yshift=2pt]A.east);
\end{tikzpicture}
 \] % textidote: ignore end
 It follows that $v \otimes g$ is an isomorphism whenever $v$ and $g$ are.
 \qed
\end{proof}}

{\blue \section{The cocommutative and associative cases}}

If we consider cocommutative bialgebras, then the category $\sf {\blue BiAlg_{\C,coc}}$ of cocommutative bialgebras can be seen as the category of internal magmas in the category of cocommutative coalgebras. Indeed, the categorical product of cocommutative coalgebras is given by the tensor product. 

In this particular case, {\red the condition \eqref{coco} in Definition \ref{def action} becomes trivial and the definition can be reformulated explicitly as}:
\begin{definition}
Let $X$ and $B$ be cocommutative bialgebras. An action of $B$ on $X$ is a morphism of coalgebras $\act \colon B \otimes X \rightarrow X $ such that
\begin{equation*}
\act \cdot (u_B \ox 1_X)= 1_X,
\end{equation*} 
\begin{equation*}
\act \cdot (1_B \ox u_X)= u_X \cdot {\blue \epsilon_B}.
\end{equation*} 
\end{definition} 
{\red Similarly}, thanks to the cocommutativity we can drop the condition (5) in the definition of split extensions of bialgebras (Definition \ref{definition split extension}), and this turns out to be exactly the internal version, in the category of coalgebras, of Definition 1.4 in \cite{GJS}. Then, in the case of cocommutative bialgebras, the above theorem reduces to the results in Section 4.6 in \cite{GJS}. {\green Indeed, if $\mathcal{C}$ is a category with finite limits, it is in particular a cartesian monoidal category. Hence, any magma in such a category can be seen as a non-associative cocommutative bialgebra in this category ({\rouge where the unique comultiplication is the diagonal). Accordingly,} the results in \cite{GJS} become a particular case of our theorem. }

%\begin{example}\label{example iso}
%Let us see a special case of this particular class of split extensions of bialgebras. 
%Any isomorphism $\gamma : A \rightarrow A$ of bialgebras is a split extension of bialgebras, \\ $% textidote: ignore begin
%\begin{tikzpicture}%[descr/.style={fill=white},baseline=(A.base)] 
%\node (A) at (0,0) {$A$};
%\node (B) at (2.5,0) {$A$};
%\node (C) at (-2.5,0) {$K$};
%\path[->,font=\scriptsize]
%([yshift=-4pt]C.east) edge node[below] {$u$} ([yshift=-4pt]A.west)
%([yshift=2pt]A.west) edge node[above] {$\epsilon$} ([yshift=2pt]C.east)
%([yshift=-4pt]A.east) edge node[below] {$\gamma$} ([yshift=-4pt]B.west)
%([yshift=2pt]B.west) edge node[above] {$\gamma^{-1}$} ([yshift=2pt]A.east);
%\end{tikzpicture} % textidote: ignore end
% $. We easily check that this extension satisfies all the axioms of Definition \ref{definition split extension}.\end{example}
{\blue 
It is interesting to note that the multiplication $m_{X \rtimes B} : (X \ox B) \ox (X \ox B) \rightarrow X \ox B$, defined in \eqref{semi-direct product}, is not associative in general. Even if $X$ and $B$ are associative bialgebras this structure need not be associative. Let us illustrate this observation with a simple example. 
\begin{example}\label{ex ass bialgebra} 
 Let $\C$ be the category of sets, and we consider the monoids $(\mathbb{N},+,0)$ and $(\mathbb{N}_0,{\ro \cdot},1)$. In particular we can construct an action of $(\mathbb{N}_0,{\ro \cdot},1)$ on $(\mathbb{N},+,0)$ via the function $\act \colon \mathbb{N}_0 \times \mathbb{N} \to \mathbb{N}$ sending $(b,x)$ to $x$ to the power of $b$, $x^b$. We can observe that this function satisfies the conditions \eqref{conditionaction1} and \eqref{condition action = e} of Definition \ref{def action} (the other ones being trivial in $\C = \sf Set$): $x^1 = x$, $0^b = 0$ for any $x \in \mathbb{N}, b \in \mathbb{N}_0$. By defining, via this action, the bialgebra structure on $\mathbb{N} \rtimes \mathbb{N}_0$ as $m\left((x,b),(y,c)\right) = (x + y^b,b{\ro \cdot} c)$, this multiplication is non-associative: 
\[ m\left(m\left((0,2),(1,1)\right),(1,1)\right) = m((1,2),(1,1)) = (1+ 1^2,2 {\ro \cdot} 1) = (2,2)\]
\[ m\left((0,2),m\left((1,1),(1,1)\right)\right) = m((0,2),(2,1)) = (0+ 2^2,2 {\ro \cdot} 1) = (4,2).\]
\end{example}
We need additional conditions on the actions of two associative bialgebras $B$ and $X$, to obtain an associative bialgebra structure on $X \rtimes B$. In particular, we have the following Lemma.}

\begin{lemma}\label{ass bialgebra} 
{\ro Let $X$ and $B$ be two associative bialgebras}, then $m_{X \rtimes B} : (X \ox B) \ox (X \ox B) \rightarrow X \ox B$ is associative as well if and only if the following conditions are satisfied
% textidote: ignore begin
\begin{equation}\label{ass action_1}
\triangleright \cdot (m \otimes 1_X) = \triangleright \cdot (1_B \otimes \triangleright),
\end{equation}
\begin{equation}\label{ass action_2}
\triangleright \cdot (1_B \otimes m) = m \cdot (\triangleright \otimes \triangleright) \cdot (1_B \otimes {\blue \sigma_{B,X}} \otimes 1_X) \cdot (\Delta \otimes 1_X \otimes 1_X).
\end{equation}
% textidote: ignore end
\end{lemma}
\begin{proof}
Via Figure \ref{3.1thenasss}, we show that if $m_{X \rtimes B}$ is associative then \eqref{ass action_1} is immediately satisfied. With similar computations, we can show that the associativity also gives the condition \eqref{ass action_2}. The other implication is given by the commutativity of the diagram {\red in Figure \ref{sm ass} in the appendix} (it is a direct computation using the conditions \eqref{ass action_1} and \eqref{ass action_2}).
\setcounter{figure}{7}
\begin{figure}
\centering
% textidote: ignore begin
\begin{tikzpicture}[descr/.style={fill=white},baseline=(A0.base),
xscale=1.83,yscale=2] 
\node (A0) at (0.5,0) {$ B^2 \ox X$};
%\node (A1) at (2,0) {$ X \ox B \ox X$};
\node (A2) at (4,0) {$ X \ox B \ox X$};
\node (A3) at (6,0) {$X^2$};
\node (A4) at (8,0) {$X$};
\node (B0) at (0.5,1) {$  B^2 \ox X$};
\node (B1) at (2,1) {$X \ox B^2 \ox X$};
\node (B2) at (4,1) {$ (X \ox B)^2 $};
\node (B3) at (6,1) {$X \ox B$};
\node (B4) at (8,1) {$X^2$};
\node (C1) at (2,2) {$ B \ox X \ox B^2 \ox X$};
\node (D0) at (0.5,3) {$B^2 \ox X$};
\node (D1) at (2,3) {$B^3 \ox X$};
\node (D2) at (4,3) {$(X \ox B)^3$};
\node (D3) at (6,3) {$(X \ox B)^2$};
\node (D4) at (8,3) {$X \ox B \ox X$};
%\node (E0) at (0,4) {$ X \ox B$};
\node (E1) at (2,4) {$B^2 \ox X$};
\node (E2) at (4,4) {$ X \ox B^3 \ox X$};
\node (E3) at (6,4) {$X \ox B^2 \ox X \ox B$};
%\node (E4) at (8,4) {$X \ox B$};
\node (F0) at (0.5,5) {$B^2 \ox X$};
\node (F4) at (8,5) {$B^2 \ox X$};
\node at (1,2) {$\eqref{condition action = e}$};
\node at (5,2.25) {$associativity$};
\node at (1,4) {$\eqref{coass comultiplication}$};
\node at (7,4.5) {$\eqref{counit comultiplication}$};
\node at (3,0.75) {$\eqref{u et epsilon}$};
\node at (3,3) {$\eqref{unital multiplication}$};
\node at (5,3.75) {$\eqref{unital multiplication}$};
\node at (5.5,0.5) {$(A)$};
\node at (7,2) {$(A)$};
\draw[commutative diagrams/.cd, ,font=\scriptsize]
(F0.south) edge[commutative diagrams/equal]  (D0.north)
(F0.east) edge[commutative diagrams/equal]  (F4.west)
(F0.south east) edge[commutative diagrams/equal]  (E1.north west)
(D0.south) edge[commutative diagrams/equal]  (B0.north)
(B0.south) edge[commutative diagrams/equal]  (A0.north);
%\draw[->] (K.north) to[bend left=15]node[descr] {$ \Delta_{X \rtimes B}$} ([xshift=-1pt]I.north);
%\draw[->] ([xshift=1pt]I.north) to[bend left=15]node[descr] {$S_{X \rtimes B} \ox  1_X \ox 1_B$} (F.north);
\draw[->] (F4.east) to[bend left=15]node[descr,scale=0.8,xshift=-15pt] {${\blue \act\cdot (1_B \ox \act)}$} (A4.east);
\draw[->] (A0.south east) to[bend right=12]node[descr,scale=0.8] {${\blue \act\cdot (m \ox 1_X)}$} (A4.south west);
\path[->,font=\scriptsize]
(E1.south east) edge node[descr] {$ \; \; \; \; \; \; \; \; \; \; \; \; \; \; \; \; \; \; \; \; \; \; \; \; \; \; \; \; \; \; \; \; \; \; \; u_X \ox 1_B \ox u_X \ox 1_B \ox 1_X \ox u_B $} (D2.north west)
(A0.east) edge node[above] {$u_X \ox m \ox 1_X $} (A2.west)
(A2.east) edge node[above] {$1_X \ox \act $} (A3.west)
(A3.east) edge node[above] {$m$} (A4.west)
(B0.east) edge node[above] {$u_X \ox (1_B)^2 \ox 1_X $} (B1.west)
(B1.east) edge node[above] {$1_X \ox m \ox 1_X \ox u_B$} (B2.west)
(B2.east) edge node[above] {$m_{X \rtimes B} $} (B3.west)
(B3.south east) edge node[descr] {$1_X \ox {\blue \epsilon_B}$} (A4.north west)
(D1.west) edge node[above] {$ {\blue \epsilon_B} \ox (1_B)^2 \ox 1_X $} (D0.east)
(D2.east) edge node[above] {$1_X \ox 1_B \ox m_{X \rtimes B}$} (D3.west)
(D3.east) edge node[above] {$1_X \ox 1_B \ox 1_X \ox {\blue \epsilon_B}$} (D4.west)
(E1.east) edge node[above] {$ u_X \ox 1_B\ox \Delta \ox 1_X$} (E2.west)
(E2.east) edge node[above] {$1_X \ox (1_B)^2 \ox {\blue \sigma_{B,X}} $} (E3.west)
(E1.south) edge node[descr] {$\Delta \ox 1_B \ox 1_X$} (D1.north)
(D1.south) edge node[descr] {$1_B \ox u_X \ox (1_B)^2 \ox 1_X$} (C1.north)
(C1.south) edge node[descr] {$\act \ox (1_B)^2 \ox 1_X$} (B1.north)
(D2.south) edge node[descr] {$m_{X \rtimes B} \ox 1_X \ox 1_B$} (B2.north)
(D3.south) edge node[descr] {$m_{X \rtimes B}$} (B3.north)
(B2.south) edge node[descr] {$1_X \ox 1_B \ox 1_X \ox {\blue \epsilon_B}$} (A2.north)
(D4.south) edge node[descr] {$1_X \ox  \act$} (B4.north)
(B1.south) edge node[descr] {$1_X \ox m \ox 1_X  $} (A2.north west)
(B4.south) edge node[descr] {$m$} (A4.north)
(F4.south) edge node[descr] {$u_X \ox 1_B \ox \act$} (D4.north)
(E3.south) edge node[descr] {$1_X \ox 1_B \ox \act \ox 1_B$} (D3.north)
;
\end{tikzpicture} 
\caption{Associativity implies \eqref{ass action_1}}
\label{3.1thenasss}
% textidote: ignore end
 \end{figure}
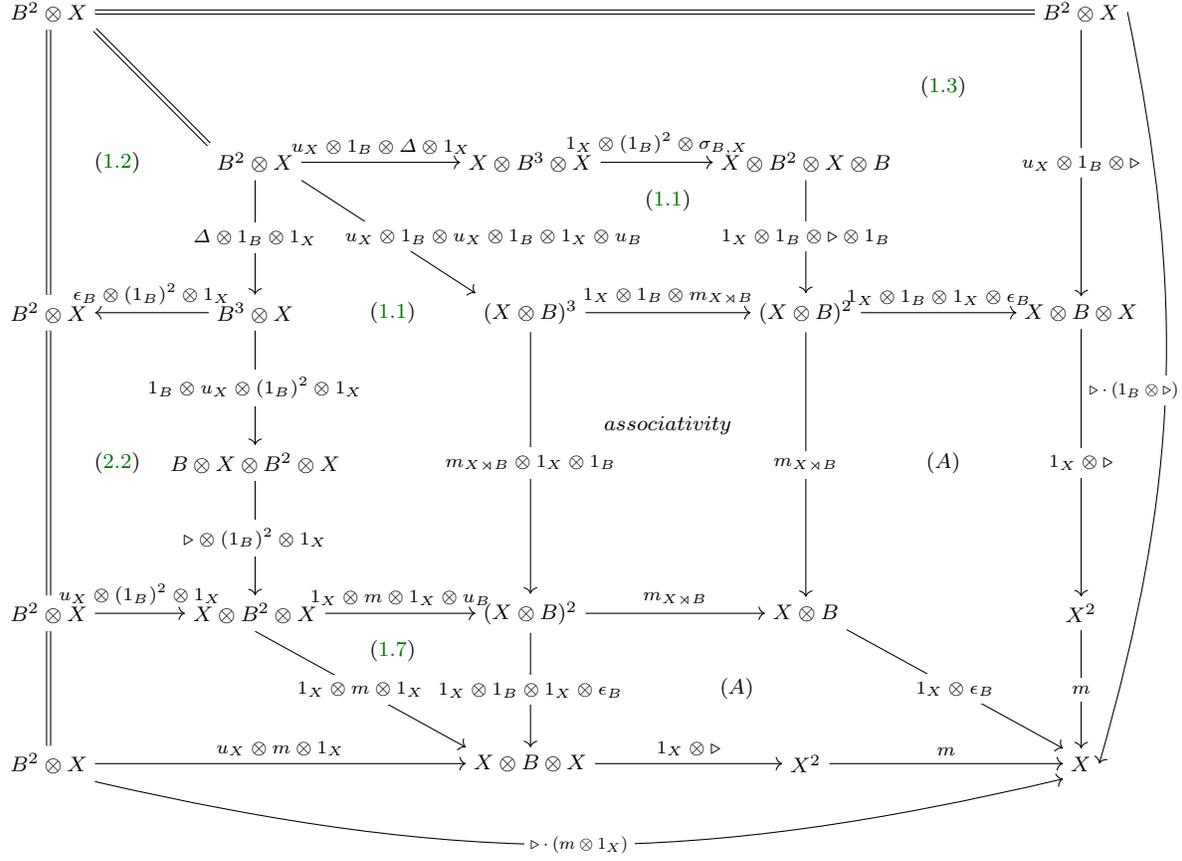
The {\blue trapezoids $(A)$} commute thanks to \eqref{counit comultiplication} and \eqref{m et epsilon}, as it is shown in the following diagram
\begin{center}
% textidote: ignore begin
\begin{tikzpicture}[descr/.style={fill=white},baseline=(A0.base),
xscale=1.8,yscale=1.7] 
\node (A0) at (0.5,0) {$ X \ox B \ox X$};
%\node (A1) at (2,0) {$ A^2$};
\node (A2) at (4.5,0) {$ X \ox B \ox X$};
\node (A3) at (6.5,0) {$X^2$};
\node (A4) at (8,0) {$X.$};
%\node (B0) at (0,1) {$  A$};
\node (B1) at (1.5,1) {$X \ox B^2 \ox X \ox B$};
\node (B2) at (4.5,1) {$ X \ox B \ox X \ox B^2 $};
\node (B3) at (7,1) {$X^2 \ox B^2$};
%\node (B4) at (8,1) {$X^2$};
\node (C0) at (0.5,2) {$(X \ox B)^2$};
%\node (C1) at (2,2) {$A^3 $};
%\node (C2) at (4,2) {$A^4 $};
%\node (C3) at (6,2) {$A^4 $};
\node (C4) at (8,2) {$X \ox B $};
%\node at (1,2) {$\kappa_\alpha \; equalizer$};
%\node at (3,4) {$\eqref{coass comultiplication}$};
\node at (1.5,0.25) {$\eqref{counit comultiplication}$};
\node at (7.5,1.25) {$\eqref{m et epsilon}$};
%\node at (3,3) {$\eqref{unital multiplication}$};
%\node at (5,3.75) {$\eqref{unital multiplication}$};
%\node at (5.5,0.5) {$\eqref{coass comultiplication}$};
%\node at (7,2) {$\eqref{coass comultiplication}$};
\draw[commutative diagrams/.cd, ,font=\scriptsize]
(A0.east) edge[commutative diagrams/equal]  (A2.west);
%(F0.south east) edge[commutative diagrams/equal]  (E1.north west)
%(D0.south) edge[commutative diagrams/equal]  (B0.north)
%(B0.south) edge[commutative diagrams/equal]  (A0.north);
%\draw[->] (K.north) to[bend left=15]node[descr] {$ \Delta_{X \rtimes B}$} ([xshift=-1pt]I.north);
%\draw[->] ([xshift=1pt]I.north) to[bend left=15]node[descr] {$S_{X \rtimes B} \ox  1_X \ox 1_B$} (F.north);
%\draw[->] (F4.east) to[bend left=15]node[descr] {$ \act\cdot (1_B \ox \act)$} (A4.east);
%\draw[->] (A0.south east) to[bend right=12]node[descr] {$ \act\cdot (m \ox 1_X)$} (A4.south west);
\path[->,font=\scriptsize]
(A2.east) edge node[above] {$1_X \ox \act $} (A3.west)
(A3.east) edge node[above] {$m$} (A4.west)
(B1.east) edge node[above] {$1_X \ox 1_B \ox {\blue \sigma_{B,X}} \ox 1_B$} (B2.west)
(B2.east) edge node[above] {$1_X \ox \act  \ox (1_B)^2$} (B3.west)
(C0.east) edge node[above] {$m_{X \rtimes B}$} (C4.west)
(C0.south east) edge node[descr] {$1_X \ox \Delta \ox 1_X \ox 1_B$} (B1.north)
(B1.south east) edge node[descr] {$1_X \ox 1_B \ox {\blue \epsilon_B} \ox 1_X \ox {\blue \epsilon_B} \; \; \; \; \; \; \; \; \; \; \; \; \; \; $} (A2.north west)
(B3.south) edge node[descr] {$m\ox {\blue \epsilon_B} \ox {\blue \epsilon_B}$} (A4.north west)
(C0.south) edge node[descr,yshift=-15pt] {$1_X \ox 1_B \ox 1_X \ox {\blue \epsilon_B} $} (A0.north)
(B2.south) edge node[descr,xshift=15pt] {$1_X \ox 1_B \ox 1_X \ox {\blue \epsilon_B} \ox {\blue \epsilon_B}$} (A2.north)
(C4.south) edge node[descr] {$1_X \ox {\blue \epsilon_B}$} (A4.north)
(B3.north) edge node[descr] {$m \ox m $} (C4.south west)
;
\end{tikzpicture} 
% textidote: ignore end
 \end{center}
 \qed
\end{proof}

{\blue
\begin{remark}
In the symmetric monoidal category of sets, the conditions \eqref{ass action_1} and \eqref{ass action_2} correspond to the conditions \eqref{ass1 group} and \eqref{ass2 group} (considered in the more restricted setting of groups). 
\end{remark}
 }

Let us consider associative bialgebras. We define the categories $\mathsf{Act(Ass{\blue BiAlg_\C})}$ and $ \mathsf{SplitExt(Ass{\blue BiAlg_\C})}$. An object in $\mathsf{Act(Ass{\blue BiAlg_\C})}$, is an action of associative bialgebras (Definition \ref{def action}) satisfying \eqref{ass action_1} and \eqref{ass action_2}, the morphisms are the morphisms of $\mathsf{Act({\blue BiAlg_\C})}$. The category $ \mathsf{SplitExt(Ass{\blue BiAlg_\C})}$ is a full subcategory of $ \mathsf{SplitExt({\blue BiAlg_\C})}$ since the conditions $(6),(7)$ and $(8)$ become redundant. {\blue In particular, since Definition \ref{definition split extension} is a generalization of the split extensions of magmas introduced in \cite{GJS} (which are a generalization of the ``Schreier split epimorphisms'' of monoids), it is clear that the split extensions of associative bialgebras generalize the notion of  ``Schreier split epimorphisms'' of monoids introduced in \cite{BMS}.} 

\begin{corollary}
{\blue Let $\C$ be a symmetric monoidal category}. There is an equivalence between $\mathsf{SplitExt(Ass{\blue BiAlg_\C})}$ the category of split extensions of associative bialgebras {\blue in $\C$} and $\mathsf{Act(Ass{\blue BiAlg_\C})}$ the category of actions of associative bialgebras {\blue in $\C$}.
\end{corollary}
\begin{proof}
It is clear by applying Lemma \ref{ass bialgebra} and Theorem \ref{equi bialg}.  
\qed
\end{proof}

 \section{Split extensions of non-associative Hopf algebras}

In this section, we consider a similar result for non-associative Hopf algebras. We prove an equivalence between the category of split extensions of non-associative Hopf algebras and the category of actions of non-associative Hopf algebras. 

\begin{Convention} For the sake of simplicity, in this section ``Hopf algebra'' will mean ``non-associative Hopf algebra'' (unless the associativity is explicitly mentioned).
\end{Convention}

\begin{definition}\label{definition split extenion Hopf algebras}
A split extension of Hopf algebras is a split extension of bialgebras
\begin{equation}\label{split ext Hopf algebras}
% textidote: ignore begin
\begin{tikzpicture}[descr/.style={fill=white},baseline=(A.base)] 
\node (A) at (0,0) {$A$};
\node (B) at (2.5,0) {$B$};
\node (C) at (-2.5,0) {$X$};
\path[dashed,->,font=\scriptsize]
([yshift=2pt]A.west) edge node[above] {$\lambda$} ([yshift=2pt]C.east);
\path[->,font=\scriptsize]
([yshift=-4pt]C.east) edge node[below] {$\kappa$} ([yshift=-4pt]A.west)
([yshift=-4pt]A.east) edge node[below] {$\alpha$} ([yshift=-4pt]B.west)
([yshift=2pt]B.west) edge node[above] {$e$} ([yshift=2pt]A.east);
\end{tikzpicture}
% textidote: ignore end
\end{equation} such that $X$, $A$, $B$ are Hopf algebras and $\kappa, \alpha, e$ are morphisms of Hopf algebras, with an additional condition of associativity (condition $(9')$) and an additional condition about the {\rouge left and right} antipodes (conditions (10') and (11')). More precisely, the split extension \eqref{split ext Hopf algebras} satisfies
\begin{itemize}
\item[(1')] $\lambda \cdot \kappa = 1_X$, $\alpha \cdot e =1_B$ 
\item[(2')] $\lambda \cdot e = u_X \cdot\epsilon_B $, $\alpha \cdot \kappa  = u_B \cdot \epsilon_X $
\item[(3')] $m \cdot ((\kappa\cdot \lambda) \otimes (e \cdot \alpha)) \cdot \Delta = 1_A$
\item[(4')]$ \lambda \cdot m \cdot (\kappa \otimes e) =  1_X \otimes \epsilon_B$
\item[(5')] $(1_B \otimes \lambda) \cdot (1_B \otimes m) \cdot ( 1_B \otimes e\otimes \kappa)  \cdot (\Delta \otimes 1_X)= (1_B \otimes \lambda) \cdot (1_B \otimes m) \cdot ( 1_B \otimes e\otimes \kappa) \cdot ({\blue \sigma_{B,B}} \otimes 1_X) \cdot (\Delta \otimes 1_X)$
\item[(6')] 
$m \cdot( m \otimes 1_A) \cdot (\kappa \otimes e \otimes 1_A) = m \cdot( 1_A \otimes m) \cdot (\kappa \otimes e \otimes 1_A)$
\item[(7')]
$m \cdot( m \otimes 1_A) \cdot (\kappa \otimes 1_A \otimes e) = m \cdot( 1_A \otimes m) \cdot (\kappa \otimes 1_A \otimes e)$ 
\item[(8')]
$m \cdot( m \otimes 1_A) \cdot (1_A \otimes \kappa \otimes e) = m \cdot( 1_A \otimes m) \cdot (1_A \otimes \kappa \otimes e)$
\item[(9')]
$m \cdot( m \otimes 1_A) \cdot (e \otimes 1_A \otimes \kappa) = m \cdot( 1_A \otimes m) \cdot (e \otimes 1_A \otimes \kappa)$
{\rouge \item[(10')] $ S_L \cdot \lambda \cdot m \cdot (e \ox \kappa) = \lambda \cdot m \cdot  (e \ox \kappa) \cdot (1_B \ox S_L)$,
\item[(11')] $ \epsilon_B \ox S_R = \lambda \cdot m \cdot  (e \ox \kappa) \cdot (S_R \ox S_R) \cdot (1_B \ox \lambda)\cdot (1_B \ox m) \cdot (1_B \ox e \ox \kappa)\cdot (\Delta \ox 1_X)$,}
\item[(12')] $\lambda$ is a morphism of coalgebras preserving the unit.
\end{itemize} 
\end{definition}

The following definition is inspired by the definition given in \cite{Majid} in the case of associative Hopf algebras.

\begin{definition}\label{def Hopf action} 
Let $X$ and $B$ be Hopf algebras, $\act \colon B \otimes X \rightarrow X$ is an action of Hopf algebras if it is an action of bialgebras such that the following additional conditions are satisfied
% textidote: ignore begin
\begin{equation}\label{ass action}
 \act \cdot (1_B \ox \act) = \act \cdot (m \ox 1_X) ,
\end{equation}
\begin{equation}\label{coass action}
\act \cdot (1_B \ox m) = m \cdot (\act \ox \act) \cdot (1_B \ox {\blue \sigma_{B,X}} \ox 1_X) \cdot (\Delta \ox 1_X \ox 1_X),
\end{equation}
\begin{equation}\label{S_L action}
\act \cdot (1_B \ox S_L) = S_L \cdot \act,
\end{equation}
\begin{equation}\label{S_R action}
\act \cdot (S_R \ox S_R)\cdot (1_B \ox \act) \cdot (\Delta \ox 1_X) = {\blue \epsilon_B} \ox S_R .
\end{equation}
% textidote: ignore end
These conditions can be expressed by the commutativity of the four diagrams below.
% textidote: ignore begin
\begin{center}
\begin{tikzpicture}[descr/.style={fill=white},baseline=(A.base),yscale=0.85] 
\node (A) at (0,0) {$B \ox X$};
\node (B) at (3,0) {$X$};
\node (C) at (3,3) {$B\ox X$};
\node (D) at (0,3) {$B \ox B \ox X$};
\node at (1.5,1.5) {$\eqref{ass action}$};
\path[->,font=\scriptsize]
(A.east) edge node[above] {$\act$} (B.west)
(D.east) edge node[above] {$1_B \ox \act$} (C.west)
(C.south) edge node[descr] {$\act$} (B.north)
(D.south) edge node[descr] {$m \ox 1_X$} (A.north)
;
\end{tikzpicture} 
\qquad
\begin{tikzpicture}[descr/.style={fill=white},baseline=(A.base),yscale=0.85,xscale=1.2] 
\node (A) at (0,0) {$B \ox X$};
\node (B) at (6,0) {$X$};
\node (C) at (6,1.5) {$X^2$};
\node (D) at (6,3) {$(B \ox X)^2$};
\node (E) at (3,3) {$B^2 \ox X^2$};
\node (F) at (0,3) {$B \ox X\ox X$};
\node at (3,1.5) {$\eqref{coass action}$};
\path[->,font=\scriptsize]
(A.east) edge node[above] {$\act$} (B.west)
(F.east) edge node[above] {$\Delta \ox 1_X \ox 1_X$} (E.west)
(E.east) edge node[above] {$1_B \ox {\blue \sigma_{B,X}}  \ox 1_X$} (D.west)
(D.south) edge node[descr] {$\act \ox \act$} (C.north)
(C.south) edge node[descr] {$m$} (B.north)
(F.south) edge node[descr] {$1_B \ox m$} (A.north)
;
\end{tikzpicture} 
\qquad
{\rouge \begin{tikzpicture}[descr/.style={fill=white},baseline=(A.base),yscale=0.85] 
\node (A) at (0,0) {$B \ox X$};
\node (B) at (3,0) {$X$};
\node (C) at (3,3) {$X$};
\node (D) at (0,3) {$B \ox X$};
\node at (1.5,1.5) {$\eqref{S_L action}$};
\path[->,font=\scriptsize]
(A.east) edge node[above] {$\act$} (B.west)
(D.east) edge node[above] {$ \act$} (C.west)
(C.south) edge node[descr] {$S_L$} (B.north)
(D.south) edge node[descr] {$1_B \ox S_L$} (A.north)
;
\end{tikzpicture}
\qquad
\begin{tikzpicture}[descr/.style={fill=white},baseline=(A.base),yscale=0.85,xscale=1.6] 
\node (A) at (0,0) {$X$};
\node (B) at (3,0) {$B \ox X$};
\node (C) at (3,3) {$ B  \ox X$};
\node (C') at (1.5,3) {$B^2 \ox X$};
\node (D) at (0,3) {$B \ox X$};
\node at (1.5,1.5) {$\eqref{S_R action}$};
\path[->,font=\scriptsize]
(B.west) edge node[above] {$\act$} (A.east)
(D.east) edge node[above] {$ \Delta \ox 1_X$} (C'.west)
(C'.east) edge node[above] {$1_B \ox  \act$} (C.west)
(C.south) edge node[descr] {$S_R \ox S_R$} (B.north)
(D.south) edge node[descr] {${\blue \epsilon_B} \ox S_R$} (A.north)
;
\end{tikzpicture}}
% textidote: ignore end
 \end{center}
\end{definition}
{\rouge Note that whenever $S_L = S_R$, the condition \eqref{S_R action} follows from \eqref{ass action} and \eqref{S_L action}.}
{\orange We notice that, when we consider associative Hopf algebras, the conditions \eqref{S_L action} and \eqref{S_R action} are trivially satisfied thanks to the uniqueness of the antipode.} 
%action of Hopf algebras is a \emph{$B$-module Hopf algebra} satisfying also condition \eqref{coco}. In other words, an action of Hopf algebras is an action of bialgebras satisfying the conditions \eqref{ass action_1} and \eqref{ass action_2}. 

We define the map $\Theta \colon B \ox X \rightarrow X \ox B$ as the composition \[\Theta {\blue \coloneqq} (\act \ox 1_B) \cdot (1_B \ox {\blue \sigma_{B,X}}) \cdot (\Delta \ox 1_X).\] We will use this map to obtain shorter computations.
We re-formulate the conditions \eqref{coass action}, \eqref{conditionaction1}, \eqref{ass action} and \eqref{condition action = e} in terms of $\Theta$. These new conditions will help us to prove that the semi-direct product is a Hopf algebra when we construct it with an action as defined above (Definition \ref{def Hopf action}).
\begin{lemma}
Let $\act \colon B \ox X \rightarrow X$ be an action of Hopf algebras, the morphism $\Theta {\blue \coloneqq} (\act \ox 1_B) \cdot (1_B \ox {\blue \sigma_{B,X}}) \cdot (\Delta \ox 1_X)$ satisfies the following conditions
\begin{equation}\label{theta multiplication x}
(m \ox 1_B) \cdot (1_X \ox \Theta) \cdot (\Theta \ox 1_X)= \Theta \cdot (1_B \ox m),
\end{equation} 
 \begin{equation}\label{theta unit B}
\Theta \cdot (u_B \ox 1_X) = 1_X \otimes u_B,
\end{equation}
 \begin{equation}\label{theta multiplication B}
(1_X \ox m) \cdot (\Theta \ox 1_B) \cdot (1_B \ox \Theta)= \Theta \cdot (m \ox 1_X),
\end{equation}
 \begin{equation}\label{theta unit X}
\Theta \cdot (1_B \ox u_X) = u_X \otimes 1_B ,
\end{equation} 
\begin{center}
% textidote: ignore begin
\begin{tikzpicture}[descr/.style={fill=white},baseline=(A.base),yscale=0.85] 
\node (A) at (0,0) {$B \ox X$};
\node (B) at (6,0) {$X \ox B$};
\node (D) at (6,3) {$X \ox X \ox B$};
\node (E) at (3,3) {$X \ox B \ox X$};
\node (F) at (0,3) {$B \ox X\ox X$};
\path[->,font=\scriptsize]
(A.east) edge node[above] {$\Theta$} (B.west)
(F.east) edge node[above] {$\Theta \ox 1_X$} (E.west)
(E.east) edge node[above] {$1_X \ox \Theta$} (D.west)
(D.south) edge node[descr] {$m \ox 1_B$} (B.north)
(F.south) edge node[descr] {$1_B \ox m$} (A.north)
;
\end{tikzpicture} 
\qquad
\begin{tikzpicture}[descr/.style={fill=white},baseline=(C.base),scale=0.85] 
\node (A) at (0,0) {$ X$};
\node (B) at (3,0) {$B \ox X$};
\node (C) at (3,-3) {$X \ox B$};
\path[->,font=\scriptsize]
(A.east) edge node[above] {$u_B \ox 1_X$} (B.west)
(B.south) edge node[descr] {$\Theta$} (C.north)
(A.south east) edge node[descr] {$1_X \ox u_B$} (C.north west)
;
\end{tikzpicture} 
 \quad
\begin{tikzpicture}[descr/.style={fill=white},baseline=(A.base)] 
\node (A) at (0,0) {$B \ox X$};
\node (B) at (6,0) {$X \ox B$};
\node (D) at (6,3) {$X \ox B \ox B$};
\node (E) at (3,3) {$B \ox X \ox B$};
\node (F) at (0,3) {$B \ox B \ox X$};
\path[->,font=\scriptsize]
(A.east) edge node[above] {$\Theta$} (B.west)
(F.east) edge node[above] {$1_B \ox \Theta $} (E.west)
(E.east) edge node[above] {$ \Theta \ox 1_B$} (D.west)
(D.south) edge node[descr] {$1_X \ox m $} (B.north)
(F.south) edge node[descr] {$m \ox 1_X$} (A.north)
;
\end{tikzpicture} 
\qquad
\begin{tikzpicture}[descr/.style={fill=white},baseline=(C.base),scale=0.85] 
\node (A) at (0,0) {$ B$};
\node (B) at (3,0) {$B \ox X$};
\node (C) at (3,-3) {$X \ox B$};
\path[->,font=\scriptsize]
(A.east) edge node[above] {$1_B \ox u_X$} (B.west)
(B.south) edge node[descr] {$\Theta$} (C.north)
(A.south east) edge node[descr] {$u_X \ox 1_B$} (C.north west)
;
\end{tikzpicture}
% textidote: ignore end
 \end{center}
 \end{lemma}

\begin{proof}
We only show the two first equalities since the computations are similar. First, we prove \eqref{theta multiplication x} via the following diagram, where the key part is given by \eqref{coass action}.
\begin{center}
% textidote: ignore begin
\begin{tikzpicture}[descr/.style={fill=white},baseline=(A.base),yscale=0.8,xscale=1.15] 
\node (A) at (0.65,0) {$X \ox B \ox X$};
\node (C) at (4,0) {$X \ox B^2 \ox X$};
\node (D) at (7.5,0) {$(X \ox B)^2$};
\node (E) at (12.5,0) {$X^2 \ox  B$};
\node (F) at (12.5,2) {$X^2 \ox  B$};
\node (G) at (12.5,8) {$X \ox B$};
\node (H) at (10,8) {$B \ox X \ox B$};
\node (I) at (7.5,8) {$B^2 \ox X$};
\node (J) at (5,8) {$B^2  \ox X^2$};
%\node (K) at (2.5,8) {$B \ox X^2$};
\node (L) at (0.65,8) {$B \ox X^2$};
\node (M) at (0.65,4) {$B^2 \ox X^2$};
%\node (N) at (2.5,4) {$B^2 \ox X^2$};
\node (O) at (5,4) {$B^3 \ox X^2$};
\node (P) at (10,6) {$B \ox X^2\ox B$};
\node (Q) at (10,4) {$B^2 \ox X^2 \ox B$};
\node (R) at (0.65,2) {$(B \ox X)^2 $};
%\node (S) at (2.5,2) {$(B \ox X)^2 $};
\node (T) at (5,2) {$B \ox X \ox B^2 \ox X$};
\node (U) at (10,2) {$(B \ox X)^2 \ox B$};
\node at (2,5) {$\eqref{coass comultiplication}$};
\node at (11.5,5.5) {$\eqref{coass action}$};
\draw[->] (A.south east) to[bend right=12]node[descr,scale=0.8] {$ {\blue 1_X \ox \Theta}$} (E.south west);
\draw[->] (L.north) to[bend left=12]node[descr,scale=0.8] {${\blue \Theta  \cdot (1_B \ox m)}$} (G.north west);
\draw[->] (L.south west) to[bend right=10]node[descr,scale=0.8,xshift=3mm,yshift=8mm] {$ {\blue \Theta \ox 1_X}$} (A.north west);
\draw[commutative diagrams/.cd, ,font=\scriptsize]
(F.south) edge[commutative diagrams/equal]  (E.north);
\path[->,font=\scriptsize]
(A.east) edge node[above] {$1_X \ox \Delta \ox 1_X$} (C.west)
(C.east) edge node[above] {$1_X \ox 1_B \ox {\blue \sigma_{B,X}}$} (D.west)
(D.east) edge node[above] {$1_X \ox \act  \ox 1_B$} (E.west)
(R.east) edge node[above] {$1_B \ox 1_X \ox \Delta \ox 1_X$} (T.west)
(T.east) edge node[above] {$1_B \ox 1_X \ox 1_B \ox {\blue \sigma_{B,X}}$} (U.west)
(U.east) edge node[above] {$\act^2  \ox 1_B$} (F.west)
(M.east) edge node[above] {$1_B \ox \Delta \ox 1_X \ox 1_X$} (O.west)
(O.east) edge node[above] {$ 1_B \ox 1_B \ox \sigma_{B,XX}$} (Q.west)
(L.east) edge node[above] {$\Delta \ox 1_X \ox  1_X$} (J.west)
(J.east) edge node[above] {$1_B \ox 1_B \ox m$} (I.west)
(I.east) edge node[above] {$1_B \ox {\blue \sigma_{B,X}}$} (H.west)
(H.east) edge node[above] {$\act \ox 1_B$} (G.west)
(L.south) edge node[descr] {$\Delta \ox 1_X \ox 1_X$} (M.north)
(M.south) edge node[descr] {$1_B \ox {\blue \sigma_{B,X}} \ox 1_X$} (R.north)
(R.south) edge node[descr] {$\act \ox 1_B \ox 1_X$} (A.north)
%(K.south) edge node[descr] {$\Delta \ox 1_X \ox 1_X$} (N.north)
%(N.south) edge node[descr] {$1_B \ox \sigma \ox 1_X$} (S.north)
(J.south) edge node[descr] {$\Delta \ox 1_B \ox 1_X \ox 1_X$} (O.north)
(O.south) edge node[descr] {$1_B \ox \sigma_{BB,X} \ox 1_X$} (T.north)
(P.south) edge node[descr] {$\Delta \ox 1_X \ox 1_X \ox 1_B$} (Q.north)
(Q.south) edge node[descr] {$1_B \ox {\blue \sigma_{B,X}} \ox 1_X \ox 1_B$} (U.north)
(P.south) edge node[descr] {$\Delta \ox 1_X \ox 1_X \ox 1_B$} (Q.north)
(J.south east) edge node[descr] {$1_B \ox \sigma_{B,XX} $} (P.north west)
(P) edge node[descr] {$1_B \ox m \ox 1_B$} (H)
(F.north) edge node[descr] {$m \ox 1_B$} (G.south)
;
\end{tikzpicture}
% textidote: ignore end
 \end{center}
 The condition \eqref{conditionaction1} provides directly the condition \eqref{theta unit B} as we can see in the following diagram
 
\begin{center}
% textidote: ignore begin
\begin{tikzpicture}[descr/.style={fill=white},baseline=(E.base),scale=0.8] 
\node (A) at (-1,0) {$ X$};
\node (B) at (3,0) {$B \ox X$};
\node (C) at (3,-2) {$B^2 \ox X$};
\node (D) at (3,-4) {$B \ox X \ox B$};
\node (E) at (3,-6) {$X \ox B.$};
\node (F) at (-1,-6) {$X $};
\node at (1,-5) {$\eqref{conditionaction1}$};
\node at (1,-1) {$\eqref{u et delta}$};
\draw[commutative diagrams/.cd, ,font=\scriptsize]
(A.south) edge[commutative diagrams/equal]  (F.north);
\path[->,font=\scriptsize]
(A.east) edge node[above] {$u_B \ox 1_X$} (B.west)
(A.south east) edge node[descr] {$u_B \ox 1_X \ox u_B$} (D.north west)
(B.south) edge node[descr] {$\Delta \ox 1_X$} (C.north)
(C.south) edge node[descr] {$1_B \ox {\blue \sigma_{B,X}}$} (D.north)
(D.south) edge node[descr] {$\act \ox 1_B$} (E.north)
(F.east) edge node[above] {$1_X \ox u_B$} (E.west)
;
\end{tikzpicture}
% textidote: ignore end
 \end{center}

The other equalities follow from \eqref{ass action} and \eqref{condition action = e}, the proofs being similar to the ones given above.
%
%\begin{center}% textidote: ignore begin
%\begin{tikzpicture}[descr/.style={fill=white},baseline=(A.base),yscale=0.9] 
%\node (A) at (-1,0) {$ B$};
%\node (B) at (4,0) {$B \ox X$};
%\node (C) at (4,-2) {$B^2 \ox X$};
%\node (D) at (4,-4) {$B \ox X \ox B$};
%\node (E) at (4,-6) {$X \ox B$};
%\node (F) at (-1,-6) {$B $};
%\node (G) at (1,-2) {$B^2 $}; 
%\node (H) at (1,-6) {$X \ox B $};
%\node at (2.25,-5) {$\eqref{condition action = e}$};
%\draw[commutative diagrams/.cd, ,font=\scriptsize]
%(H.east) edge[commutative diagrams/equal]  (E.west)
%(A.south) edge[commutative diagrams/equal]  (F.north);
%\path[->,font=\scriptsize]
%(A.east) edge node[above] {$1_B \ox u_X$} (B.west)
%(A.south east) edge node[descr] {$\Delta$} (G.north west)
%(G.south east) edge node[descr] {$1_B \ox u_X \ox 1_B$} (D.north west)
%(B.south) edge node[descr] {$\Delta \ox 1_X$} (C.north)
%(C.south) edge node[descr] {$1_B \ox \sigma$} (D.north)
%(D.south) edge node[descr] {$\act \ox 1_B$} (E.north)
%(F.east) edge node[above] {$u_X \ox 1_B$} (H.west)
%(G.south) edge node[descr] {$(u_X \cdot \epsilon) \ox 1_B$} (H.north)
%;
%\end{tikzpicture} % textidote: ignore end
% \end{center}
\qed
\end{proof}

Starting from an action of Hopf algebras, we can define a split extension of Hopf algebras 
\begin{equation}\label{split ext Hopf algebra}
% textidote: ignore begin
\begin{tikzpicture}[descr/.style={fill=white},baseline=(A.base)] 
\node (A) at (0,0) {$X \rtimes B$};
\node (B) at (2.5,0) {$B$};
\node (C) at (-2.5,0) {$X$};
\path[dashed,->,font=\scriptsize]
([yshift=2pt]A.west) edge node[above] {$\pi_1$} ([yshift=2pt]C.east)
;
\path[->,font=\scriptsize]
([yshift=-4pt]C.east) edge node[below] {$i_1$} ([yshift=-4pt]A.west)
([yshift=-4pt]A.east) edge node[below] {$\pi_2$} ([yshift=-4pt]B.west)
([yshift=2pt]B.west) edge node[above] {$i_2$} ([yshift=2pt]A.east);
\end{tikzpicture} 
% textidote: ignore end
\end{equation}
where the structure of $X\rtimes B$ is given by  
% textidote: ignore begin
\begin{align*}
m_{X \rtimes B} &= (m \otimes m) \cdot (1_X \ox \Theta \ox 1_B)\\
u_{X \rtimes B} &=u_X \otimes u_B,\\
\Delta_{X \rtimes B} &= (1_X \otimes {\blue \sigma_{X,B}} \otimes 1_B) \cdot (\Delta \otimes \Delta),\\
\epsilon_{X \rtimes B}&= \epsilon_X \otimes \epsilon_B,\\
 S_{X \rtimes B_L} &= \Theta \cdot (S_L \ox S_L) \cdot {\blue \sigma_{X,B}},\\
S_{X \rtimes B_R} &= \Theta \cdot (S_R \ox S_R) \cdot {\blue \sigma_{X,B}}.
\end{align*}
% textidote: ignore end
Thanks to Lemma \ref{obs semi direct bialg}, we already know that \eqref{split ext Hopf algebra} is a split extension of bialgebras. {\green It is also easy to check that $i_1$, $i_2$ and $\pi_2$ are morphisms of Hopf algebras as it is shown in the following diagrams for the left antipode, similar computations work for the right antipode,
\begin{center}
%textidote: ignore begin
\begin{tikzpicture}[descr/.style={fill=white},baseline=(A.base),xscale=1,yscale=1.4] 
\node (A0) at (0,-1) {$ B$};
\node (A2) at (5,-1) {$  X \ox B$};
\node (B0) at (0,0) {$ B$};
\node (B2) at (5,0) {$   B \ox X$};
\node (C0) at (0,1) {$ B$};
\node (C1) at (2.5,1) {$ X \ox B$};
\node (C2) at (5,1) {$   B \ox X$};
\node at (2.5,-0.5) {$\eqref{theta unit X}$};
\draw[commutative diagrams/.cd, ,font=\scriptsize]
(B0.south) edge[commutative diagrams/equal]  (A0.north)
;
%%\draw[->] (A.south east) to[bend right=12]node[descr] {$ (u_X \ox u_B) \cdot (\epsilon \ox \epsilon)$} (D.south west);
\path[->,font=\scriptsize]
(A0.east) edge node[above] {$u_X \ox 1_B $} (A2.west)
(B0.east) edge node[above] {$ 1_B\ox u_X $} (B2.west)
(C0.east) edge node[above] {$u_X \ox 1_B $} (C1.west)
(C1.east) edge node[above] {${\blue \sigma_{X,B}} $} (C2.west)
(C0.south) edge node[descr] {$ S_L $} (B0.north)
(C2.south) edge node[descr] {$ S_L \ox S_L $} (B2.north)
(B2.south) edge node[descr] {$\Theta$} (A2.north)
;
\end{tikzpicture}
\qquad
\begin{tikzpicture}[descr/.style={fill=white},baseline=(A.base),xscale=1,yscale=1.4] 
\node (A0) at (0,-1) {$ X$};
\node (A2) at (5,-1) {$  X \ox B$};
\node (B0) at (0,0) {$ X$};
\node (B2) at (5,0) {$   B \ox X$};
\node (C0) at (0,1) {$ X$};
\node (C1) at (2.5,1) {$ X \ox B$};
\node (C2) at (5,1) {$   B \ox X$};
\node at (2.5,-0.5) {$\eqref{theta unit B}$};
\draw[commutative diagrams/.cd, ,font=\scriptsize]
(B0.south) edge[commutative diagrams/equal]  (A0.north)
;
%%\draw[->] (A.south east) to[bend right=12]node[descr] {$ (u_X \ox u_B) \cdot (\epsilon \ox \epsilon)$} (D.south west);
\path[->,font=\scriptsize]
(A0.east) edge node[above] {$1_X \ox u_B $} (A2.west)
(B0.east) edge node[above] {$ u_B\ox 1_X $} (B2.west)
(C0.east) edge node[above] {$1_X \ox u_B $} (C1.west)
(C1.east) edge node[above] {${\blue \sigma_{X,B}} $} (C2.west)
(C0.south) edge node[descr] {$ S_L $} (B0.north)
(C2.south) edge node[descr] {$ S_L \ox S_L $} (B2.north)
(B2.south) edge node[descr] {$\Theta$} (A2.north)
;
\end{tikzpicture}\qquad
\begin{tikzpicture}[descr/.style={fill=white},baseline=(A.base),xscale=1,yscale=1.4] 
\node (X0) at (0,-2) {$ B$};
\node (X2) at (5,-2) {$  X \ox B$};
\node (A0) at (0,-1) {$ B$};
\node (A2) at (5,-1) {$ B \ox  X \ox B$};
\node (B0) at (0,0) {$ B$};
\node (B2) at (5,0) {$   B^2 \ox X$};
\node (C0) at (0,1) {$ X \ox B$};
\node (C1) at (2.5,1) {$  B \ox X $};
\node (C2) at (5,1) {$  B \ox X $};
\node at (2.5,-1.5) {$\eqref{eps et act}$};
\node at (3.5,-0.5) {$\eqref{counit comultiplication}$};
\draw[commutative diagrams/.cd, ,font=\scriptsize]
(A0.south) edge[commutative diagrams/equal]  (X0.north)
;
%%\draw[->] (A.south east) to[bend right=12]node[descr] {$ (u_X \ox u_B) \cdot (\epsilon \ox \epsilon)$} (D.south west);
\path[->,font=\scriptsize]
(X2.west) edge node[above] {${\blue \epsilon_X} \ox 1_B $} (X0.east)
(A2.west) edge node[above] {${\blue \epsilon_B} \ox {\blue \epsilon_X} \ox 1_B $} (A0.east)
(C0.east) edge node[above] {$ {\blue \sigma_{X,B}} $} (C1.west)
(C1.east) edge node[above] {$S_L \ox S_L $} (C2.west)
(C0.south) edge node[descr] {$ {\blue \epsilon_X} \ox 1_B $} (B0.north)
(B0.south) edge node[descr] {$ S_L $} (A0.north)
(C2.south) edge node[descr] {$ \Delta \ox 1_X $} (B2.north)
(B2.south) edge node[descr] {$1_B \ox {\blue \sigma_{B,X}}$} (A2.north)
(A2.south) edge node[descr] {$\act \ox 1_B$} (X2.north)
(C2.south west) edge node[descr] {$1_B \ox {\blue \epsilon_X}$} (A0.north east)
;
\end{tikzpicture}.
 % textidote: ignore end 
\end{center}
} Furthermore, {\orange thanks to \eqref{coco} {\rouge one can show that} $S_{X \rtimes B_L}$ and $S_{X \rtimes B_R}$ are antihomomorphisms of coalgebras and thanks to \eqref{coco}, \eqref{S_L action},  \eqref{S_R action}, \eqref{theta multiplication x} and \eqref{theta multiplication B} {\rouge one can show that} they are antihomomorphisms of algebras. Moreover, we check that the above construction satisfies the antipode conditions \eqref{antipode} } thanks to the following two diagrams
\begin{center}
% textidote: ignore begin
\begin{tikzpicture}[descr/.style={fill=white},baseline=(A.base),
xscale=1.25,yscale=0.8] 
\node (A) at (1.2,0) {$ X \ox B$};
\node (A') at (3,0) {$ X^2 \ox B$};
\node (B) at (6,0) {$ X^2$};
\node (C) at (10,0) {$X^2 \ox B$};
\node (D) at (12.5,0) {$X \ox B$};
\node (E) at (12.5,2.5) {$X^2 \ox B^2$};
\node (F) at (12.5,5) {$(X \ox B)^2$};
\node (G) at (10,5) {$X \ox B^2 \ox X$}; 
\node (H) at (7.5,5) {$X \ox B^2 \ox X  $};
\node (I) at (5,5) {$(X \ox B)^2$};
\node (J) at (2.5,5) {$X^2 \ox B^2 $};
\node (K) at (1.2,5) {$X \ox B $}; 
\node (L) at (10,2.5) {$X \ox B \ox X$}; 
\node (M) at (7.5,2.5) {$X^2 \ox B  $};
\node (N) at (5,2.5) {$X^2 \ox B^2$};
\node (O) at (2.5,2.5) {$X^2 \ox B^2 $};
\node at (2.5,1.25) {$\eqref{antipode}$};
\node at (9,0.75) {$\eqref{theta unit B}$};
\node at (11.25,1.25) {$\eqref{theta multiplication B}$};
\draw[->] (K.north) to[bend left=15]node[descr,scale=0.8] {${\blue \Delta_{X \rtimes B}}$} ([xshift=-1pt]I.north);
\draw[->] ([xshift=1pt]I.north) to[bend left=15]node[descr,scale=0.8] {$ {\blue 1_X \ox 1_B \ox S_{X \rtimes B_R}}$} (F.north);
%\draw[->] (F.east) to[bend left=19]node[descr] {$ m_{X \rtimes B}$} (D.east);
\draw[commutative diagrams/.cd, ,font=\scriptsize]
(J.south) edge[commutative diagrams/equal]  (O.north)
(K.south) edge[commutative diagrams/equal]  (A.north);
\draw[->] (A.south east) to[bend right=12]node[descr,scale=0.8] {$ {\blue (u_X \ox u_B) \cdot ({\blue \epsilon_X} \ox {\blue \epsilon_B})}$} (D.south west);
\path[->,font=\scriptsize]
(A.east) edge node[above] {$\Delta \ox 1_B $} (A'.west)
(A'.east) edge node[above] {$1_X \ox S_R \ox {\blue \epsilon_B}$} (B.west)
(B.east) edge node[above] {$1_X \ox 1_X \ox u_B$} (C.west)
(C.east) edge node[above] {$m \ox 1_B$} (D.west)
(K.east) edge node[below] {$\Delta^2$} (J.west)
(J.east) edge node[above] {$1_X \ox {\blue \sigma_{X,B}} \ox 1_B$} (I.west)
(I.east) edge node[above] {$1_X \ox 1_B \ox {\blue \sigma_{X,B}}$} (H.west)
(H.east) edge node[above] {$1_X \ox 1_B \ox S_R^2 $} (G.west)
(O.east) edge node[above] {$1_X \ox S_R \ox 1_B \ox S_R$} (N.west)
(N.east) edge node[above] {$1_X \ox 1_X \ox m$} (M.west)
(M.east) edge node[above] {$1_X \ox {\blue \sigma_{X,B}}$} (L.west)
(G.east) edge node[above] {$1_X \ox 1_B \ox \Theta $} (F.west)
(E.south) edge node[descr] {$m \ox m$} (D.north)
(F.south) edge node[descr] {$1_X \ox \Theta \ox 1_B$} (E.north)
(G.south) edge node[descr] {$1_X \ox m \ox 1_X$} (L.north)
(L.south) edge node[descr] {$1_X \ox \Theta$} (C.north)
(B.north east) edge node[descr] {$1_X \ox u_B \ox 1_X $} (L.south west)
;
\end{tikzpicture} 
 \end{center}
\begin{center}
\begin{tikzpicture}[descr/.style={fill=white},baseline=(A.base),
xscale=1.25,yscale=0.8] 
\node (A) at (1.2,0) {$ X \ox B$};
\node (A') at (3,0) {$ X \ox B^2$};
\node (B) at (6,0) {$ B^2$};
\node (C) at (10,0) {$X \ox B^2$};
\node (D) at (12.5,0) {$X \ox B.$};
\node (E) at (12.5,2.5) {$X^2 \ox B^2$};
\node (F) at (12.5,5) {$(X \ox B)^2$};
\node (G) at (10,5) {$B \ox X^2 \ox B $}; 
\node (H) at (7.5,5) {$B \ox X^2 \ox B  $};
\node (I) at (5,5) {$(X \ox B)^2$};
\node (J) at (2.5,5) {$X^2 \ox B^2 $};
\node (K) at (1.2,5) {$X \ox B $}; 
\node (L) at (10,2.5) {$B \ox X \ox B$}; 
\node (M) at (7.5,2.5) {$X \ox B^2  $};
\node (N) at (5,2.5) {$X^2 \ox B^2$};
\node (O) at (2.5,2.5) {$X^2 \ox B^2 $};
\node at (2.5,1.25) {$\eqref{antipode}$};
\node at (9,0.75) {$\eqref{theta unit X}$};
\node at (11.25,1.25) {$\eqref{theta multiplication x}$};
\draw[commutative diagrams/.cd, ,font=\scriptsize]
(J.south) edge[commutative diagrams/equal]  (O.north)
(K.south) edge[commutative diagrams/equal]  (A.north);
\draw[->] (K.north) to[bend left=15]node[descr,scale=0.8] {${\blue \Delta_{X \rtimes B}}$} ([xshift=-1pt]I.north);
\draw[->] ([xshift=1pt]I.north) to[bend left=15]node[descr,scale=0.8] {${\blue S_{X \rtimes B_L} \ox  1_X \ox 1_B}$} (F.north);
%\draw[->] (F.east) to[bend left=19]node[descr] {$ m_{X \rtimes B}$} (D.east);
\draw[->] (A.south east) to[bend right=12]node[descr,scale=0.8] {${\blue (u_X \ox u_B) \cdot ({\blue \epsilon_X} \ox {\blue \epsilon_B})}$} (D.south west);
\path[->,font=\scriptsize]
(A.east) edge node[above] {$1_X \ox \Delta $} (A'.west)
(A'.east) edge node[above] {${\blue \epsilon_X} \ox S_L \ox 1_B$} (B.west)
(B.east) edge node[above] {$u_X \ox 1_B \ox 1_B$} (C.west)
(C.east) edge node[above] {$1_X \ox m$} (D.west)
(K.east) edge node[below] {$\Delta^2$} (J.west)
(J.east) edge node[above] {$1_X \ox {\blue \sigma_{X,B}} \ox 1_B$} (I.west)
(I.east) edge node[above] {${\blue \sigma_{X,B}} \ox 1_X \ox 1_B $} (H.west)
(H.east) edge node[above] {$ S_L^2 \ox 1_X \ox 1_B $} (G.west)
(O.east) edge node[above] {$S_L \ox 1_X \ox S_L \ox 1_B$} (N.west)
(N.east) edge node[above] {$m \ox 1_B \ox 1_B$} (M.west)
(M.east) edge node[above] {$ {\blue \sigma_{X,B}} \ox 1_B$} (L.west)
(G.east) edge node[above] {$\Theta \ox 1_X \ox 1_B  $} (F.west)
(E.south) edge node[descr] {$m \ox m$} (D.north)
(F.south) edge node[descr] {$1_X \ox \Theta \ox 1_B$} (E.north)
(G.south) edge node[descr] {$1_B \ox m \ox 1_B$} (L.north)
(L.south) edge node[descr] {$\Theta \ox 1_B $} (C.north)
(B.north east) edge node[descr] {$1_B \ox u_X \ox 1_B $} (L.south west)
;
\end{tikzpicture}
% textidote: ignore end
 \end{center}

Moreover, the conditions \eqref{theta multiplication B} and \eqref{theta multiplication x} imply that 
% textidote: ignore begin
\begin{tikzpicture}[descr/.style={fill=white},baseline=(A.base)] 
\node (A) at (0,0) {$X \rtimes B$};
\node (B) at (2.5,0) {$B$};
\node (C) at (-2.5,0) {$X$};
\path[dashed,->,font=\scriptsize]
([yshift=2pt]A.west) edge node[above] {$\pi_1$} ([yshift=2pt]C.east)
;
\path[->,font=\scriptsize]
([yshift=-4pt]C.east) edge node[below] {$i_1$} ([yshift=-4pt]A.west)
([yshift=-4pt]A.east) edge node[below] {$\pi_2$} ([yshift=-4pt]B.west)
([yshift=2pt]B.west) edge node[above] {$i_2$} ([yshift=2pt]A.east);
\end{tikzpicture} 
% textidote: ignore end
 satisfies the condition $(9')$ of the Definition \ref{definition split extenion Hopf algebras} as it is shown in the diagram below.
\begin{center}
%\begin{figure}[H]
%\centering
% textidote: ignore begin
\begin{tikzpicture}[descr/.style={fill=white},baseline=(A.base),
xscale=1.5,yscale=1.1] 
\node (A) at (0.5,-1) {$ (B \ox X)^2$};
\node (A') at (3.5,-1) {$  (X \ox B)^3$};
\node (B) at (7,-1) {$ X^2 \ox B^2 \ox X \ox B $};
\node (C) at (9.5,-1) {$(X \ox B)^2$};
\node (D) at (9.5,0.5) {$X^2 \ox B^2$};
\node (E) at (9.5,3) {$X \ox B$};
\node (F) at (9.5,5.5) {$X^2 \ox B^2$};
\node (G) at (9.5,7) {$(X \ox B)^2 $}; 
\node (H) at (7,7) {$X \ox B \ox X^2 \ox B^2  $};
\node (I) at (3.5,7) {$(X \ox B)^3$};
\node (J) at (0.5,7) {$(B \ox X)^2$};
\node (K) at (0.5,5.5) {$(B \ox X)^2 $}; 
\node (L) at (2.5,5.5) {$B \ox X^2 \ox B$}; 
\node (M) at (5,5.5) {$B \ox X \ox B$};
\node (N) at (7.5,5.5) {$X \ox B^2 $};
\node (O) at (0.5,3) {$(B \ox X)^2 $};
\node (P) at (3.5,3) {$(X \ox B)^2$};
\node (Q) at (7,3) {$X^2 \ox B^2 $};
\node (R) at (0.5,0.5) {$(B \ox X)^2 $}; 
\node (S) at (2.5,0.5) {$X \ox B^2 \ox X$}; 
\node (T) at (5,0.5) {$X \ox B \ox X  $};
\node (U) at (7.5,0.5) {$X^2 \ox B$};
\node (X1) at (4,2) {$(X \ox B)^2$};
\node (X2) at (6,2) {$X^2 \ox B^2$};
\node (Y1) at (4,4) {$(X \ox B)^2$};
\node (Y2) at (6,4) {$X^2 \ox B^2$};
\node at (5,-0.25) {$\eqref{unital multiplication}$};
\node at (5,6.25) {$\eqref{unital multiplication}$};
\node at (5,4.75) {$\eqref{theta multiplication x}$};
\node at (5,1.25) {$\eqref{theta multiplication B}$};
\draw[commutative diagrams/.cd, ,font=\scriptsize]
(J.south) edge[commutative diagrams/equal]  (K.north)
(K.south) edge[commutative diagrams/equal]  (O.north)
(O.south) edge[commutative diagrams/equal]  (R.north)
(R.south) edge[commutative diagrams/equal]  (A.north);
\path[->,font=\scriptsize]
(A.east) edge node[above] {$u_X \ox (1_B \ox 1_X)^2 \ox u_B $} (A'.west)
(X1.east) edge node[above] {$1_X \ox \Theta \ox 1_B $} (X2.west)
(Y1.east) edge node[above] {$1_X \ox \Theta \ox 1_B $} (Y2.west)
(A'.east) edge node[above] {$1_X  \ox \Theta  \ox 1_B \ox 1_X \ox 1_B$} (B.west)
(Y2.north) edge node[descr] {$m  \ox (1_B)^2 $} (N.south west)
(S.north) edge node[descr] {$1_X  \ox 1_B \ox \Theta$} (X1.south west)
(X2.south east) edge node[descr] {$(1_X)^2  \ox m$} (U.north)
(L.south east) edge node[descr] {$\Theta \ox 1_X \ox 1_B$} (Y1.north)
(B.east) edge node[above] {$m^2 \ox 1_X \ox 1_B$} (C.west)
(R.east) edge node[above] {$\Theta \ox 1_B \ox 1_X$} (S.west)
(S.east) edge node[above] {$1_X \ox m \ox 1_X$} (T.west)
(T.east) edge node[above] {$1_X \ox \Theta$} (U.west)
(O.east) edge node[above] {$\Theta \ox \Theta $} (P.west)
(P.east) edge node[above] {$ 1_X \ox \Theta \ox 1_B $} (Q.west)
(Q.east) edge node[above] {$m \ox m$} (E.west)
(K.east) edge node[above] {$1_B \ox 1_X \ox \Theta$} (L.west)
(L.east) edge node[above] {$1_B \ox m \ox 1_B$} (M.west)
(M.east) edge node[above] {$\Theta \ox 1_B$} (N.west)
(J.east) edge node[above] {$u_X \ox (1_B \ox 1_X )^2 \ox u_B $} (I.west)
(I.east) edge node[above] {$1_X \ox 1_B \ox 1_X  \ox \Theta  \ox 1_B $} (H.west)
(H.east) edge node[above] {$1_X \ox 1_B \ox m^2$} (G.west)
(F.south) edge node[descr] {$m \ox m$} (E.north)
(G.south) edge node[descr] {$1_X \ox \Theta \ox 1_B$} (F.north)
(D.north) edge node[descr] {$m \ox m$} (E.south)
(C.north) edge node[descr] {$1_X \ox \Theta \ox 1_B$} (D.south)
(U.north east) edge node[descr] {$m \ox 1_B $} (E.south west)
(N.south east) edge node[descr] {$1_X \ox m$} (E.north west)
;
\end{tikzpicture} 
% textidote: ignore end
%\caption{Condition $(9')$}
%\label{Condition (9')}
 \end{center}
{\orange Finally, the condition{\blue s} (10') and (11') hold thanks to \eqref{S_L action}, \eqref{S_R action} and \eqref{act split}, and we can conclude that \eqref{split ext Hopf algebra} is a split extension of Hopf algebras as defined in Definition \ref{definition split extenion Hopf algebras}.}

On the other hand, if we have a split extension of Hopf algebras, we can define an action of Hopf algebras. Thanks to Proposition \ref{prop lamda morph} and condition $(9')$, we can prove two identities which are crucial properties, for our purpose{\blue s}, of a split extension of Hopf algebras.

\begin{lemma}\label{lemma act hopf}
Let 
% textidote: ignore begin
\begin{tikzpicture}[descr/.style={fill=white},baseline=(A.base)] 
\node (A) at (0,0) {$A$};
\node (B) at (2.5,0) {$B$};
\node (C) at (-2.5,0) {$X$};
\path[dashed,->,font=\scriptsize]
([yshift=2pt]A.west) edge node[above] {$\lambda$} ([yshift=2pt]C.east);
\path[->,font=\scriptsize]
([yshift=-4pt]C.east) edge node[below] {$\kappa$} ([yshift=-4pt]A.west)
([yshift=-4pt]A.east) edge node[below] {$\alpha$} ([yshift=-4pt]B.west)
([yshift=2pt]B.west) edge node[above] {$e$} ([yshift=2pt]A.east);
\end{tikzpicture} 
% textidote: ignore end
 be a split extension of Hopf algebras, we have 
% textidote: ignore begin
\begin{equation}\label{mB action} 
\lambda \cdot m \cdot (e \ox \kappa) \cdot (m \ox 1_X) =\lambda \cdot m \cdot (e \ox (\kappa \cdot \lambda)) \cdot (1_B \ox m) \cdot (1_B \ox e \ox \kappa),
\end{equation} 
\begin{equation}\label{mX action}
\lambda \cdot m \cdot (e \ox \kappa) \cdot (1_B \ox m) = m \cdot (\lambda \ox \lambda) \cdot (m \ox m) \cdot (e \otimes \kappa \otimes e \ox \kappa) \cdot (1_B \ox {\blue \sigma_{B,X}} \ox 1_X) \cdot (\Delta \ox 1_X \ox 1_X),
\end{equation}   
\begin{center}
\begin{tikzpicture}[descr/.style={fill=white},yscale=0.9,xscale=1.2] 
\node (A) at (0,0) {$B \ox X$};
\node (B) at (3,0) {$X$};
\node (C) at (3,3) {$B\ox X$};
\node (D) at (0,3) {$B^2 \ox X$};
\node at (1.5,1.5) {$\eqref{mB action}$};
\path[->,font=\scriptsize]
(A.east) edge node[above] {$\lambda \cdot m \cdot (e \ox \kappa)$} (B.west)
(D.east) edge node[above] {$1_B \ox (\lambda \cdot m \cdot (e \ox \kappa))$} (C.west)
(C.south) edge node[descr] {$\lambda \cdot m \cdot (e \ox \kappa)$} (B.north)
(D.south) edge node[descr] {$m \ox 1_X$} (A.north)
;
\end{tikzpicture} 
\qquad
\begin{tikzpicture}[descr/.style={fill=white},yscale=0.9,xscale=1.2] 
\node (A) at (0,0) {$B \ox X$};
\node (B) at (6,0) {$X.$};
\node (C) at (6,1.5) {$X\ox X$};
\node (D) at (6,3) {$(B \ox X)^2$};
\node (E) at (3,3) {$B^2 \ox X^2$};
\node (F) at (0,3) {$B \ox X\ox X$};
\node at (3,1.5) {$\eqref{mX action}$};
\path[->,font=\scriptsize]
(A.east) edge node[above] {$\lambda \cdot m \cdot (e \ox \kappa)$} (B.west)
(F.east) edge node[above] {$\Delta \ox 1_X \ox 1_X$} (E.west)
(E.east) edge node[above] {$1_B \ox {\blue \sigma_{B,X}}  \ox 1_X$} (D.west)
(D.south) edge node[descr] {$(\lambda \cdot m \cdot (e \ox \kappa))^2$} (C.north)
(C.south) edge node[descr] {$m$} (B.north)
(F.south) edge node[descr] {$1_B \ox m$} (A.north)
;
\end{tikzpicture} 
% textidote: ignore end
 \end{center}
\end{lemma}
\begin{proof}
Thanks to Proposition \ref{prop lamda morph}, the result follows as we can check {\blue by means of} the Figures \ref{proof mb action} and \ref{proof mx action}, where we use that $e, \alpha, \kappa$ are morphisms of bialgebras.
\setcounter{figure}{8}
\begin{figure}[b]
\centering
% textidote: ignore begin
\begin{tikzpicture}[descr/.style={fill=white},baseline=(A.base),yscale=1.1,xscale=0.9] 
\node (A) at (0,-1.5) {$ B^2 \ox X$};
\node (A') at (3,-1.5) {$  B \ox X$};
\node (B) at (6,-1.5) {$ A^2 $};
\node (C) at (16.5,-1.5) {$A$};
\node (D) at (16.5,1.5) {$A$};
\node (E) at (16.5,3) {$X $};
\node (F) at (16.5,4.5) {$X$};
\node (G) at (16.5,6) {$A$}; 
\node (H) at (16.5,7.5) {$A^2  $};
\node (I) at (10,7.5) {$B \ox A$};
\node (J) at (5,7.5) {$B \ox A^2$};
\node (K) at (0,7.5) {$B^2 \ox X $}; 
\node (M) at (5,6) {$B \ox A^2$};
\node (N) at (10,6) {$B \ox A$};
\node (O') at (12.5,7.5) {$B \ox X$};
\node (O) at (12.5,6) {$A^2$};
\node (P1) at (0,4.5) {$B^2 \ox X$};
\node (P2) at (2.5,4.5) {$B^2 \ox A^2$};
\node (P3) at (5,4.5) {$B^2 \ox A$};
\node (P4) at (8,4.5) {$X \ox A^2$};
\node (P5) at (12,4.5) {$X \ox A^2$};
\node (P6) at (14,4.5) {$X \ox A$};
\node (P7) at (15.5,4.5) {$X^2$};
\node (Q1) at (5,3) {$A^3$};
\node (Q2) at (10,3) {$A^3$};
\node (Q3) at (12.5,3) {$A^2$};
\node (Q4) at (15,3) {$X^2$};
\node (R) at (0,1.5) {$B^2 \ox X $}; 
\node (S) at (2.5,1.5) {$A^3$}; 
\node (T) at (5,1.5) {$A^2  $};
\node (U) at (3,0) {$A^3$};
\node (V) at (6,0) {$A^2  $};
\node at (4.5,0.75) {$(9')$};
\node at (10,2.25) {$ \eqref{lambda morph of alg}$};
\node at (10,3.75) {$ \eqref{unital multiplication}$};
\node at (10,5.25) {$\eqref{unital multiplication} + \eqref{coass comultiplication}$};
\draw[commutative diagrams/.cd, ,font=\scriptsize]
(K.south) edge[commutative diagrams/equal]  (P1.north)
(P1.south) edge[commutative diagrams/equal]  (R.north)
(R.south) edge[commutative diagrams/equal]  (A.north)
(V.south) edge[commutative diagrams/equal]  (B.north)
(F.south) edge[commutative diagrams/equal]  (E.north)
(D.south) edge[commutative diagrams/equal]  (C.north);
%\draw[->] (A.south east) to[bend right=12]node[descr] {$ (u_X \ox u_B) \cdot (\epsilon \ox \epsilon)$} (D.south west);
\path[->,font=\scriptsize]
(A.east) edge node[above] {$m \ox 1_X $} (A'.west)
(A'.east) edge node[above] {$e  \ox \kappa$} (B.west)
(B.east) edge node[above] {$m $} (C.west)
(R.east) edge node[above] {$e \ox e \ox \kappa$} (S.west)
(S.east) edge node[above] {$1_A \ox m $} (T.west)
(T.east) edge node[above] {$m$} (D.west)
(U.east) edge node[above] {$m \ox 1_A $} (V.west)
(P1.east) edge node[above] {$ \Delta \ox e \ox \kappa $} (P2.west)
(P2.east) edge node[above] {$(1_B)^2 \ox m $} (P3.west)
(P3.east) edge node[above] {$ (u_X \cdot {\blue \epsilon_B}) \ox e \ox 1_A $} (P4.west)
(P4.east) edge node[above] {$ 1_A \ox (e \cdot \alpha) \ox (\kappa \cdot \lambda) $} (P5.west)
(P5.east) edge node[above] {$ 1_X \ox m$} (P6.west)
(P6.east) edge node[above] {$ 1_X  \ox \lambda$} (P7.west)
(P7.east) edge node[above] {$ m$} (F.west)
(Q1.east) edge node[above] {$ 1_A \ox (e \cdot \alpha) \ox (\kappa \cdot \lambda) $} (Q2.west)
(Q2.east) edge node[above] {$ 1_A \ox m$} (Q3.west)
(Q3.east) edge node[above] {$ \lambda  \ox \lambda$} (Q4.west)
(Q4.east) edge node[above] {$ m$} (E.west)
(M.east) edge node[above] {$1_B \ox m$} (N.west)
(N.east) edge node[above] {$e \ox 1_A$} (O.west)
(K.east) edge node[above] {$1_B \ox e \ox \kappa$} (J.west)
(J.east) edge node[above] {$1_B \ox m$} (I.west)
(I.east) edge node[above] {$1_B \ox \lambda $} (O'.west)
(O'.east) edge node[above] {$e \ox \kappa $} (H.west)
(H.south) edge node[descr] {$m $} (G.north)
(G.south) edge node[descr] {$\lambda$} (F.north)
(D.north) edge node[descr] {$\lambda$} (E.south)
(T.north) edge node[descr] {$\Delta \ox 1_A $} (Q1.south)
(P3.south) edge node[descr] {$e \ox e \ox 1_A$} (Q1.north)
(R.south east) edge node[descr] {$e \ox e \ox \kappa$} (U.north west)
(K.south east) edge node[descr] {$1_B \ox e \ox \kappa$} (M.north west)
(O.north east) edge node[descr] {$(e \cdot \alpha) \ox (\kappa \cdot \lambda) $} (H.south west)
;
\end{tikzpicture} 
% textidote: ignore end
\caption{Proof of \eqref{mB action}}
\label{proof mb action}
 \end{figure}
\begin{figure}
\centering
% textidote: ignore begin
\begin{tikzpicture}[descr/.style={fill=white},baseline=(A.base),xscale=1.15,yscale=1.4] 
\node (A1) at (2.5,-1) {$ B^2 \ox X^2$};
\node (A2) at (5,-1) {$  (B \ox X)^2$};
\node (A3) at (10,-1) {$  A^4$};
\node (A4) at (15,-1) {$  A^2$};
\node (B1) at (2.5,1.5) {$B \ox X^2$};
\node (B2) at (5,1.5) {$B^2 \ox X^3$};
\node (B3) at (5,0) {$(B \ox X)^2 \ox X$};
\node (B4) at (7.5,0) {$A^5$};
\node (B5) at (10.5,0) {$A^5$};
\node (B6) at (12.5,0) {$A^3$};
\node (B7) at (15,0) {$A^2$};
\node (B8) at (15,1) {$X^2$};
\node (B9) at (15,2) {$X$};
\node (E1) at (2.5,3) {$B \ox X^2$};
\node (E2) at (7.5,3) {$A^3$};
\node (E3) at (12.5,3) {$A^2$};
\node (E4) at (15,3) {$A$};
%\node (F1) at (0,3) {$B \ox X^2$};
\node (F2) at (7.5,4) {$A^3$};
\node (F3) at (12.5,4) {$A^2$};
\node (F4) at (15,4) {$A$};
\node (G1) at (2.5,5) {$B \ox X^2$};
\node (G2) at (10,5) {$B \ox X$};
\node (G3) at (15,5) {$A^2$};
%\node (G4) at (20,6) {$A$};
\node (C) at (7.5,1.5) {$A^5 $}; 
\node (D) at (10,1.5) {$A^3$}; 
\node (E) at (12.5,1.5) {$A^3$};
\node at (4.5,-0.5) {$(1)+(2)$};
\node at (4,4) {$(9')$};
\node at (13.75,1.5) {$\eqref{lambda morph of alg}$};
\draw[commutative diagrams/.cd, ,font=\scriptsize]
(G1.south) edge[commutative diagrams/equal]  (E1.north)
(E1.south) edge[commutative diagrams/equal]  (B1.north)
(E4.north) edge[commutative diagrams/equal]  (F4.south)
(B7.south) edge[commutative diagrams/equal]  (A4.north)
;
%%\draw[->] (A.south east) to[bend right=12]node[descr] {$ (u_X \ox u_B) \cdot (\epsilon \ox \epsilon)$} (D.south west);
\path[->,font=\scriptsize]
(A1.east) edge node[above] {$1_B \ox {\blue \sigma_{B,X}} \ox 1_X $} (A2.west)
(A2.east) edge node[above] {$(e \ox \kappa)^2 $} (A3.west)
(A3.east) edge node[above] {$m \ox m$} (A4.west)
(B1.east) edge node[above] {$\Delta \ox \Delta \ox 1_X $} (B2.west)
(B2.south) edge node[descr,xshift=-10pt] {$1_B \ox {\blue \sigma_{B,X}} \ox 1_{X^2} $} (B3.north)
(B3.east) edge node[above] {$(e \ox \kappa)^2  \ox \kappa$} (B4.west)
(B4.east) edge node[above] {$1_{A^2}  \ox (e \cdot \alpha)^2 \ox (\kappa \cdot \lambda) $} (B5.west)
(B5.east) edge node[above] {$m \ox m \ox 1_A $} (B6.west)
(B6.east) edge node[above] {$1_A \ox m  $} (B7.west)
(B7.north) edge node[descr] {$ \lambda^2 $} (B8.south)
(B8.north) edge node[descr] {$m $} (B9.south)
(G1.east) edge node[above] {$1_B  \ox m$} (G2.west)
(G2.east) edge node[above] {$e \ox \kappa $} (G3.west)
(F2.east) edge node[above] {$1_A  \ox m$} (F3.west)
(F3.east) edge node[above] {$m $} (F4.west)
(E1.east) edge node[above] {$e \ox \kappa^2   $} (E2.west)
(E2.east) edge node[above] {$m \ox 1_A  $} (E3.west)
(E3.east) edge node[above] {$m $} (E4.west)
%(H.south) edge node[descr] {$m $} (G.north)
(E4.south) edge node[descr] {$\lambda$} (B9.north)
(E3.south) edge node[descr] {$\Delta \ox 1_A$} (E.north)
(E.south) edge node[descr] {$\; \; \; \; \; \;\; \; \; \; \; \; \; \; \; \; \; \; \;\; \; \; \; \; \;\;\; \; \; \; \; \;1_A \ox (e \cdot \alpha) \ox (\kappa \cdot \lambda)$} (B6.north)
(E2.south) edge node[descr] {$\Delta^2 \ox 1_A$} (C.north)
(C.south) edge node[descr] {$1_A \ox {\blue \sigma_{A,A}} \ox 1_{A^2} \; \; \; \; \; \; \; \; \; \; \; \; \; \; \; \; \; \;$} (B4.north)
(B1.south) edge node[descr] {$\Delta \ox 1_{X^2}$} (A1.north)
(G3.south) edge node[descr] {$m$} (F4.north)
%(D.north) edge node[descr] {$\lambda$} (E.south)
(B4.north east) edge node[descr] {$m \ox m \ox 1_A $} (D.south west)
(G1.south east) edge node[descr] {$e \ox \kappa  \ox \kappa$} (F2.north west)
(D.south east) edge node[descr] {$1_A \ox (e \cdot \alpha)  \ox (\kappa \cdot \lambda)$} (B6.north west)
%(R.south east) edge node[descr] {$e \ox e \ox \kappa$} (U.north west)
%(K.south east) edge node[descr] {$1_B \ox e \ox \kappa$} (M.north west)
%(O.north east) edge node[descr] {$(e \cdot \alpha) \ox (\kappa \cdot \lambda) $} (H.south west)
%(V.north east) edge node[descr] {$m $} (D.south west)
;
\end{tikzpicture}
 % textidote: ignore end 
 \caption{Proof of \eqref{mX action}}
 \label{proof mx action}
\end{figure}
\qed
\end{proof}

This lemma implies that the action of bialgebras defined by \eqref{Action bialgebra} ($\act = \lambda \cdot m \cdot (e \ox \kappa)$) satisfies the conditions
\eqref{ass action} and \eqref{coass action}. Hence, this action becomes an action of Hopf algebras {\orange since the conditions \eqref{S_L action} and \eqref{S_R action} are given by the conditions (10') and (11').} 

\begin{remark}\label{action hopf reexpressed}
The construction of the action of Hopf algebras given by $\act = \lambda \cdot m \cdot (e \ox \kappa)$ can be reformulated without $\lambda$, {\red when we compose it by $\kappa$}. Indeed, by pre-composing by $(1_B \ox 1_X  \ox e) \cdot (1_B \ox 1_X \ox S_R)\cdot (1_B \ox {\blue \sigma_{B,X}}) \cdot (\Delta \ox 1_X)$ and post-composing by $m$, the two components of the equality \eqref{lemma1.7}, we obtain the following equality 
% textidote: ignore begin
\begin{equation}\label{re-expres act}
\kappa \cdot \act = m \cdot (m \ox 1_A) \cdot (e \ox \kappa \ox e) \cdot (1_B \ox 1_X \ox S_R) \cdot (1_B \ox {\blue \sigma_{B,X}}) \cdot (\Delta \ox 1_X).
\end{equation}
We notice that thanks to the condition $(8')$, this is equivalent to 
\begin{equation*}
\kappa \cdot \act = m \cdot (1_A \ox m) \cdot (e \ox \kappa \ox e) \cdot (1_B \ox 1_X \ox S_R) \cdot (1_B \ox {\blue \sigma_{B,X}}) \cdot (\Delta \ox 1_X).
\end{equation*}
{\red When the symmetric monoidal category is ${\blue {\sf Vect}_K}$, $\kappa$ can be viewed as an inclusion and \eqref{re-expres act} gives us a way to construct the action without $\lambda$.}

\end{remark}

{\blue For the sake of clarity, we give an explicit description of the morphisms of split extensions and actions of Hopf algebras in $\C$. 

\begin{definition}\label{morph split ext HA}
A morphism of split extensions of Hopf algebras from
% textidote: ignore begin
\begin{tikzpicture}[descr/.style={fill=white},baseline=(A.base),
xscale=0.8] 
\node (A) at (0,0) {$A$};
\node (B) at (2.5,0) {$B$};
\node (C) at (-2.5,0) {$X$};
\path[dashed,->,font=\scriptsize]
([yshift=2pt]A.west) edge node[above] {$\lambda$} ([yshift=2pt]C.east);
\path[->,font=\scriptsize]
([yshift=-4pt]C.east) edge node[below] {$\kappa$} ([yshift=-4pt]A.west)
([yshift=-4pt]A.east) edge node[below] {$\alpha$} ([yshift=-4pt]B.west)
([yshift=2pt]B.west) edge node[above] {$e$} ([yshift=2pt]A.east);
\end{tikzpicture} % textidote: ignore end
  to 
 \\
  % textidote: ignore begin
\begin{tikzpicture}[descr/.style={fill=white},baseline=(A.base),
xscale=0.8] 
\node (A) at (0,0) {$A'$};
\node (B) at (2.5,0) {$B'$};
\node (C) at (-2.5,0) {$X'$};
\path[dashed,->,font=\scriptsize]
([yshift=2pt]A.west) edge node[above] {$\lambda'$} ([yshift=2pt]C.east);
\path[->,font=\scriptsize]
([yshift=-4pt]C.east) edge node[below] {$\kappa'$} ([yshift=-4pt]A.west)
([yshift=-4pt]A.east) edge node[below] {$\alpha'$} ([yshift=-4pt]B.west)
([yshift=2pt]B.west) edge node[above] {$e'$} ([yshift=2pt]A.east);
\end{tikzpicture} 
% textidote: ignore end
is given by three morphisms of Hopf algebras $g\colon B \rightarrow B'$, $v \colon X \rightarrow X'$ and $p \colon A \rightarrow A'$ such that the following diagram commutes in $\C$
\begin{equation}
% textidote: ignore begin
\begin{tikzpicture}[descr/.style={fill=white},baseline=(current  bounding  box.center),xscale=0.7] 
\node (A) at (0,0) {$A$};
\node (B) at (2.5,0) {$B$};
\node (C) at (-2.5,0) {$X$};
\node (A') at (0,-2) {$A'$};
\node (B') at (2.5,-2) {$B'$};
\node (C') at (-2.5,-2) {$X'$};
\path[dashed,->,font=\scriptsize]
([yshift=2pt]A.west) edge node[above] {$\lambda$} ([yshift=2pt]C.east)
([yshift=2pt]A'.west) edge node[above] {$\lambda'$} ([yshift=2pt]C'.east);
\path[->,font=\scriptsize]
(B.south) edge node[right] {$ g$}  (B'.north)
 (C.south) edge node[left] {$ v $}  (C'.north)
(A.south) edge node[left] {$ p$} (A'.north)
([yshift=-4pt]C'.east) edge node[below] {$\kappa'$} ([yshift=-4pt]A'.west)
([yshift=-4pt]A'.east) edge node[below] {$\alpha'$} ([yshift=-4pt]B'.west)
([yshift=2pt]B'.west) edge node[above] {$e'$} ([yshift=2pt]A'.east)
([yshift=-4pt]C.east) edge node[below] {$\kappa$} ([yshift=-4pt]A.west)
([yshift=-4pt]A.east) edge node[below] {$\alpha$} ([yshift=-4pt]B.west)
([yshift=2pt]B.west) edge node[above] {$e$} ([yshift=2pt]A.east);
\end{tikzpicture}
% textidote: ignore end
\end{equation}
\end{definition}

\begin{definition}\label{morph action HA}
Let $\act \colon B \otimes X \rightarrow X$ and $\act' \colon B' \otimes X' \rightarrow X'$ be two actions of Hopf algebras. A morphism between them is defined as a pair of morphisms of Hopf algebras $g \colon B \rightarrow B'$ and $v \colon X \rightarrow X'$ such that  \[v \cdot \act = \act' \cdot (g \ox v).\]
\end{definition}
 
The actions of Hopf algebras (Definition \ref{def Hopf action}) endowed with the morphisms of actions of Hopf algebras (Definition \ref{morph action HA}) form the category ${\blue \mathsf{Act(Hopf_\C)}}$ of actions of Hopf algebras in $\C$. The split extensions of Hopf algebras with the morphisms given by Definition \ref{morph split ext HA} form the category ${\blue \mathsf{SplitExt(Hopf_\C)}}$ of split extensions of Hopf algebras in $\C$.} 

%The actions of Hopf algebras (Definition \ref{def Hopf action}) endowed with the morphisms of actions of bialgebras (Definition \ref{morph action}){\green , where the pair is a pair of morphisms of Hopf algebras,} form the category $\mathsf{Act(Hopf)}$ of actions of Hopf algebras. The split extensions of Hopf algebras with the morphisms given by Definition \ref{morph split ext} {\green , where the triple is a triple of morphisms of Hopf algebras,} form the category $\mathsf{SplitExt(Hopf)}$ of split extensions Hopf algebras. %We denote $\mathsf{Hopf}$, the category of non-associative Hopf algebras.
\begin{theorem}\label{thm Hopg algebra}
{\blue Let $\C$ be a symmetric monoidal category}. There is an equivalence between ${\blue \mathsf{SplitExt(Hopf_\C)}}$ the category of split extensions of Hopf algebras  {\ro in $\C$} and ${\blue \mathsf{Act(Hopf_\C)}}$ the category of actions of Hopf algebras in $\C$.
\end{theorem}
{\green
\begin{proof}
Let $(g,v,p)$ be a morphism in ${\blue \mathsf{SplitExt(Hopf_\C)}}$, then it is clear that $(v,g)$ is a morphism in ${\blue\mathsf{Act(Hopf_\C)}}$. On the other hand, if $(v,g)$ is a morphism of actions of Hopf algebras, the triple $(g,v,v\ox g)$ is a morphism of split extensions of Hopf algebras since $v \ox g$ preserves the antipode. Moreover, the isomorphisms $\phi \coloneqq m \cdot (\kappa \ox e)$ and $\psi \coloneqq (\lambda \ox \alpha) \cdot \Delta$ (in \eqref{canonical}) form an isomorphism in ${\blue \mathsf{SplitExt(Hopf_\C)}}$ since they are morphisms of Hopf algebras, as we can see in the following diagram \begin{center}
% textidote: ignore begin
\begin{tikzpicture}[descr/.style={fill=white},baseline=(A.base),xscale=1.4,yscale=1.3] 
\node (A0) at (0,0) {$ A^2 $};
\node (A1) at (2,0) {$ A$};
\node (A2) at (4,0) {$ A$};
\node (A5) at (10,0) {$ A$};
\node (B2) at (4,1) {$ A^2$};
\node (B5) at (10,1) {$ A^2$};
\node (C0) at (0,2) {$  X \ox B$};
\node (C1) at (2,2) {$ B \ox X$};
\node (C2) at (4,2) {$  B \ox X$};
\node (C3) at (6,2) {$  B^2 \ox X$};
\node (C4) at (8,2) {$  B \ox X \ox B$};
\node (C5) at (10,2) {$X \ox B $};
\node at (7,1) {$\eqref{lemma1.7}$};
%\node at (4,4) {$(9')$};
%\node at (13.75,1.5) {$Prop \; \ref{lambda morph of alg}$};
\draw[commutative diagrams/.cd, ,font=\scriptsize]
(A2.east) edge[commutative diagrams/equal]  (A5.west);
%%%\draw[->] (A.south east) to[bend right=12]node[descr] {$ (u_X \ox u_B) \cdot (\epsilon \ox \epsilon)$} (D.south west);
\path[->,font=\scriptsize]
(A1.east) edge node[above] {$S_L $} (A2.west)
(A0.east) edge node[above] {$m  $} (A1.west)
(C0.east) edge node[above] {${\blue \sigma_{X,B}}$} (C1.west)
(C1.east) edge node[above] {$S_L \ox S_L$} (C2.west)
(C2.east) edge node[above] {$\Delta \ox 1_X$} (C3.west)
(C3.east) edge node[above] {$1_B \ox {\blue \sigma_{B,X}}$} (C4.west)
(C4.east) edge node[above] {$\act \ox 1_B$} (C5.west)
(C0.south) edge node[descr] {$\kappa \ox e $} (A0.north)
(C5.south) edge node[descr] {$\kappa \ox e $} (B5.north)
(C2.south) edge node[descr] {$e \ox \kappa $} (B2.north)
(B5.south) edge node[descr] {$m $} (A5.north)
(B2.south) edge node[descr] {$m $} (A2.north)
;
\end{tikzpicture} 
% textidote: ignore end

 \end{center}
 where the left square commutes since $S_L$ is an antihomomorphism of algebras {\rouge (a similar computation holds for $S_R$)}.
In conclusion, we obtain our statement thanks to the observations about the split extension \eqref{split ext Hopf algebra} and Lemma \ref{lemma act hopf}.
\qed
\end{proof}
}
{\rouge
\begin{remark}
Whenever we consider Hopf algebras such that $S_L = S_R$, the semi-direct product in \eqref{split ext Hopf algebra} also satisfies this property ($S_{X \rtimes B_{L}} = S_{X \rtimes B_{R}}$) and Theorem \ref{thm Hopg algebra} can be restricted to such Hopf algebras. Moreover, the condition $(11')$ in Definition \ref{definition split extenion Hopf algebras} is trivially satisfied thanks to (10') and \eqref{mB action}, and as we have already noticed the condition \eqref{S_R action} in Definition \ref{def Hopf action} always holds.
\end{remark}
}

%{\blue \section{The associative Hopf algebras}}

In the case of associative Hopf algebras, we can define ${\blue \mathsf{SplitExt(AssHopf_\C)}}$ and ${\blue \mathsf{Act(AssHopf_\C)}}$. The actions of associative Hopf algebras are actions of Hopf algebras where the conditions \eqref{S_L action} and \eqref{S_R action}  always hold thanks to the uniqueness of the antipode. A split extension of associative Hopf algebras is the same as in ${\blue \mathsf{SplitExt(Hopf_\C)}}$ where the conditions $(6')$, $(7')$, $(8')$ and $(9')$ become trivial. Moreover, the conditions (10') and (11') are not required, they become properties that any split extension of associative Hopf algebras has.  
\begin{corollary}\label{cor Hopf algebra}
{\blue Let $\C$ be a symmetric monoidal category}. There is an equivalence between ${\blue \mathsf{SplitExt(AssHopf_\C)}}$ the category of split extensions of associative Hopf algebras {\blue in $\C$} and ${\blue \mathsf{Act(AssHopf_\C)}}$ the category of actions of associative Hopf algebras {\blue in $\C$}. 
\end{corollary} 

Let us notice that, since {\green any morphism in ${\blue \mathsf{SplitExt(Hopf_\C)}}$, ${\blue \mathsf{SplitExt(AssHopf_\C)}}$ and $\mathsf{SplitExt(Ass{\blue BiAlg_\C})}$ is a morphism in $\mathsf{SplitExt({\blue BiAlg_\C})}$, the Split Short Five Lemma also holds in these categories. }

In some sense, this result, in the associative Hopf algebras, is similar to a property obtained for the {\em exact cleft sequences} of associative Hopf $K$-algebras (with bijective antipodes) investigated by \cite{AD} (see Lemma 3.2.19).
{\red We would like to emphasize the differences and shared properties between the definition of an exact cleft sequence of associative Hopf algebras and the definition of split extension of associative Hopf algebras (Definition \ref{definition split extenion Hopf algebras}), in the symmetric monoidal category ${\blue {\sf Vect}_K}$ of vector spaces. First, we recall the definition of an exact cleft sequence of associative Hopf algebras \cite{A}.
\begin{definition}\label{def exact}
{\blue A} sequence of morphisms of associative Hopf algebras
\begin{equation}\label{exact}
% textidote: ignore begin
\begin{tikzpicture}[descr/.style={fill=white},baseline=(A.base)] 
\node (A) at (0,0) {$C'$};
\node (B) at (2.5,0) {$B'$};
\node (C) at (-2.5,0) {$A'$};
\path[->,font=\scriptsize]
(C.east) edge node[above] {$\iota$} (A.west)
(A.east) edge node[above] {$\pi$} (B.west);
\end{tikzpicture}
% textidote: ignore end
\end{equation}
is \emph{exact} if 
\begin{itemize}
\item[1)] $\iota$ is injective,
\item[2)] $\pi$ is surjective,
\item[3)] $ker(\pi)= C'\iota(A')^+$ ($ker(\pi)$ is the kernel in ${\blue {\sf Vect}_K}$ and $\iota(A')^+ = \{x \in \iota(A') \mid {\blue \epsilon_{C'}}(x) = 0 \}$) ,
\item[4)] $\iota(A') = LKer(\pi) = \{ x \in C' \mid (\pi \ox 1_{C'}) \cdot \Delta (x) = u_{B'} \ox x \}$.
\end{itemize}
\end{definition}
\begin{definition}
Let \eqref{exact} be an exact sequence of associative Hopf algebras, then the sequence {\blue $(\iota,\pi)$} \begin{equation}\label{cleft}
% textidote: ignore begin
\begin{tikzpicture}[descr/.style={fill=white},baseline=(A.base)] 
\node (A) at (0,0) {$C'$};
\node (B) at (2.5,0) {$B'$};
\node (C) at (-2.5,0) {$A'$};
\path[->,font=\scriptsize]
([yshift=-4pt]C.east) edge node[below] {$\iota$} ([yshift=-4pt]A.west)
([yshift=-4pt]A.east) edge node[below] {$\pi$} ([yshift=-4pt]B.west);
\path[dotted,->,font=\scriptsize]
([yshift=2pt]B.west) edge node[above] {$\chi$} ([yshift=2pt]A.east)
([yshift=2pt]A.west)edge node[above] {$\xi$} ([yshift=2pt]C.east);
\end{tikzpicture} 
% textidote: ignore end
\end{equation}
is cleft if and only if there exist a morphism of $A'$-modules $\xi : C' \rightarrow A'$ (i.e.\ the equality $\xi \cdot m \cdot (\iota \ox 1_{C'})= m \cdot (1_{A'} \ox \xi)$ holds) and a morphism of $B'$-comodules $\chi \colon B' \rightarrow C'$ (i.e.\ the equality $(\pi \ox 1_{C'}) \cdot \Delta \cdot \chi = (1_{B'} \ox \chi) \cdot \Delta $ is satisfied) such that the following two equations hold
\begin{equation}\label{cleft 1}
\xi \cdot \chi = u_{A'} \cdot {\blue \epsilon_{B'}},
\end{equation} 
\begin{equation}\label{cleft 2}
m \cdot ((\iota \cdot \xi) \ox (\chi \cdot \pi) ) \cdot \Delta= 1_{C'}.
\end{equation}
\end{definition}
{\green Remark that in \cite{A,AD} the above definition is not the definition of an exact cleft sequence, but it is equivalent to it by Lemma 3.1.14 in \cite{A},}
It is straightforward to observe that the conditions \eqref{cleft 1} and \eqref{cleft 2} of the sequence \eqref{cleft} are the same as the conditions (2') and (3') in Definition \ref{definition split extenion Hopf algebras}. Moreover, let
% textidote: ignore begin
\begin{tikzpicture}[descr/.style={fill=white},baseline=(A.base)] 
\node (A) at (0,0) {$A$};
\node (B) at (1.5,0) {$B$};
\node (C) at (-1.5,0) {$X$};
\path[dashed,->,font=\scriptsize]
([yshift=2pt]A.west) edge node[above] {$\lambda$} ([yshift=2pt]C.east);
\path[->,font=\scriptsize]
([yshift=-4pt]C.east) edge node[below] {$\kappa$} ([yshift=-4pt]A.west)
([yshift=-4pt]A.east) edge node[below] {$\alpha$} ([yshift=-4pt]B.west)
([yshift=2pt]B.west) edge node[above] {$e$} ([yshift=2pt]A.east);
\end{tikzpicture} 
% textidote: ignore end
 be a split extension of associative Hopf algebras, then $\lambda$ is a $X$-module morphism (thanks to Proposition \ref{prop lamda morph}) and $e$ is a $B$-comodule morphism.

However, there are major differences. Indeed, a split extension of associative Hopf algebras (Definition \ref{definition split extenion Hopf algebras}) is (in general) not exact in the sense of \cite{A}. Conversely, the Hopf algebra morphism $\pi$ in the exact cleft sequence \eqref{cleft} is not a split epimorphism of Hopf algebras, since $\chi$ is neither a morphism of algebras nor a morphism of coalgebras {\rouge (see \cite{AD} for such an example)}. Then, it is clear that one definition does not imply the other and vice versa. Nevertheless, there are sequences of associative Hopf algebras that are exact cleft sequences and split extensions of associative Hopf algebras. For example, any exact sequence of associative Hopf algebras \eqref{exact} such that $\pi$ is a split epimorphism is an example of both definitions (see example 2) in ${\blue {\sf Vect}_K}$).}

To end this paper, we investigate the two main symmetric monoidal categories of interest{\blue :} $\sf Set$ and ${\blue {\sf Vect}_K}$. On the one hand, we specify our results in $\sf Set$.
 \subsection*{Split extensions of Hopf algebras in the category of sets}
\begin{example}
 Any split extension of groups \eqref{group extension} is a split extension of associative Hopf algebras when the symmetric monoidal category is $\mathsf{Set}$. In particular, Corollary \ref{cor Hopf algebra} becomes the well-known equivalence of categories between split extensions of groups and group actions.
\end{example}

\begin{example} \label{exampleloop}
In $\mathsf{Set}$, non-associative Hopf algebras will be structures given by a set $G$, with a non-associative multiplication, a neutral element $1$, left inverses and right inverses such that
\begin{equation}\label{loop}
g^{-1}_Lg = 1 = gg^{-1}_R.
\end{equation} 
In particular, the non-zero octonions \cite{Cayley,Grave} are equipped with a non-associative multiplication satisfying \eqref{loop}. {\blue This structure is quite general. A special case is given by the structure of loops since any loop satisfies \eqref{loop}.}
We can describe what split extensions of this algebraic structure should be in order to be equivalent to actions of such an algebraic structure. Indeed, a split extension should be a split morphism of these algebraic structures 
% textidote: ignore begin
\begin{tikzpicture}[descr/.style={fill=white},baseline=(A.base)] 
\node (A) at (0,0) {$A$};
\node (B) at (2.5,0) {$B$};
\node (C) at (-2.5,0) {$X$};
\path[->,font=\scriptsize]
(C.east) edge node[below] {$\kappa$} (A.west)
([yshift=-4pt]A.east) edge node[below] {$\alpha$} ([yshift=-4pt]B.west)
([yshift=2pt]B.west) edge node[above] {$e$} ([yshift=2pt]A.east);
\end{tikzpicture} 
% textidote: ignore end
 such that the following conditions are satisfied for any $a \in A, b \in B ,x \in X$
\begin{itemize}
\item[(3')] 
{\rouge $(a (e \cdot \alpha) (a^{-1}_R))(e \cdot \alpha) (a) = a$}
\item[(6')] 
$(\kappa(x)e(b))a = \kappa(x)(e(b)a)$
\item[(7')]
$(\kappa(x)a)e(b) = \kappa(x)(ae(b))$ 
\item[(8')]
$a(\kappa(x)e(b)) = (a\kappa(x))e(b)$
\item[(9')]
$(e(b)a)\kappa(x) = e(b)(a\kappa(x))$
\item[(10')]
{\rouge $ e(b^{-1}_R)^{-1}_L = e(b)$
\item[(11')]
$ \Big(e(b^{-1}_R)\Big(e(b^{-1}_R)^{-1}_R\big(x^{-1}_Re(b^{-1}_R)\big)\Big)\Big)e(b^{-1}_R)^{-1}_R = x^{-1}_R $.}
\end{itemize}
The other conditions are trivially satisfied with $\lambda(a) = a(e\cdot \alpha) (a^{-1}_R)$.
\end{example}
On the other hand, in the symmetric monoidal category ${\blue {\sf Vect}_K}$ of vector spaces over a field $K$, we give some particular cases of associative split extensions of Hopf algebras. 

\subsection*{Split extensions of Hopf algebras in the category of vector spaces}

1) We consider a split epimorphism $\alpha$ of associative Hopf $K$-algebras, 
\begin{equation}\label{HKer=LKer}
% textidote: ignore begin
\begin{tikzpicture}[descr/.style={fill=white},baseline=(A.base)] 
\node (A) at (0.4,0) {$A$};
\node (B) at (2.5,0) {$B$};
\node (C) at (-2.5,0) {$HKer(\alpha)$};
\path[->,font=\scriptsize]
(C.east) edge node[above] {$\kappa_\alpha$} (A.west)
%([yshift=2pt]A.west) edge node[above] {$\lambda$} ([yshift=2pt]C.east)
([yshift=-4pt]A.east) edge node[below] {$\alpha$} ([yshift=-4pt]B.west)
([yshift=2pt]B.west) edge node[above] {$e$} ([yshift=2pt]A.east);
\end{tikzpicture} ,
% textidote: ignore end
\end{equation} where $HKer(\alpha)$ is the kernel of $\alpha$ in the category $\mathsf{Hopf}_K$ of Hopf $K$-algebras, $\kappa_\alpha$ stands for the equalizer in ${\blue {\sf Vect}_K}$ of $(1_A \ox u_B \ox 1_A) \cdot \Delta $ and $(1_A \ox \alpha \ox 1_A) \cdot (\Delta \ox 1_A) \cdot \Delta$,
\begin{equation}\label{equalizer}
% textidote: ignore begin 
\begin{tikzpicture}[descr/.style={fill=white},baseline=(A.base)] 
\node (A) at (-1,0) {$A$};
\node (B) at (4,0) {$A \ox B \ox A $};
\node (C) at (-3,0) {$HKer(\alpha)$};
\path[->,font=\scriptsize]
([yshift=-4pt]A.east) edge node[below] {$(1_A \ox \alpha \ox 1_A) \cdot (\Delta \ox 1_A) \cdot \Delta$} ([yshift=-4pt]B.west)
([yshift=0pt]C.east) edge node[above] {$\kappa_\alpha$} ([yshift=0pt]A.west)
([yshift=4pt]A.east) edge node[above] {$(1_A \ox u_B \ox 1_A) \cdot \Delta$} ([yshift=4pt]B.west)
;
\end{tikzpicture}.
% textidote: ignore end 
\end{equation}
We recall that the equalizer in ${\blue {\sf Vect}_K}$ of $f,g \colon A \rightarrow B$ is given by $\{ a \in A \mid f(a) = g(a) \}$. 
We also define the following equalizers 
\begin{center}
% textidote: ignore begin 
\begin{tikzpicture}[descr/.style={fill=white},baseline=(A.base)] 
\node (A) at (-1,0) {$A$};
\node (B) at (3,0) {$B \ox A $};
\node (C) at (-3,0) {$LKer(\alpha)$};
\path[->,font=\scriptsize]
([yshift=-4pt]A.east) edge node[below] {$(\alpha \ox 1_A) \cdot \Delta$} ([yshift=-4pt]B.west)
([yshift=0pt]C.east) edge node[above] {$\kappa_{\alpha,L}$} ([yshift=0pt]A.west)
([yshift=4pt]A.east) edge node[above] {$ u_B \ox 1_A $} ([yshift=4pt]B.west)
;
\end{tikzpicture},
% textidote: ignore end 
\end{center}
\begin{center}
% textidote: ignore begin 
\begin{tikzpicture}[descr/.style={fill=white},baseline=(A.base)] 
\node (A) at (-1,0) {$A$};
\node (B) at (3,0) {$A \ox B $};
\node (C) at (-3,0) {$RKer(\alpha)$};
\path[->,font=\scriptsize]
([yshift=-4pt]A.east) edge node[below] {$(1_A \ox \alpha ) \cdot \Delta$} ([yshift=-4pt]B.west)
([yshift=0pt]C.east) edge node[above] {$\kappa_{\alpha,R}$} ([yshift=0pt]A.west)
([yshift=4pt]A.east) edge node[above] {$ 1_A \ox u_B$} ([yshift=4pt]B.west)
;
\end{tikzpicture}.
% textidote: ignore end 
\end{center}
\begin{proposition}\label{ex Lker=Hker}
A split epimorphism \eqref{HKer=LKer} satisfying the condition $HKer(\alpha) = LKer(\alpha)$ is an extension of associative Hopf algebras (Definition \ref{definition split extenion Hopf algebras}) in the symmetric monoidal category ${\blue {\sf Vect}_K}$.
\end{proposition}
\begin{proof}
%\section{Example of split extension of associative Hopf algebras}\label{ex 3)}
{\red First, we recall that for any morphism $\alpha$  in $\mathsf{ AssHopf}_{K}$ it is well-known that the following conditions are equivalent (see \cite{A}) 
\begin{itemize}
\item $HKer(\alpha) = LKer(\alpha)$,
\item $HKer(\alpha) = RKer(\alpha)$,
\item $LKer(\alpha) = RKer(\alpha)$,
\item $LKer(\alpha)$ is an associative Hopf algebra,
\item $RKer(\alpha)$ is an associative Hopf algebra.
\end{itemize}}
Let
% textidote: ignore begin
\begin{tikzpicture}[descr/.style={fill=white},baseline=(A.base)] 
\node (A) at (0.4,0) {$A$};
\node (B) at (2.5,0) {$B$};
\path[->,font=\scriptsize]
%([yshift=2pt]A.west) edge node[above] {$\lambda$} ([yshift=2pt]C.east)
([yshift=-4pt]A.east) edge node[below] {$\alpha$} ([yshift=-4pt]B.west)
([yshift=2pt]B.west) edge node[above] {$e$} ([yshift=2pt]A.east);
\end{tikzpicture}
% textidote: ignore end
be a split epimorphism of associative Hopf algebras satisfying the condition $HKer(\alpha) = LKer(\alpha)${\blue . Since $A$ is an associative Hopf algebra,} we can define the following section of $\kappa_\alpha \colon HKer(\alpha) \rightarrow A$ 
\[ \lambda = m \cdot (1_A \ox (S \cdot e \cdot \alpha)) \ox \Delta.\]
%\[ \lambda(a) = a_1S((e \cdot \alpha)(a_2)).\]
First, we use the condition  $HKer(\alpha) = RKer(\alpha)$ on the kernel to prove that $\lambda$ factors through $HKer(\alpha)$,
\begin{center}
% textidote: ignore begin
\begin{tikzpicture}[descr/.style={fill=white},baseline=(A.base),xscale=1.8,yscale=1.5] 
\node (A0) at (0,0) {$ A^2$};
\node (A1) at (2,0) {$ A^2$};
%\node (A2) at (4,0) {$ X \ox B \ox X$};
%\node (A3) at (6,0) {$X^2$};
\node (A4) at (8,0) {$A \ox B .$};

\node (B3) at (6,1) {$A^2 \ox B^2$};
%\node (B4) at (8,1) {$X^2$};
%\node (C1) at (2,2) {$ B \ox X \ox B^2 \ox X$};
\node (C1) at (2,2) {$A^3 $};
\node (C2) at (4,2) {$A^4 $};
\node (C3) at (6,2) {$A^4 $};
\node (C4) at (8,2) {$A^2 $};
\node (D1) at (2,3) {$A^3$};
\node (D2) at (4,3) {$A^4$};
\node (E1) at (2,4) {$A^2 $};
\node (E2) at (4,4) {$ A^4$};
\node (E3) at (6,4) {$A^4$};
%\node (E4) at (8,4) {$X \ox B$};
\node (F0) at (0,5) {$A$};
\node (F2) at (4,5) {$A^2$};
\node (F3) at (6,5) {$A^2$};
\node (F4) at (8,5) {$A$};
\node at (3,4) {$\eqref{coass comultiplication}$};
\node at (1,2) {$\eqref{counit comultiplication}$};
\node at (7,4) {$\eqref{m et delta}$};
\node at (3,1) {$\eqref{antipode}$};
\path[->,font=\scriptsize]
(A0.east) edge node[above] {$1_A \ox (S \cdot e \cdot \alpha) $} (A1.west)
(A1.east) edge node[above] {$m \ox u_B$} (A4.west)
(C1.east) edge node[above] {$1_A \ox 1_A \ox \Delta$} (C2.west)
(C2.east) edge node[above] {$(1_A \ox (S \cdot e \cdot \alpha))^2$} (C3.west)
(C3.east) edge node[above] {$m \ox m$} (C4.west)
(F0.east) edge node[above] {$\Delta$} (F2.west)
(F2.east) edge node[above] {$1_A \ox (S \cdot e \cdot \alpha)$} (F3.west)
(F3.east) edge node[above] {$ m$} (F4.west)
(B3.south east) edge node[descr] {$m \ox m$} (A4.north west)
(F0.south) edge node[descr] {$\Delta $} (A0.north)
(F0.south east) edge node[descr] {$\Delta$} (E1.north west)
(E1.south) edge node[descr] {$1_A \ox \Delta$} (D1.north)
(D1.south) edge node[descr] {$1_A \ox  {\blue \sigma_{A,A}}$} (C1.north)
(C1.south) edge node[descr] {$1_A \ox (S \cdot e \cdot \alpha) \ox {\blue \epsilon_A}$} (A1.north)
(F2.south) edge node[descr] {$\Delta \ox \Delta$} (E2.north)
(E2.south) edge node[descr] {$1_A \ox 1_A \ox  {\blue \sigma_{A,A}}$} (D2.north)
(C2.south east) edge node[descr] {$1_A \ox (S \cdot e \cdot \alpha) \ox \alpha \ox (S \cdot \alpha) \; \; \; \; \; \; \; \; \; \; \; \; \; \; \; \;  \; \; \; \; \; $} (B3.north west)
(D2.south) edge node[descr] {$1_A \ox  {\blue \sigma_{A,A}} \ox 1_A $} (C2.north)
(F3.south) edge node[descr] {$\Delta \ox \Delta$} (E3.north)
(E3.south) edge node[descr] {$1_A \ox  {\blue \sigma_{A,A}} \ox 1_A$} (C3.north)
(C3.south) edge node[descr] {$(1_A)^2 \ox \alpha^2$} (B3.north)
(F4.south) edge node[descr] {$\Delta$} (C4.north)
(C4.south) edge node[descr] {$1_A \ox  \alpha$} (A4.north)
;
\end{tikzpicture}
% textidote: ignore end
 \end{center}
%\begin{align*}
%\lambda(a)_1 \otimes \alpha(\lambda(a)_2) &= (a_1 S(e \cdot \alpha(a_2)))_1 \otimes \alpha((a_1 S(e \cdot \alpha(a_2)))_2 )\\
%&= a_1 S(e \cdot \alpha(a_4)) \otimes \alpha(a_2 S(e \cdot \alpha(a_3))) \\
%&=  a_1 S(e \cdot \alpha(a_4)) \otimes \alpha(a_2) S( \alpha(a_3)) \\
%&= a_1 S(e \cdot \alpha(a_2)) \otimes u_B\\
%&= \lambda (a) \otimes u_B.
%\end{align*} 
By using that $\lambda$ factors through $RKer(\alpha) = HKer(\alpha)$, we prove that $\Delta \cdot \lambda = (\lambda \otimes \lambda) \cdot \Delta$ in Figure \ref{lamdba coalgebra}. %(this proof is left to the reader).
Indeed, the central rectangle commutes since $RKer(\alpha) = HKer(\alpha)$, the commutativity of the part $(A)$ is clarified in Figure \ref{DiagramA3} (in the appendix).
\begin{figure}
% textidote: ignore begin
\begin{tikzpicture}[descr/.style={fill=white},baseline=(A.base),xscale=1.5,yscale=1.5] 
\node (A0) at (1,0) {$ A^2$};
\node (A6) at (10,0) {$A^2 $};
\node (B5) at (8,1) {$A^3$};
\node (C1) at (2,2) {$A $};
\node (C2) at (4,2) {$A^2 $};
\node (C3) at (6,2) {$A^3 $};
\node (C4) at (8,2) {$A \ox B \ox A $};
\node (D1) at (2,3) {$A^2$};
\node (D4) at (8,3) {$A^2$};
\node (D6) at (10,3) {$A^2$};
\node (E1) at (2,4) {$A^2 $};
\node (F1) at (2,5) {$A $};
\node (F2) at (4,5) {$A^2 $};
\node (F3) at (6,5) {$A^2 $};
\node (F4) at (8,5) {$A $};
\node (G0) at (1,6) {$ A$};
\node (G6) at (10,6) {$A $};
\node at (4,1) {$(A)$};
\node at (9,2) {$\eqref{unital multiplication}$};
%\node at (0.75,4) {$\eqref{coass comultiplication}$};
%\node at (7,4.5) {$\eqref{counit comultiplication}$};
%\node at (3,0.75) {$\eqref{u et epsilon}$};
%\node at (3,3) {$\eqref{unital multiplication}$};
%\node at (5,3.75) {$\eqref{unital multiplication}$};
%\node at (5.5,0.5) {$\eqref{coass comultiplication}$};
%\node at (7,2) {$\eqref{coass comultiplication}$};
\draw[commutative diagrams/.cd, ,font=\scriptsize]
(G0.south east) edge[commutative diagrams/equal]  (F1.north west)
(D4.east) edge[commutative diagrams/equal]  (D6.west)
(D6.south) edge[commutative diagrams/equal]  (A6.north);
%(B0.south) edge[commutative diagrams/equal]  (A0.north);
%\draw[->] (K.north) to[bend left=15]node[descr] {$ \Delta_{X \rtimes B}$} ([xshift=-1pt]I.north);
%\draw[->] ([xshift=1pt]I.north) to[bend left=15]node[descr] {$S_{X \rtimes B} \ox  1_X \ox 1_B$} (F.north);
%\draw[->] (F4.east) to[bend left=15]node[descr] {$ \act\cdot (1_B \ox \act)$} (A4.east);
%\draw[->] (A0.south east) to[bend right=12]node[descr] {$ \act\cdot (m \ox 1_X)$} (A4.south west);
\path[->,font=\scriptsize]
(A0.east) edge node[above] {$\lambda \ox \lambda $} (A6.west)
(C1.east) edge node[above] {$\Delta$} (C2.west)
(C2.east) edge node[above] {$1_A \ox  \Delta$} (C3.west)
(C3.east) edge node[above] {$1_A \ox \alpha \ox 1_A$} (C4.west)
(F1.east) edge node[above] {$\Delta$} (F2.west)
(F2.east) edge node[above] {$1_A \ox  (S \cdot e \cdot \alpha)$} (F3.west)
(F3.east) edge node[above] {$m$} (F4.west)
(G0.east) edge node[above] {$\lambda$} (G6.west)
(B5.south east) edge node[descr] {$m \ox 1_A$} (A6.north west)
(C4.south) edge node[descr] {$1_A \ox (S \cdot e)  \ox 1_A $} (B5.north)
(G6.south) edge node[descr] {$\Delta$} (D6.north)
(F4.south) edge node[descr] {$\Delta$} (D4.north)
(D4.south) edge node[descr] {$1_A \ox u_B \ox 1_A$} (C4.north)
(F1.south) edge node[descr] {$\Delta $} (E1.north)
(E1.south) edge node[descr] {$1_A \ox (S \cdot e \cdot \alpha)$} (D1.north)
(D1.south) edge node[descr] {$m $} (C1.north)
(G0.south) edge node[descr] {$\Delta $} (A0.north)
;
\end{tikzpicture}
% textidote: ignore end
\caption{$\lambda$ is a morphism of coalgebras}
\label{lamdba coalgebra}
 \end{figure}
The condition $(3')$ is trivially respected.   
%\begin{align*}
%\lambda(a_1)(e \cdot \alpha(a_2)) &= a_{1_1} S(e \cdot \alpha(a_{1_2}))(e \cdot \alpha(a_2))\\
%&= a_1 (e \cdot \alpha(S(a_2)a_3)) \\
%&= a.
%\end{align*}
The condition $(4')$ is also satisfied by this definition of $\lambda$ thanks to the commutativity of Figure \ref{Condition (4')}, where we use that $HKer(\alpha) = RKer(\alpha)$,
\begin{figure}
\centering
% textidote: ignore begin
\begin{tikzpicture}[descr/.style={fill=white},baseline=(A.base),xscale=1.6
,yscale=1.2] 
\node (A0) at (0.75,0) {$ HKer(\alpha) \ox B$};
\node (A1) at (2,0) {$ A$};
%\node (A2) at (4,0) {$ X \ox B \ox X$};
%\node (A3) at (6,0) {$X^2$};
%\node (A4) at (8,0) {$A$};
\node (A5) at (9.5,0) {$A$};
%\node (B0) at (0,1) {$  A$};
%\node (B1) at (2,1) {$A \ox B$};
\node (B2) at (4,1) {$ A \ox B $};
%\node (B3) at (6,1) {$A^3$};
\node (B4) at (8,1) {$A^2$};
\node (C0) at (0.75,2) {$HKer(\alpha) \ox B $};
\node (C1) at (2,2) {$A $};
\node (C2) at (4,2) {$A^2 $};
%\node (C3) at (6,2) {$A^4 $};
\node (C4) at (8,2) {$A^3 $};
%\node (D1) at (2,3) {$A^3$};
%\node (D2) at (4,3) {$A^3$};
%\node (D3) at (6,3) {$A^4$};
\node (D4) at (8,3) {$A^4$};
\node (D5) at (9.5,3) {$A^2$};
%\node (E0) at (0,4) {$ X \ox B$};
\node (E1) at (2,4) {$A^2 \ox B$};
\node (E2) at (4,4) {$ A \ox B \ox A$};
\node (E3) at (6,4) {$A \ox B^2 \ox A$};
\node (E4) at (8,4) {$A^4$};
%\node (F0) at (0,5) {$X \ox B$};
\node (F1) at (2,5) {$A \ox B$};
\node (F2) at (4,5) {$A^2 \ox B^2$};
\node (F3) at (6,5) {$(A \ox B)^2$};
%\node (F4) at (8,5) {$A$};
%\node (G0) at (0,6) {$ X \ox B$};
%\node (G1) at (2,6) {$ A^4$};
\node (G2) at (4,6) {$ A^4$};
%\node (G3) at (6,6) {$X^2$};
\node (G4) at (8,6) {$A^4$};
\node (G5) at (9.5,6) {$A^2$};
\node (H0) at (0.75,7) {$ HKer(\alpha) \ox B$};
%\node (H1) at (2,7) {$ A^4$};
\node (H2) at (4,7) {$ A^2$};
%\node (H3) at (6,7) {$X^2$};
%\node (H4) at (8,7) {$A^4$};
\node (H5) at (9.5,7) {$A$};
\node at (6,0.5) {$\eqref{unital multiplication}$};
\node at (6,1.5) {$\eqref{unital multiplication}$};
\node at (6,3) {$\eqref{antipode}$};
\node at (2,1) {$HKer(\alpha) = RKer(\alpha)$};
\node at (8.75,2) {$(ass)$};
\node at (7,6.5) {$\eqref{m et delta}$};
%\node at (3,3) {$\eqref{unital multiplication}$};
%\node at (5,3.75) {$\eqref{unital multiplication}$};
%\node at (5.5,0.5) {$\eqref{coass comultiplication}$};
%\node at (7,2) {$\eqref{coass comultiplication}$};
\draw[commutative diagrams/.cd, ,font=\scriptsize]
(H0.south) edge[commutative diagrams/equal]  (C0.north)
(C0.south) edge[commutative diagrams/equal]  (A0.north)
(A1.east) edge[commutative diagrams/equal]  (A5.west);
%(F0.south east) edge[commutative diagrams/equal]  (E1.north west)
%(D0.south) edge[commutative diagrams/equal]  (B0.north)
%(B0.south) edge[commutative diagrams/equal]  (A0.north);
%\draw[->] (K.north) to[bend left=15]node[descr] {$ \Delta_{X \rtimes B}$} ([xshift=-1pt]I.north);
%\draw[->] ([xshift=1pt]I.north) to[bend left=15]node[descr] {$S_{X \rtimes B} \ox  1_X \ox 1_B$} (F.north);
%\draw[->] (F4.east) to[bend left=15]node[descr] {$ \act\cdot (1_B \ox \act)$} (A4.east);
%\draw[->] (A0.south east) to[bend right=12]node[descr] {$ \act\cdot (m \ox 1_X)$} (A4.south west);
\path[->,font=\scriptsize]
(A0.east) edge node[above] {$\kappa_\alpha \ox {\blue \epsilon_B} $} (A1.west)
(C0.east) edge node[above] {$\kappa_\alpha \ox {\blue \epsilon_B} $} (C1.west)
(C1.east) edge node[above] {$\Delta$} (C2.west)
(C2.east) edge node[above] {$1_A \ox u_A  \ox (S \cdot e \cdot \alpha)$} (C4.west)
(B2.east) edge node[above] {$1_A \ox  (S \cdot e)$} (B4.west)
(D4.east) edge node[above] {$m  \ox m$} (D5.west)
(E1.east) edge node[above] {$1_A \ox  {\blue \sigma_{A,B}}  $} (E2.west)
(E2.east) edge node[above] {$1_A \ox \Delta \ox 1_A$} (E3.west)
(E3.east) edge node[above] {$1_A \ox e^2 \ox 1_A$} (E4.west)
(F1.east) edge node[above] {$ \Delta \ox \Delta$} (F2.west)
(F2.east) edge node[above] {$1_A \ox  {\blue \sigma_{A,B}}  \ox 1_A$} (F3.west)
(G2.east) edge node[above] {$1_A \ox  {\blue \sigma_{A,A}}  \ox 1_A$} (G4.west)
(G4.east) edge node[above] {$m \ox m$} (G5.west)
(H0.east) edge node[above] {$\kappa_\alpha \ox e$} (H2.west)
(H2.east) edge node[above] {$m$} (H5.west)
(H0.south east) edge node[descr] {$\kappa_\alpha \ox 1_B$} (F1.north)
(E4.south) edge node[descr] {$\; \; \; \; \; (1_A)^2 \ox (S \cdot e \cdot \alpha)^2$} (D4.north)
(F1.south) edge node[descr] {$\Delta \ox 1_B$} (E1.north)
(E3.south east) edge node[left] {$1_A \ox e \ox (S \cdot e) \ox (S \cdot e \cdot \alpha) \; \; \; \; \; $} (D4.north west)
(H2.south) edge node[descr] {$\Delta \ox \Delta $} (G2.north)
(E2.south) edge node[descr] {$1_A \ox {\blue \epsilon_B} \ox 1_A $} (C2.north)
(F3.south) edge node[descr] {$1_A \ox 1_B \ox  {\blue \sigma_{A,B}} $} (E3.north)
(G4.south) edge node[descr] {$(1_A)^2 \ox  {\blue \sigma_{A,A}}$} (E4.north)
(D4.south) edge node[descr] {$1_A \ox m \ox 1_A$} (C4.north)
(C4.south) edge node[descr] {$1_A \ox m $} (B4.north)
(C2.south) edge node[descr] {$1_A \ox \alpha $} (B2.north)
(A1.north east) edge node[descr] {$1_A \ox u_B $} (B2.south west)
(F3.north east) edge node[descr] {$(1_A \ox e)^2 $} (G4.south west)
(H5.south) edge node[descr] {$\Delta $} (G5.north)
(D5.south) edge node[descr] {$m $} (A5.north)
(G5.south) edge node[descr,xshift=-10pt] {$1_A \ox (S \cdot e \cdot \alpha)$} (D5.north)
(B4.south) edge node[descr] {$m$} (A5.north west)
;
\end{tikzpicture} 
\caption{Condition $(4')$}
\label{Condition (4')}
% textidote: ignore end
 \end{figure}
 The last condition $(5')$ is left to the reader, to prove it we use the fact that $\lambda \cdot m \cdot (e \ox \kappa_\alpha)$ factors through $HKer(\alpha)$.
%The morphism $inc$ was omitted in the above computations.\\
%The other conditions are trivially respected 
%\begin{align*}
%\lambda(e(a)) &= e(a)_1 S(e \cdot \alpha (e(a)_2))\\
%&= e(a_1) S(e(a_2)) \\
%&= \epsilon(a).
%\end{align*}
%if $x \in HKer(\alpha) $, then it is clear that $\lambda(x)=x$.\\

To conclude, it is a split extension of associative Hopf algebras. 
\qed
\end{proof}
% In the appendix \ref{ex 3)}, we show that this example is indeed a split extension of associative Hopf algebras in details. 
%Moreover, we notice that the associate action of this split extension is  given by the conjugation $B \ox HKer(\alpha) \rightarrow Hker(\alpha) \colon b \ox k \rightarrow e(b_1)kS(e(b_2))$. 

Notice that this proposition can be extended to any symmetric monoidal category with equalizers that are preserved by all endofunctors on $\mathcal{C}$ of the form $ - \ox X$ and $X \ox -$.

 2) In the symmetric monoidal category $({\blue {\sf Vect}_K},\ox,K)$, an exact sequence of associative Hopf algebras (Definition \ref{def exact}) 
\begin{equation}\label{ex vect}
% textidote: ignore begin
\begin{tikzpicture}[descr/.style={fill=white},baseline=(A.base)] 
\node (A) at (0,0) {$C'$};
\node (B) at (2.5,0) {$B'$};
\node (C) at (-2.5,0) {$A'$};
\path[->,font=\scriptsize]
(C.east) edge node[below] {$\iota$} (A.west)
([yshift=-4pt]A.east) edge node[below] {$\pi$} ([yshift=-4pt]B.west)
([yshift=2pt]B.west) edge node[above] {$e$} ([yshift=2pt]A.east);
\end{tikzpicture} ,
% textidote: ignore end
\end{equation}
such that $\pi $ is a split epimorphism of Hopf algebras, is an exact cleft sequence and a split extension of associative Hopf algebras. Indeed, since the condition 4) in Definition \ref{def exact} is equivalent to the condition $LKer(\pi) = HKer(\pi)$ \cite{A}, this example is a particular case of Proposition \ref{ex Lker=Hker}. {\rouge Due to Definition \ref{def exact} the sequence \eqref{ex vect} has to be isomorphic to the following one, 
\begin{center}
% textidote: ignore begin
\begin{tikzpicture}[descr/.style={fill=white},baseline=(A.base)] 
\node (A) at (0,0) {$C'$};
\node (B) at (3.5,0) {$\frac{C'}{C'(LKer(\pi))^+}$};
\node (C) at (-2.5,0) {$HKer(\pi)$};
\path[->,font=\scriptsize]
(C.east) edge node[above] {$\kappa_\pi$} (A.west)
%([yshift=2pt]A.west) edge node[above] {$\lambda$} ([yshift=2pt]C.east)
([yshift=-4pt]A.east) edge node[below] {$\pi$} ([yshift=-4pt]B.west)
([yshift=2pt]B.west) edge node[above] {$e$} ([yshift=2pt]A.east);
\end{tikzpicture} ,
% textidote: ignore end
\end{center}
where $HKer(\pi) = LKer(\pi)$.}

3) If we consider cocommutative associative Hopf $K$-algebras, then we can drop the condition $HKer(\alpha)= RKer(\alpha)$ in Proposition \ref{ex Lker=Hker}. So any split epimorphism of cocommutative associative Hopf algebras induces a split extension as defined in \ref{definition split extenion Hopf algebras} {\red (and an exact cleft sequence)}. The Corollary \ref{cor Hopf algebra} becomes the well-known equivalence between points over $B$ and $B$-module Hopf algebras \cite{VW}.

\section{{\blue Conclusion}}
 To sum up, we defined the category $\sf SplitExt({\blue BiAlg_\C})$ and proved that this category is equivalent to the category $\sf Act({\blue BiAlg_\C})$.
 Moreover, we proved that a suitable version of the Split Short Five Lemma holds when the split extensions occurring in it belong to the category $\sf SplitExt({\blue BiAlg_\C})$. These results were proved to hold also in the category of Hopf algebras, and we gave some examples of split extensions of Hopf algebras in the categories of sets and of vector spaces. It is worthwhile to observe that any isomorphism $\gamma : A \rightarrow A$ of Hopf algebras determines a split extension in the sense of Definition \ref{definition split extenion Hopf algebras} as indicated in the following diagram
\begin{center}
% textidote: ignore begin
\begin{tikzpicture}[descr/.style={fill=white},baseline=(A.base)] 
\node (A) at (0,0) {$A$};
\node (B) at (2.5,0) {$A$};
\node (C) at (-2.5,0) {$I$};
\path[->,font=\scriptsize]
([yshift=-4pt]C.east) edge node[below] {$u$} ([yshift=-4pt]A.west)
([yshift=2pt]A.west) edge node[above] {${\blue \epsilon_A}$} ([yshift=2pt]C.east)
([yshift=-4pt]A.east) edge node[below] {$\gamma$} ([yshift=-4pt]B.west)
([yshift=2pt]B.west) edge node[above] {$\gamma^{-1}$} ([yshift=2pt]A.east);
\end{tikzpicture}. 
% textidote: ignore end
\end{center} This elementary example motivates the study of internal structures in the context of non-associative bialgebras and Hopf algebras. Indeed, thanks to this example, a discrete reflexive graph
\begin{equation}\label{discreet}
% textidote: ignore begin
\begin{tikzpicture}[descr/.style={fill=white},baseline=(A.base)] 
\node (A) at (0,0) {$A$};
\node (B) at (2.5,0) {$A$};
\node (C) at (-2.5,0) {$I$};
\path[->,font=\scriptsize]
([yshift=-4pt]C.east) edge node[below] {$u$} ([yshift=-4pt]A.west)
([yshift=2pt]A.west) edge node[above] {$ {\blue \epsilon_A} $} ([yshift=2pt]C.east)
([yshift=-8pt]A.east) edge node[below] {$1_A$} ([yshift=-8pt]B.west)
([yshift=0pt]B.west) edge node[descr] {$1_A$} ([yshift=0pt]A.east)
([yshift=6pt]A.east) edge node[above] {$1_A$} ([yshift=6pt]B.west);
\end{tikzpicture} ,
% textidote: ignore end
\end{equation}
 is a reflexive graph in $\sf Hopf_\C$ such that {\blue its ``legs'' are} in $\sf SplitExt(Hopf_\C)$. Such an internal structure is different from the internal structure called \emph{pre-cat$^1$-Hopf algebra} in \cite{Vilaboa}. Indeed, the discrete reflexive graph \eqref{discreet} is not a pre-cat$^1$-Hopf algebra without asking that $A$ is cocommutative.
 The example \eqref{discreet} suggests that the adequate internal notion corresponding to a precrossed module of Hopf Algebras (as defined in \cite{Majid}) is the one of a reflexive graph such that one of the two ``legs" is a split extension of Hopf algebras.  In {\blue a} forthcoming paper, we will construct an equivalence of categories between these two structures, and we will investigate the equivalence of categories between Hopf crossed modules (as defined in \cite{Majid}) and internal structures that we will call \emph{cat$^1$-Hopf algebras}. Similarly {\blue to} what we did in this paper we will work with non-associative bialgebras, associative bialgebras, non-associative Hopf algebras and associative Hopf algebras in any symmetric monoidal category.

%4) The split extension of bialgebras of  \ref{example iso} is a split extension of Hopf algebras if $A$ is a Hopf algebra. The associated action is given by %the trivial one, $A \otimes K \rightarrow K : a \otimes k \rightarrow \epsilon(a)k$

% textidote: ignore begin

\textit{Data sharing not applicable to this article as no datasets were generated or analysed during the current study. }

% textidote: ignore end

%

\appendix
\section{Appendix}

This appendix contains {\blue five figures} given below. The monoidal product is denoted by juxtaposition. {\blue Figure \ref{DiagramA3} is used in the proof of Proposition \ref{ex Lker=Hker}.} By combining the diagrams of Figure \ref{sm bialgebra 1} and Figure \ref{sm bialgebra 2}, we show that the structure of $X \rtimes B$ as defined in \eqref{semi-direct product} gives a bialgebra structure. Thanks to the commutativity of the diagram of Figure \ref{sm ass}, we  can conclude that whenever $X$ and $B$ are associative bialgebras and \eqref{ass action_1} and \eqref{ass action_2} are satisfied $m_{X\rtimes B}$ as defined in \eqref{semi-direct product} is associative, which is a part of the proof of Lemma \ref{ass bialgebra}.  Finally, the commutativity of Figure \ref{combinationABC} allows one to prove Proposition \ref{prop lamda morph}.
\setcounter{figure}{12}
\begin{figure}[h]
 \centering
% textidote: ignore begin
\begin{tikzpicture}[descr/.style={fill=white},baseline=(A.base),
xscale=1.7,yscale=1.8] 
\node (A0) at (2.5,0) {$ A^4$};
%\node (A1) at (2,0) {$ A$};
\node (A2) at (4,0) {$ A^4$};
%\node (A3) at (6,0) {$X^2$};
%\node (A4) at (8,0) {$A$};
%\node (A5) at (10,0) {$A$};
\node (A6) at (11,0) {$A^2$};
%\node (B0) at (0,1) {$  A$};
%\node (B1) at (2,1) {$A \ospace B$};
\node (B2) at (4,1) {$ A ^4$};
\node (B3) at (6,1) {$A^5$};
\node (B4) at (8,1) {$A^5$};
\node (B5) at (9.5,1) {$A^3$};
%\node (C0) at (0.75,2) {$HKer(\alpha) \ospace B $};
%\node (C1) at (2,2) {$A $};
%\node (C2) at (4,2) {$A^2 $};
%\node (C3) at (6,2) {$A^4 $};
\node (C4) at (8,2) {$A^6 $};
\node (C5) at (9.5,2) {$A^4 $};
\node (C6) at (11,2) {$A^3 $};
%\node (D1) at (2,3) {$A^3$};
\node (D2) at (4,3) {$A^4$};
\node (D3) at (6,3) {$A^6$};
\node (D4) at (8,3) {$A^6$};
\node (D5) at (9.5,3) {$A^4$};
\node (E0) at (2.5,4) {$ A^4$};
%\node (E1) at (2,4) {$A^2 \ospace B$};
%\node (E2) at (4,4) {$ A \ospace B \ospace A$};
%\node (E3) at (6,4) {$A \ospace B^2 \ospace A$};
%\node (E4) at (8,4) {$A^4$};
\node (E5) at (9.5,4) {$A \ospace B^2 \ospace A$};
\node (E6) at (11,4) {$A \ospace B \ospace A$};
%\node (F0) at (0,5) {$X \ospace B$};
%\node (F1) at (2,5) {$A \ospace B$};
\node (F2) at (4,5) {$A^5$};
\node (F3) at (6,5) {$A^6$};
\node (F4) at (8,5) {$A^6$};
\node (F5) at (9.5,5) {$A^6$};
\node (F6) at (11,5) {$A^3$};
%\node (G0) at (0,6) {$ X \ospace B$};
%\node (G1) at (2,6) {$ A^4$};
%\node (G2) at (4,6) {$ A^4$};
\node (G3) at (6,6) {$X^6$};
\node (G4) at (8,6) {$A^6$};
\node (G5) at (9.5,6) {$A^6$};
%\node (H0) at (0.75,7) {$ HKer(\alpha) \ospace B$};
\node (H1) at (3,7) {$ A^3$};
\node (H2) at (4,7) {$ A^3$};
\node (H3) at (6,7) {$A^4$};
\node (H4) at (8,7) {$A^4$};
\node (H6) at (11,7) {$A^2$};
%\node (I0) at (0.75,7) {$ HKer(\alpha) \ospace B$};
%\node (I1) at (2,8) {$ A^3$};
\node (I2) at (4,8) {$ A^4$};
\node (I3) at (6,8) {$A^4$};
\node (I4) at (8,8) {$A^4$};
%\node (I6) at (12,8) {$A^2$};
\node (J0) at (2.5,9) {$A$};
\node (J1) at (3,9) {$A^2$};
%\node (J2) at (4,9) {$A^5$};
%\node (J3) at (6,9) {$A^6$};
\node (J4) at (8,9) {$A^2$};
%\node (J5) at (10,9) {$A^6$};
\node (J6) at (11,9) {$A$};
\node at (10,8) {$\eqref{m et delta}$};
\node at (10,1.5) {$(ass)$};
\node at (4,7.5) {$\eqref{coass comultiplication}$};
\node at (3,6) {$\eqref{counit comultiplication}$};
\node at (6,0.5) {$\eqref{unital multiplication}$};
\node at (5,3) {$\eqref{antipode}$};
%\node at (8.75,2) {$(ass)$};
%\node at (7,6.5) {$\eqref{m et delta}$};
%\node at (3,3) {$\eqref{unital multiplication}$};
%\node at (5,3.75) {$\eqref{unital multiplication}$};
%\node at (5.5,0.5) {$\eqref{coass comultiplication}$};
%\node at (7,2) {$\eqref{coass comultiplication}$};
\draw[commutative diagrams/.cd, ,font=\scriptsize]
(B2.south) edge[commutative diagrams/equal]  (A2.north);
%(C0.south) edge[commutative diagrams/equal]  (A0.north)
%(A1.east) edge[commutative diagrams/equal]  (A5.west);
%(F0.south east) edge[commutative diagrams/equal]  (E1.north west)
%(D0.south) edge[commutative diagrams/equal]  (B0.north)
%(B0.south) edge[commutative diagrams/equal]  (A0.north);
%\draw[->] (K.north) to[bend left=15]node[descr] {$ \Delta_{X \rtimes B}$} ([xshift=-1pt]I.north);
%\draw[->] ([xshift=1pt]I.north) to[bend left=15]node[descr] {$S_{X \rtimes B} \ospace  1_X \ospace 1_B$} (F.north);
%\draw[->] (F4.east) to[bend left=15]node[descr] {$ \act\cdot (1_B \ospace \act)$} (A4.east);
%\draw[->] (A0.south east) to[bend right=12]node[descr] {$ \act\cdot (m \ospace 1_X)$} (A4.south west);
\path[->,font=\scriptsize]
(A0.east) edge node[above] {$(1_A \ospace (S \cdot e \cdot \alpha))^2 $} (A2.west)
(A2.east) edge node[above] {$m \ospace m $} (A6.west)
(B2.east) edge node[above] {$(1_A)^2 \ospace u_A \ospace  (1_A)^2$} (B3.west)
(B3.east) edge node[above] {$\sigma_{A^2,A}  \ospace  (1_A)^2$} (B4.west)
(B4.east) edge node[above] {$m \ospace 1_A \ospace m  $} (B5.west)
(C4.east) edge node[above] {$m \ospace (1_A)^2 \ospace m  $} (C5.west)
(C5.east) edge node[above] {$1_A \ospace m \ospace 1_A $} (C6.west)
(D3.east) edge node[above] {$1_A \ospace  {\blue \sigma_{A,A}} \ospace (1_A)^3$} (D4.west)
(D4.east) edge node[above] {$m \ospace (1_A)^2 \ospace m$} (D5.west)
(E5.east) edge node[below] {$ 1_A \ospace m \ospace 1_A$} (E6.west)
(F2.east) edge node[above] {$(1_A)^2 \ospace \Delta  \ospace (1_A)^2$} (F3.west)
(F4.east) edge node[above] {$(1_A \ospace (S \cdot e \cdot \alpha))^3  $} (F5.west)
(F5.east) edge node[above] {$m \ospace m \ospace m$} (F6.west)
(G3.east) edge node[above,text width=2.35cm] {$\; \;(1_A)^2 \ospace (S \cdot e \cdot \alpha)^2 \ospace 1_A \ospace (S \cdot e \cdot \alpha)$} (G4.west)
(H1.east) edge node[above] {$1_A \ospace  {\blue \sigma_{A,A}}$} (H2.west)
(H2.east) edge node[above] {$(1_A)^2 \ospace \Delta$} (H3.west)
(H3.east) edge node[above] {$(1_A \ospace (S \cdot e \cdot \alpha))^2$} (H4.west)
(I2.east) edge node[above] {$(1_A)^2 \ospace  {\blue \sigma_{A,A}}$} (I3.west)
(I3.east) edge node[above] {$(1_A)^2 \ospace (S \cdot e \cdot \alpha)^2$} (I4.west)
(H4.east) edge node[above] {$m \ospace m$} (H6.west)
(J4.east) edge node[above] {$m$} (J6.west)
(J1.east) edge node[above] {$1_A \ospace (S \cdot e \cdot \alpha)$} (J4.west)
(J0.east) edge node[above] {$\Delta$} (J1.west)
(J0.south) edge node[descr] {$\Delta$} (E0.north)
(E0.south) edge node[descr,xshift=5pt] {$\Delta \ospace \Delta$} (A0.north)
(J1.south) edge node[descr] {$\Delta \ospace 1_A$} (H1.north)
(J1.south east) edge node[descr] {$\Delta \ospace \Delta$} (I2.north west)
(F2.south) edge node[descr] {$(1_A)^2  \ospace {\blue \epsilon_A} \ospace (1_A)^2 $} (D2.north)
(D2.south) edge node[descr] {$(1_A \ospace (S \cdot e \cdot \alpha))^2$} (B2.north)
(D3.south) edge node[descr] {$(1_A)^2 \ospace m \ospace 1_A$} (B3.north)
(F3.south) edge node[descr,text width=3cm] {$1_A \ospace (S \cdot e \cdot \alpha) \ospace (e \cdot \alpha) \ospace (S \cdot e \cdot \alpha)  \ospace 1_A \ospace (S \cdot e \cdot \alpha)$} (D3.north)
(H3.south) edge node[descr] {$\Delta^2 \ospace (1_A)^2 $} (G3.north)
(I3.south) edge node[descr] {$1_A \ospace  {\blue \sigma_{A,A}} \ospace 1_A $} (H3.north)
(I4.south) edge node[descr] {$1_A \ospace  {\blue \sigma_{A,A}} \ospace 1_A $} (H4.north)
(G4.south) edge node[descr] {$1_A \ospace  {\blue \sigma_{A,A}} \ospace (1_A)^3 $} (F4.north)
(F4.south) edge node[descr,text width=3cm] {$1_A \ospace (S \cdot e \cdot \alpha)^2  \ospace (S \cdot S \cdot e \cdot \alpha)  \ospace 1_A \ospace (S \cdot e \cdot \alpha)$} (D4.north)
(D4.south) edge node[descr] {$(1_A)^2 \ospace  {\blue \sigma_{A,A}} \ospace (1_A)^2 $} (C4.north)
(C4.south) edge node[descr] {$1_A \ospace m \ospace (1_A)^3 $} (B4.north)
(G5.south) edge node[descr] {$1_A \ospace  {\blue \sigma_{A,A}} \ospace (1_A)^3 $} (F5.north)
(D5.south) edge node[descr] {$1_A \ospace  {\blue \sigma_{A,A}} \ospace 1_A $} (C5.north)
(E5.south) edge node[descr] {$1_A \ospace (S \cdot e)^2 \ospace 1_A $} (D5.north)
(F5.south) edge node[descr] {$m \ospace (\alpha)^2 \ospace m$} (E5.north)
(J6.south) edge node[descr] {$\Delta$} (H6.north)
(H6.south) edge node[descr] {$\Delta \ospace 1_A$} (F6.north)
(F6.south) edge node[descr,xshift=-10pt] {$1_A \ospace \alpha \ospace 1_A$} (E6.north)
(E6.south) edge node[descr,xshift=-15pt] {$1_A \ospace (S \cdot e) \ospace 1_A$} (C6.north)
(C6.south) edge node[descr] {$m \ospace 1_A$} (A6.north)
(B5.south east) edge node[descr] {$m \ospace 1_A$} (A6.north west)
(H4.south east) edge node[descr] {$\Delta \ospace \Delta \ospace (1_A)^2$} (G5.north west)
(H3.south west) edge node[descr] {$\Delta \ospace S \ospace (1_A)^2$} (F2.north east)
(G4.south west) edge node[left] {$(1_A)^2 \ospace S^2 \ospace (1_A)^2$} (F3.north east)
(J4.south) edge node[descr] {$\Delta \ospace \Delta $} (I4.north)
%(C2.south) edge node[descr] {$1_A \ospace \alpha $} (B2.north)
%(A1.north east) edge node[descr] {$1_A \ospace u_B $} (B2.south west)
%(F3.north east) edge node[descr] {$(1_A \ospace e)^2 $} (G4.south west)
%(H5.south) edge node[descr] {$\Delta $} (G5.north)
%(D5.south) edge node[descr] {$m $} (A5.north)
%(G5.south) edge node[descr] {$1_A \ospace (S \cdot e \cdot \alpha)$} (D5.north)
%(B4.south) edge node[descr] {$m$} (A5.north west)
;
\end{tikzpicture} 
% textidote: ignore end
\caption{Commutation of the diagram $(A)$}
\label{DiagramA3}
 \end{figure}
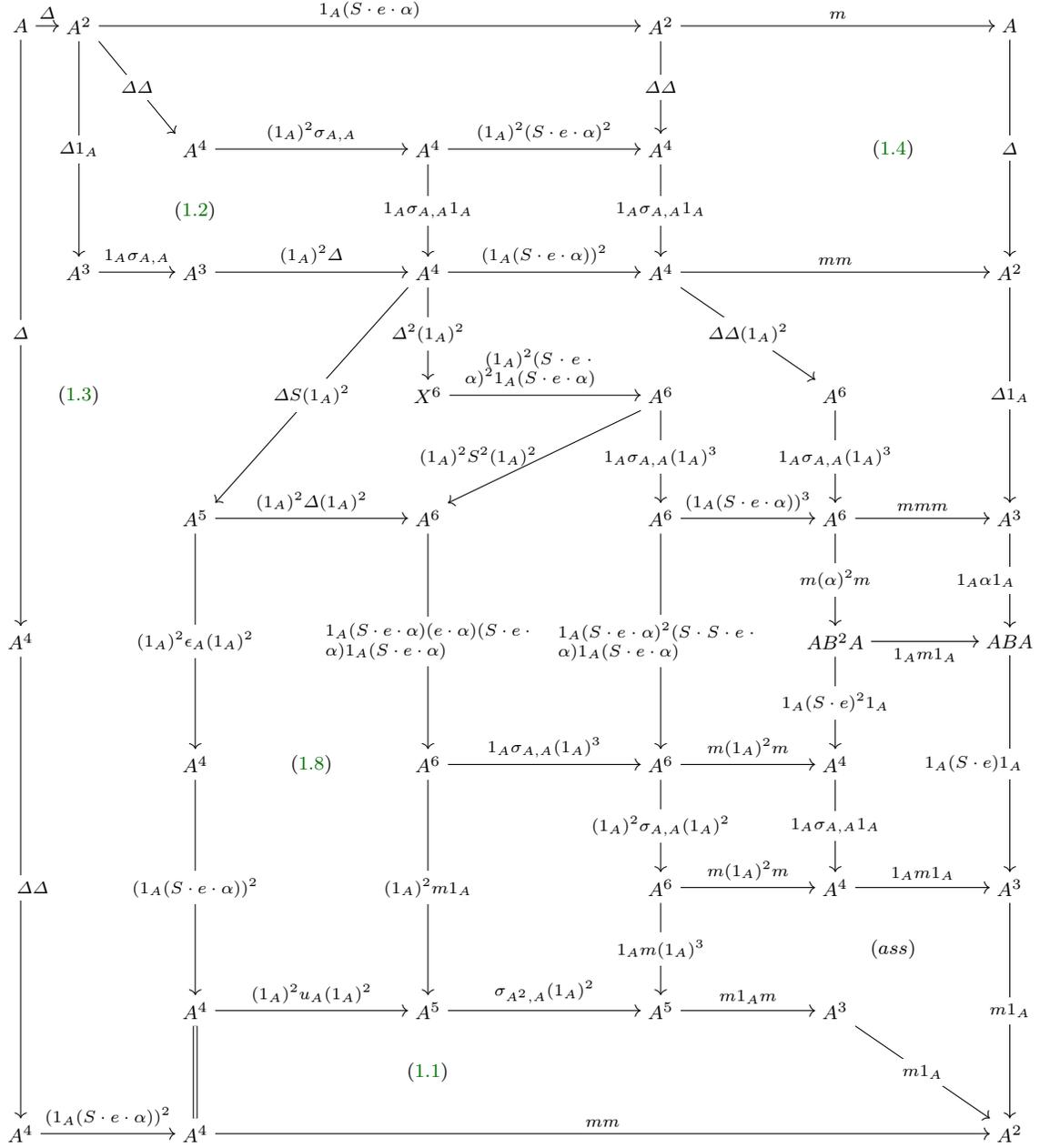 
 
\begin{landscape}
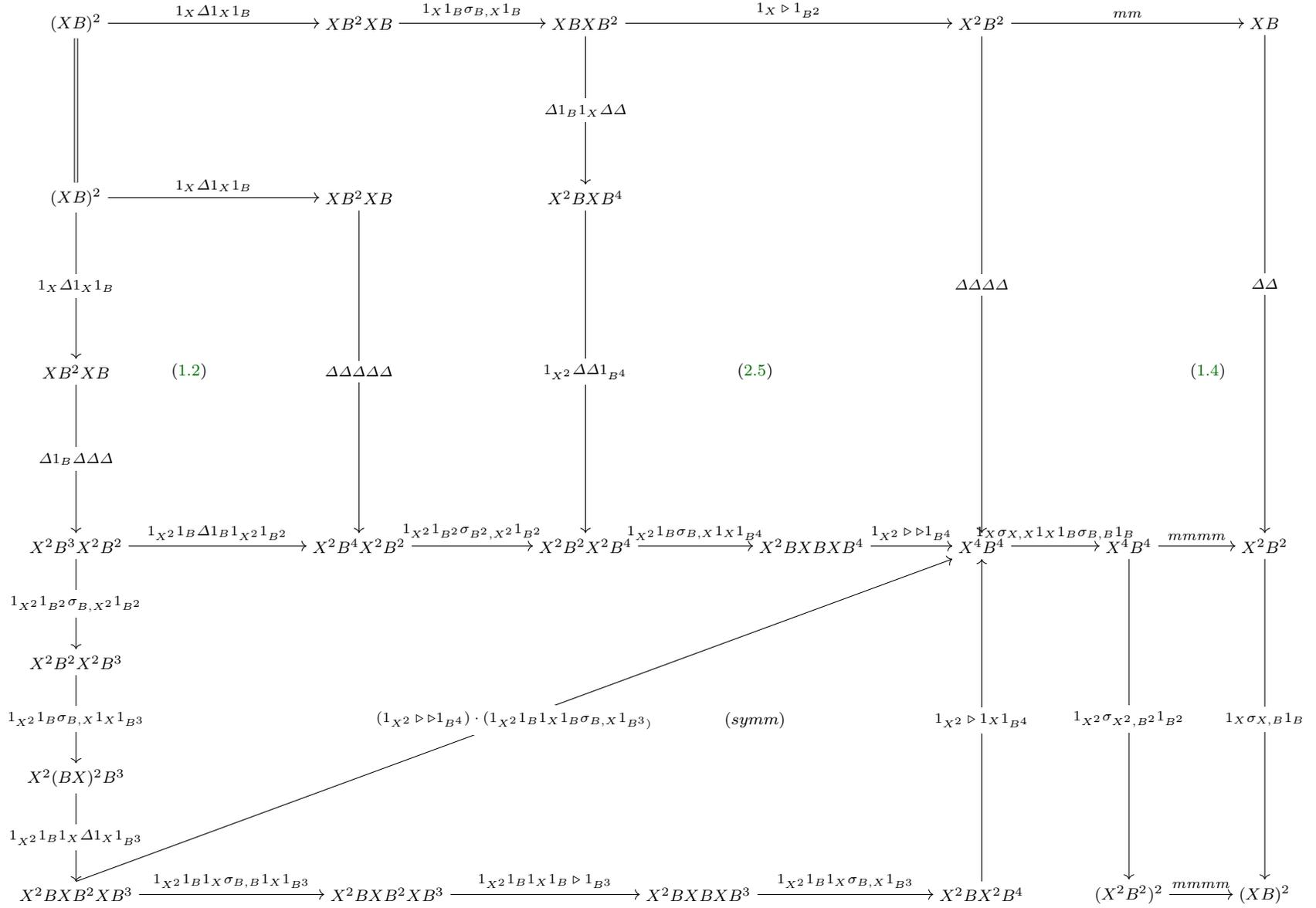
\begin{figure}
\caption{The semi-direct product is a bialgebra: part 1}
    \label{sm bialgebra 1}
% textidote: ignore begin

\begin{tikzpicture}[descr/.style={fill=white},baseline=(A.base),xscale=1.95,yscale=2] 
\node (O9) at (10.5,0) {$  (X  B)^2$};
\node (A8) at (9.3,0) {$  (X^2  B^2)^2$};
\node (C5) at (2.75,0) {$ X^2  B  X  B^2  X  B^3$};
\node (C6) at (5.5,0) {$ X^2  B  X  B  X  B^3$};
\node (C7) at (8,0) {$  X^2  B  X^2  B^4$};
\node (C8) at (9.3,3) {$  X^4 B^4$};
\node (C9) at (10.5,3) {$  X^2 B^2$};
\node (D1) at (0,4.5) {$ X  B^2  X  B$};
\node (D2) at (0,3) {$  X^2  B^3  X^2 B^2$};
\node (D3) at (0,2) {$  X^2  B^2  X^2 B^3$};
%\node (D3') at (0,0) {$  X^2  B^2  X^2  B^3$};
\node (D4) at (0,1) {$  X^2  (B X)^2  B^3$};
\node (D5) at (0,0) {$ X^2B X  B^2  X   B^3$};
\node (D6) at (6.5,3) {$ X^2B  X  B  X   B^4$};
\node (D8) at (8,3) {$  X^4 B^4$};
\node (E1) at (0,6) {$(XB)^2$};
\node (E2) at (2.5,6) {$  XB^2XB$};
\node (E3) at (2.5,3) {$  X^2B^4X^2B^2$};
\node (E6) at (4.5,3) {$ X^2B^2X^2B^4$};
\node (F6) at (4.5,6) {$  X^2BXB^4$};
\node (G1) at (0,7.5) {$(XB)^2$};
\node (G3) at (2.5,7.5) {$  XB^2XB$};
\node (G6) at (4.5,7.5) {$  XBXB^2$};
\node (G8) at (8,7.5) {$  X^2B^2$};
\node (G9) at (10.5,7.5) {$  XB$};
\node at (6,1.5) {$(symm)$};
\node at (6,4.5) {$\eqref{delta et act}$};
\node at (10,4.5) {$\eqref{m et delta}$};
\node at (1,4.5) {$\eqref{coass comultiplication}$};
\draw[commutative diagrams/.cd, ,font=\scriptsize]
(G1.south) edge[commutative diagrams/equal]  (E1.north)
;
\path[->,font=\scriptsize]
(G1.east) edge node[above] {$1_X \Delta 1_X 1_B $} (G3.west)
(G3.east) edge node[above] {$1_X 1_B  {\blue \sigma_{B,X}} 1_B $} (G6.west)
(G6.east) edge node[above] {$1_X \act 1_{B^2}$} (G8.west)
(E1.east) edge node[above] {$1_X \Delta 1_X 1_B $} (E2.west)
(E1.south) edge node[descr] {$1_X \Delta 1_X 1_B $} (D1.north)
(D1.south) edge node[descr] {$\Delta 1_B \Delta \Delta \Delta $} (D2.north)
(G8.east) edge node[above] {$mm$} (G9.west)
(D2.east) edge node[above] {$1_{X^2} 1_B \Delta 1_B 1_{X^2} 1_{B^2} $} (E3.west)
(E2.south) edge node[descr] {$\Delta \Delta \Delta \Delta \Delta $} (E3.north)
(G6.south) edge node[descr] {$\Delta 1_B 1_X \Delta \Delta $} (F6.north)
(F6.south) edge node[descr] {$1_{X^2} \Delta \Delta 1_{B^4} $} (E6.north)
(G8.south) edge node[descr] {$\Delta \Delta \Delta \Delta$} (D8.north)
(E3.east) edge node[above] {$1_{X^2} 1_{B^2} \sigma_{B^2,X^2} 1_{B^2} $} (E6.west)
(D8.east) edge node[above] {$1_{X}   {\blue \sigma_{X,X}} 1_{X}1_{B}   {\blue \sigma_{B,B}} 1_{B} $} (C8.west)
(C8.east) edge node[above] {$mmmm$} (C9.west)
(A8.east) edge node[above] {$mmmm$} (O9.west)
(C6.east) edge node[above] {$1_{X^2} 1_B 1_X  {\blue \sigma_{B,X}} 1_{B^3}$} (C7.west)
(C5.east) edge node[above] {$1_{X^2} 1_B 1_X 1_B \act 1_{B^3}$} (C6.west)
(G9.south) edge node[descr] {$\Delta \Delta $} (C9.north)
(C9.south) edge node[descr] {$1_X  {\blue \sigma_{X,B}} 1_B $} (O9.north)
(C7.north) edge node[descr] {$1_{X^2} \act 1_X 1_{B^4} $} (D8.south)
(C8.south) edge node[descr] {$1_{X^2} \sigma_{X^2,B^2} 1_{B^2} $} (A8.north)
(D2.south) edge node[descr] {$1_{X^2}1_{B^2} \sigma_{B,X^2}   1_{B^2} $} (D3.north)
(D5.north) edge node[descr] {$(1_{X^2} \act \act 1_{B^4})\cdot (1_{X^2}1_{B}1_X 1_B  {\blue \sigma_{B,X}}   1_{B^3)} $} (D8.south west)
(E6.east) edge node[above] {$1_{X^2}1_B  {\blue \sigma_{B,X}} 1_X 1_{B^4} $} (D6.west)
(D5.east) edge node[above] {$1_{X^2}1_B 1_X  {\blue \sigma_{B,B}} 1_X 1_{B^3} $} (C5.west)
(D6.east) edge node[above] {$1_{X^2} \act \act 1_{B^4} $} (D8.west)
(D3.south) edge node[descr] {$1_{X^2} 1_B  {\blue \sigma_{B,X}} 1_X 1_{B^3}$} (D4.north)
(D4.south) edge node[descr] {$1_{X^2} 1_B 1_X \Delta 1_X 1_{B^3}$} (D5.north)
;
\end{tikzpicture} 

\end{figure}

% \end{center}

%\begin{center}\hspace{-1cm}
\begin{figure}
\caption{The semi-direct product is a bialgebra: part 2}
    \label{sm bialgebra 2}
\begin{tikzpicture}[descr/.style={fill=white},baseline=(A.base),
xscale=2.0,yscale=2] 
\node (O1) at (-0.5,0) {$ (XB)^4$};
\node (O4) at (-0.5,-2) {$  (XB)^4$};
\node (O6) at (3,-2) {$  (XB^2XB)^2$};
\node (O7) at (7,-2) {$  (XBXB^2)^2$};
\node (O8) at (8.5,-2) {$  (X^2B^2)^2$};
\node (O9) at (9.5,-2) {$  (XB)^2$};
\node (A6) at (7,0) {$  (XB)^3BXB^2$};
\node (A7) at (7,-1) {$  (XBXB^2)^2$};
\node (A8) at (9.5,-1) {$  (X^2B^2)^2$};
\node (B1) at (-0.5,4) {$ (X^2B^2)^2$};
\node (B2) at (3,4) {$  X^2B^4X^2B^2$};
\node (B3) at (5,4) {$  X^2BXB^3XB^2$};
\node (B4) at (7,3) {$  X^2BX  B^2 X B^3$};
\node (C5) at (9.5,4) {$ X^2  B X B^2 X B^3$};
\node (C6) at (9.5,3) {$ X^2 B X B X B^3$};
\node (C7) at (9.5,2) {$  X^2 B X^2 B^4$};
\node (C8) at (9.5,0) {$  X^4 B^4$};
\node (D8) at (9.5,1) {$  X^4 B^4$};
\node (D1) at (1,5) {$ X B^2 X B$};
\node (D2) at (3,5) {$  X^2 B^3 X^2 B^2$};
\node (D3) at (5,5) {$  X^2 B^2 X^2 B^3$};
\node (D4) at (7,5) {$  X^2 (B X)^2 B^3$};
\node (D5) at (9.5,5) {$ X^2 B X B^2  X   B^3$};
\node (E1) at (-0.5,5) {$(X  B)^2$};
\node at (8.5,4) {$\eqref{coco}$};
\node at (0.5,4.5) {$\eqref{coco}$};
\node at (8.5,1) {$(symm)$};
\draw[commutative diagrams/.cd, ,font=\scriptsize]
(A7.south) edge[commutative diagrams/equal]  (O7.north)
;
%\draw[->] (E1.west) to[bend right=12]node[descr] {$ \Delta_{X\rtimes B}  \Delta_{X\rtimes B}$} (O1.west);
\draw[->] (O4.south) to[bend right=10]node[descr,scale=0.8] {$ {\blue m_{X\rtimes B}m_{X\rtimes B}}$}(O9.south);
\path[->,font=\scriptsize]
(E1.east) edge node[above] {$1_X  \Delta  1_X  1_B $} (D1.west)
(D1.east) edge node[above] {$\Delta  1_B  \Delta\Delta \Delta $} (D2.west)
(D2.east) edge node[above] {$1_{X^2} 1_{B^2} \sigma_{B,X^2} 1_{B^2}$} (D3.west)
(D3.east) edge node[above] {$1_{X^2} 1_{B}  {\blue \sigma_{B,X}} 1_X 1_{B^3} $} (D4.west)
(D4.east) edge node[above] {$1_{X^2} 1_{B} 1_X \Delta 1_X 1_{B^3} $} (D5.west)
(B1.east) edge node[above] {$1_{X^2} \Delta \Delta 1_{X^2} 1_{B^2}$} (B2.west)
(B2.east) edge node[above] {$1_{X^2} 1_B \sigma_{B^3,X} 1_{X} 1_{B^2} $} (B3.west)
(A7.east) edge node[above] {$(1_{X} \act 1_{B^2})^2 $} (A8.west)
(O4.east) edge node[above] {$(1_X  \Delta  1_X  1_B )^2 $} (O6.west)
(O6.east) edge node[above] {$ (1_X  1_B  {\blue \sigma_{B,X}} 1_B )^2 $} (O7.west)
(O7.east) edge node[above] {$ (1_{X} \act 1_{B^2})^2 $} (O8.west)
(O8.east) edge node[above] {$ (mm)^2 $} (O9.west)
(A8.south) edge node[descr] {$ (mm)^2 $} (O9.north)
(B1.south) edge node[descr] {$(1_X  {\blue \sigma_{X,B}} 1_B)^2 $} (O1.north)
(D8.south) edge node[descr,xshift=-10pt] {$1_X  {\blue \sigma_{X,X}} 1_X 1_B  {\blue \sigma_{B,B}} 1_B $} (C8.north)
(C8.south) edge node[descr] {$1_{X^2} \sigma_{X^2,B^2} 1_{B^2} $} (A8.north)
(C7.south) edge node[descr] {$1_{X^2} \act 1_X 1_{B^4} $} (D8.north)
(C6.south) edge node[descr] {$1_{X^2} 1_B 1_X  {\blue \sigma_{B,X}} 1_{B^3} $} (C7.north)
(C5.south) edge node[descr] {$1_{X^2} 1_B 1_X 1_B \act 1_{B^3} $} (C6.north)
(D5.south) edge node[descr,xshift=-5pt] {$1_{X^2} 1_B 1_X  {\blue \sigma_{B,B}} 1_X 1_{B^3} $} (C5.north)
(D4.south) edge node[descr] {$1_{X^2} 1_B 1_X \Delta 1_X 1_{B^3} $} (B4.north)
(D2.south) edge node[descr] {$1_{X^2} 1_B \Delta 1_B 1_{X^2} 1_{B^2} $} (B2.north)
(E1.south) edge node[descr] {$\Delta \Delta \Delta \Delta$} (B1.north)
(O1.south) edge node[descr] {$1_X 1_B \sigma_{XB,XB} 1_X 1_B$} (O4.north)
(B4.south) edge node[descr] {$1_X \sigma_{X,BXB} \sigma_{BXB,B} 1_B$} (A6.north)
(A6.south) edge node[descr] {$(1_{XB})^2  {\blue \sigma_{X,B}} 1_B 1_X 1_{B^2}$} (A7.north)
(B3.south) edge node[descr] {$1_{X^2} 1_B 1_X 1_{B^2}  {\blue \sigma_{B,X}} 1_{B^2}$} (B4.north west)
(B4.east) edge node[above] {$1_{X^2} 1_B 1_X 1_B \act 1_{B^3} $} (C6.west)
;
\end{tikzpicture} 
% textidote: ignore end
\end{figure}
% \end{center}
%\subsection{The semi-direct product is associative}
%In the following diagram we also denote the monoidal product by juxtaposition. We prove that for associative bialgebras if the conditions \eqref{ass action_1} and \eqref{ass action_2} are satisfied, then the semi-direct product is associative 
\begin{figure}
\caption{The semi-direct product is associative}
    \label{sm ass}
% textidote: ignore begin
\begin{tikzpicture}[descr/.style={fill=white},baseline=(A.base),
xscale=1.45,yscale=1.8] 
\node (A0) at (0,0) {$ X   B   X   B^2   X   B$};
%\node (A1) at (2,0) {$ A$};
%\node (A2) at (4,0) {$ A^4$};
%\node (A3) at (6,0) {$X^2$};
\node (A4) at (8,0) {$X^2   B^2   X   B$};
%\node (A5) at (10,0) {$A$};
%\node (A6) at (12,0) {$A^2$};
\node (A7) at (14,0) {$(X   B)^2$};
%\node (B0) at (0,1) {$  A$};
%\node (B1) at (2,1) {$A   B$};
%\node (B2) at (4,1) {$ A ^4$};
%\node (B3) at (6,1) {$A^5$};
\node (B4) at (8,1) {$X   B^4   X   B$};
\node (B5) at (10,1) {$X   B^4   X   B$};
\node (B7) at (14,1) {$X   B^2   X   B$};
\node (C0) at (0,2) {$X   B^2   (X   B)^2 $};
\node (C5) at (10,2) {$X   B^2   X   B^3$};
\node (C7) at (14,2) {$X   B   X   B^2 $};
\node (D1) at (2,3) {$X   B^2   X   B^2 X B$};
\node (D2) at (4,3) {$ X   B^3   X   B  X   B^2 $};
\node (D3) at (6,2) {$ X   B^2 X B   X   B^3 $};
\node (D5) at (10,3) {$X   B   X   B^2$};
\node (D7) at (14,3) {$X^2   B^2 $};
\node (E0) at (0,4) {$ (X   B)^3$};
%\node (E1) at (2,4) {$A^2   B$};
\node (E2) at (4,4) {$ X   B^3   X   B^2   X   B$};
%\node (E3) at (6,4) {$A   B^2   A$};
\node (E3) at (6,4) {$X   B   X   B^2   X   B^3$};
\node (E5) at (10,4) {$X^2   B   X   B^3$};
\node (E6) at (12,4) {$X^3   B^3 $};
\node (E7) at (14,4) {$X   B $};
\node (F1) at (2,5) {$X   B^2   X   B^2 X B$};
\node (F2) at (4,5) {$X   B   X   B^4   X   B$};
\node (F3) at (6,6) {$X   B   X   B^4   X   B$};
\node (F4) at (8,5) {$X   B^2   X^2   B^2$};
\node (F5) at (10.5,5) {$X   (B   X)^2   B^2 $};
\node (F6) at (12.5,5) {$X^3   B^2$};
\node (F7) at (14,5) {$X^2   B^2$};
\node (G0) at (0,6) {$ X   B   X   B^2   X   B$};
\node (G4) at (8,6) {$X   B   X^2   B^2 $};
\node (G7) at (14,6) {$X   B   X   B^2$};
\node (H4) at (8,7) {$X   B^2   X^2   B$};
\node (H7) at (14,7) {$X   B^2   X   B$};
\node (I0) at (0,8) {$X   B   X   B   X   B^2$};
\node (I4) at (8,8) {$X   B   X^2   B^2$};
\node (I7) at (14,8) {$(X   B)^2$};
\node at (13,4) {$(ass)$};
\node at (2,4) {$\eqref{coass comultiplication} $};
\node at (12,2.5) {$\eqref{ass action_1}$};
\node at (11,5.5) {$\eqref{ass action_2}$};
\node at (11,0.5) {$\eqref{m et delta}$};
\path[->,font=\scriptsize]
(A0.east) edge node[above] {$1_{X}  \act  1_{B^2 X B}$} (A4.west)
(A4.east) edge node[above] {$m  m 1_{XB}$} (A7.west)
(B4.east) edge node[above] {$1_{XB}  {\blue \sigma_{B,B}} 1_{BXB}  $} (B5.west)
(B5.east) edge node[above] {$1_{X} m m 1_{XB}  $} (B7.west)
(C5.east) edge node[above] {$1_{X} m 1_X m 1_{B}  $} (C7.west)
(D5.east) edge node[above] {$1_{X} \act  1_{B^2}  $} (D7.west)
(F6.east) edge node[above] {$1_{X} m  1_{B^2}  $} (F7.west)
(F5.east) edge node[above] {$1_{X} \act \act  1_{B^2}  $} (F6.west)
(F4.east) edge node[above] {$1_{XB}  {\blue \sigma_{B,X}}  1_{XB^2}  $} (F5.west)
(G4.east) edge node[above] {$1_{XB} m  1_{B^2}  $} (G7.west)
(I4.east) edge node[above] {$1_{XB} m  m  $} (I7.west)
(I0.east) edge node[above] {$1_{XBX} \act 1_{B^2} $} (I4.west)
(F2.north) edge node[descr] {$1_{XBXB}  {\blue \sigma_{B,B}} 1_{BXB} $} (F3.west)
(D2.south) edge node[descr] {$1_{XB^2} \sigma_{B,XBX} 1_{B^2} $} (D3.west)(E5.east) edge node[above] {$ 1_{X^2} \act 1_{B^3}  $} (E6.west)
(E3.east) edge node[above] {$1_X \act 1_B \act 1_{B^3} $} (E5.west)
(G0.north) edge node[descr] {$1_{XBXB}  {\blue \sigma_{B,X}} 1_B$} (I0.south)
(E0.north) edge node[descr] {$1_{XBX} \Delta 1_{XB}$} (G0.south)
(E0.south) edge node[descr] {$1_X \Delta 1_{XBXB}$} (C0.north)
(C0.south) edge node[descr] {$1_{XB}  {\blue \sigma_{B,X}} 1_{BXB} $} (A0.north)
(E0.south east) edge node[descr] {$1_X \Delta 1_{X} \Delta 1_{XB}$} (D1.north west)
(E0.north east) edge node[descr] {$1_X \Delta 1_{X} \Delta 1_{XB}$} (F1.south west)
(D1.north) edge node[descr] {$1_{XB} \Delta  1_{XB^2XB}$} (E2.south west)
(F1.south) edge node[descr] {$1_X \Delta  1_{BXB^2XB}\; \; \; \; \; \; \;$} (E2.north west)
(E2.north) edge node[descr] {$1_{XB} \sigma_{B^2,X} 1_{B^2XB}  $} (F2.south)
(F3.south) edge node[descr] {$1_{XBXB^2} \sigma_{B^2,X} 1_{B} $} (E3.north)
(E2.south) edge node[descr] {$1_{XB^3XB}  {\blue \sigma_{B,X}} 1_{B} $} (D2.north)
(D3.north) edge node[descr] {$1_{XB}  {\blue \sigma_{B,X}} 1_{BXB^3} $} (E3.south)
(B5.north) edge node[descr] {$1_{XB^2} \sigma_{B^2,X} 1_{B} $} (C5.south)
(C5.north) edge node[descr] {$1_{XB} \act m 1_{B} $} (D5.south)
(A4.north) edge node[descr] {$m \Delta \Delta 1_{XB} $} (B4.south)
(A7.north) edge node[descr] {$1_X \Delta1_{XB} $} (B7.south)
(B7.north) edge node[descr] {$1_{XB}  {\blue \sigma_{B,X}} 1_{B} $} (C7.south)
(C7.north) edge node[descr] {$1_{X} \act 1_{B^2} $} (D7.south)
(D7.north) edge node[descr] {$m m$} (E7.south)
(E6.south east) edge node[descr] {$m 1_X m 1_{B} $} (D7.north west)
(E6.north east) edge node[descr] {$ 1_X m   1_{B} m$} (F7.south west)
(I7.south) edge node[descr] {$1_{X} \Delta 1_{XB} $} (H7.north)
(H7.south) edge node[descr] {$1_{XB}  {\blue \sigma_{B,X}} 1_{B} $} (G7.north)
(G7.south) edge node[descr] {$1_{X} \act 1_{B^2} $} (F7.north)
(F7.south) edge node[descr] {$m m  $} (E7.north)
(I4.south) edge node[descr] {$1_{X} \Delta 1_{X^2} m$} (H4.north)
(H4.south) edge node[descr] {$1_{XB} \sigma_{B,X^2} 1_{B}$} (G4.north)
(G4.south) edge node[descr] {$1_{X} \Delta 1_{X^2B^2}$} (F4.north)
;
\end{tikzpicture} 

% textidote: ignore end
 \end{figure}
\end{landscape}
\begin{landscape}
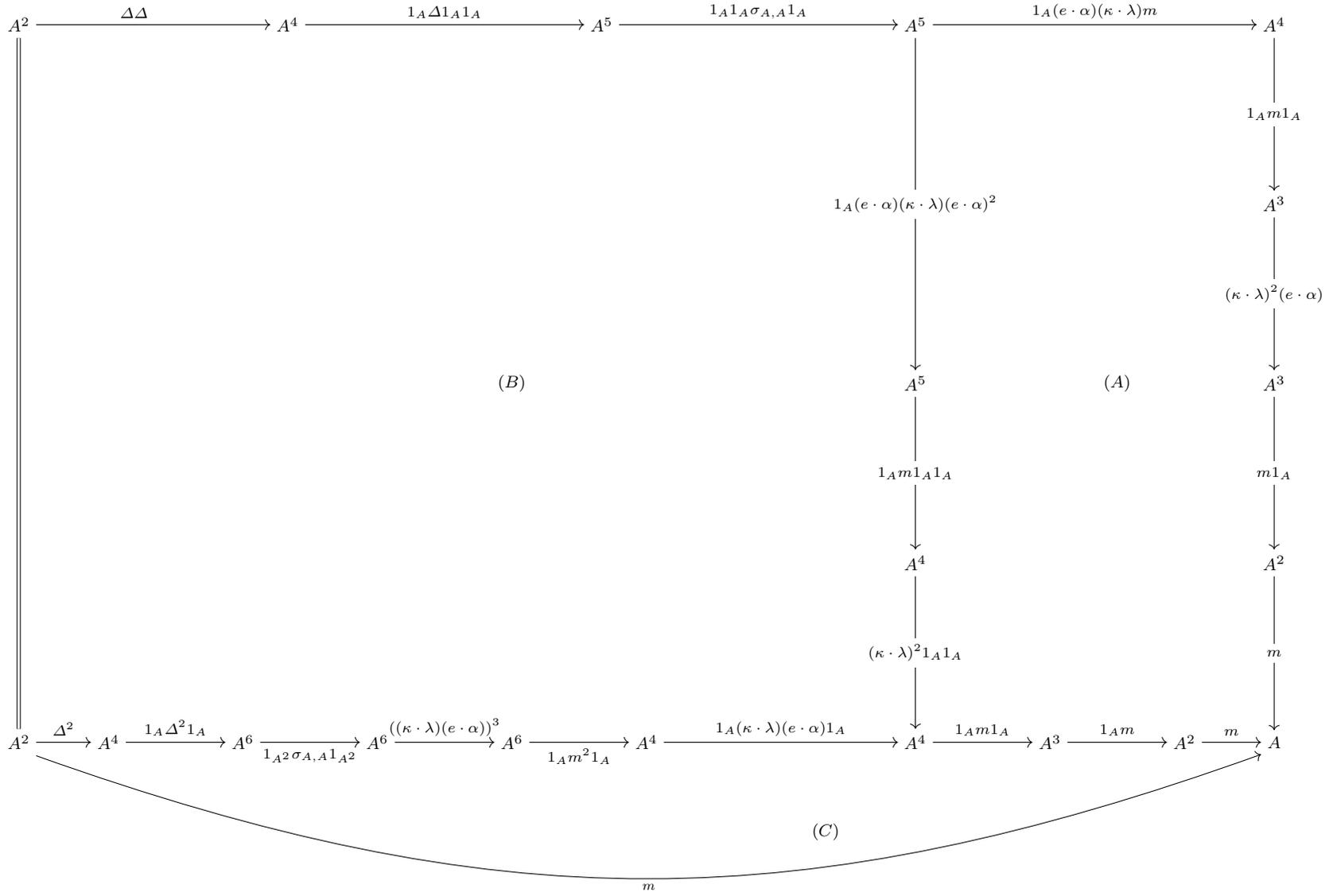
\begin{figure}
\caption{Combination of the three diagrams $(A)$, $(B)$ and $(C)$}
    \label{combinationABC}
% textidote: ignore begin

% textidote: ignore begin
\begin{tikzpicture}[descr/.style={fill=white},xscale=0.75,yscale=1] 
\node (A) at (0,0) {$ A^4$};
\node (B) at (3,0) {$A^3$};
\node (C) at (6,0) {$A^2 $};
\node (D) at (8,0) {$ A$};
\node (E) at (8,3) {$A^2$};
\node (F) at (8,6) {$ A^3$};
\node (G) at (8,9) {$ A^3 $};
\node (H) at (8,12) {$ A^4$};
\node (I) at (0,12) {$A^5$};
\node (K) at (0,6) {$ A^5$};
\node (L) at (0,3) {$A^4$};
\node (A') at (-20,12) {$ A^2$};
\node (A'') at (-20,0) {$ A^2$};
\node (B') at (-14,12) {$A^4$};
\node (C') at (-7,12) {$A^5 $};
\node (J') at (-6,0) {$A^4 $};
\node (K') at (-9,0) {$ A^6$};
%\node (L') at (0,2) {$A^6$};
\node (M') at (-12,0) {$ A^6$};
\node (N') at (-15,0) {$ A^6 $};
\node (O') at (-18,0) {$ A^4$};
\node at (-9,6) {$(B)$};
\node at (4.5,6) {$(A)$};
\node at (-2,-1.5) {$(C)$};
\draw[commutative diagrams/.cd, ,font=\scriptsize]
(A'.south) edge[commutative diagrams/equal]  (A''.north);
\draw[->] (A''.south east) to[bend right=15]node[below,scale=0.8] {${\blue m} $} (D.south west);
\path[->,font=\scriptsize]
(A.east) edge node[above] {$ 1_A   m   1_A$} (B.west)
(B.east) edge node[above] {$ 1_A   m $} (C.west)
(C.east) edge node[above] {$  m $} (D.west)
(I.east) edge node[above] {$ 1_A   (e \cdot \alpha)   (\kappa \cdot \lambda)   m $} (H.west)
(H.south) edge node[descr] {$1_A   m   1_A$} (G.north)
(G.south) edge node[descr] {$ (\kappa \cdot \lambda)^2   (e \cdot \alpha)$} (F.north)
(F.south) edge node[descr] {$ m   1_A$} (E.north)
(E.south) edge node[descr] {$ m$} (D.north)
(L.south) edge node[descr] {$ (\kappa \cdot \lambda)^2   1_A   1_A$} (A.north)
(K.south) edge node[descr] {$1_A   m   1_A    1_A$} (L.north)
(I.south) edge node[descr] {$ 1_A   (e \cdot \alpha)   (\kappa \cdot \lambda)    (e \cdot \alpha)^2$} (K.north)
(A'.east) edge node[above] {$ \Delta   \Delta  \; \; \; \; \; \; \; \;$} (B'.west)
(B'.east) edge node[above] {$ 1_A   \Delta   1_A   1_A $} (C'.west)
(C'.east) edge node[above] {$ 1_A   1_A    {\blue \sigma_{A,A}}   1_A $} (I.west)
(K'.east) edge node[below] {$ 1_A   m^2   1_A $} (J'.west)
(J'.east) edge node[above] {$ 1_A  (\kappa \cdot \lambda)    (e \cdot \alpha)   1_A $} (A.west)
(M'.east) edge node[above] {$  ((\kappa \cdot \lambda)    (e \cdot \alpha))^3  $} (K'.west)
(A''.east) edge node[above] {$ \Delta^2$} (O'.west)
(O'.east) edge node[above] {$  1_A   \Delta^2   1_A$} (N'.west)
(N'.east) edge node[below] {$1_{A^2}   {\blue \sigma_{A,A}}   1_{A^2}$} (M'.west)
;
\end{tikzpicture} % textidote: ignore end

 \end{figure}
\end{landscape}

\end{document}